\documentclass[12pt]{article}

\usepackage{authblk}

\title{Smoothing effect and asymptotic behavior of solutions to 
nonlinear elastic wave equations with viscoelastic terms in the framework of $L^{p}$-Sobolev spaces}
\author[1]{Yoshiyuki Kagei}
\author[2]{Hiroshi Takeda}
\affil[1]{Department of Mathematics, 
Tokyo Institute of Technology, 
Meguro-ku, Tokyo, 152-8551, Japan}
\affil[2]{Department of Intelligent Mechanical Engineering, 
Faculty of Engineering, Fukuoka Institute of Technology, 
3-30-1 Wajiro-higashi, Higashi-ku, Fukuoka, 811-0295, Japan}
\date{}

\usepackage{amssymb,amsmath,amsthm}
\usepackage[dvips]{graphicx}
\usepackage{bm} 

\newcommand{\R}{\mathbb R}

\newcommand{\supp}{\mathop{\mathrm{supp}}\nolimits}

\newtheorem{thm}{Theorem}[section]
\newtheorem{cor}[thm]{Corollary}
\newtheorem{prop}[thm]{Proposition}
\newtheorem{lem}[thm]{Lemma}
\theoremstyle{remark}
\newtheorem{rem}[thm]{Remark}

\theoremstyle{definition}

\begin{document}
\maketitle

\numberwithin{equation}{section}

\begin{abstract}
The Cauchy problem for nonlinear elastic wave equations with viscoelastic damping terms is investigated in $L^{p}$ framework.
It is proved that the small global solutions constructed in $L^{2}$-Sobolev spaces in our preceding paper \cite{K-T}
satisfies consistency property corresponding to the additional regularity of the initial data.
As a result, sharp estimates in $t$ and approximation formulas by the diffusion waves are established. 
\end{abstract}


\newpage
\section{Introduction}

This paper is a sequel to our recent work \cite{K-T} and concerned with the Cauchy problem for the system of quasi-linear hyperbolic equations with strong damping
\begin{equation} \label{eq:1.1}
	\left\{
	\begin{split}
		& \partial_{t}^{2} u -\mu \Delta u - (\lambda + \mu) \nabla {\rm div} u -\nu \Delta \partial_{t} u  =F(u), \quad t>0, \quad x \in \R^{3}, \\
		& u(0,x)=f_{0}(x), \quad \partial_{t} u(0,x)=f_{1}(x) , \quad x \in \R^{3}, 
	\end{split}
	\right.
\end{equation}
where, throughout the present paper, we follow the notation used in \cite{K-T}.
Namely, $u=u(t,x):(0,\infty) \times \R^{3} \to \R^{3}$ is the unknown function; and 
$f_{j}={}^t\! (f_{j1}, f_{j2}, f_{j3})$ $(j=0,1)$ are initial data, and the superscript ${}^t\! \cdot$ means the transpose of the matrix. 
The Lam\'e constants are supposed to satisfy  
\begin{equation} \label{eq:1.2}
\mu>0, \quad \lambda + 2 \mu >0, 
\end{equation}
and the viscosity parameter $\nu$ is a positive constant. 
We assume that the nonlinear term $F(u)$ is given by $\nabla u \nabla D u$,
where $\nabla$ is the spatial gradient and $D$ stands for the $t,x$ gradient.
The equation is a simplified, dimensonless model of viscoelasticity (for the detail see \cite{J-S}, \cite{Ponce} and the references therein).

%
We are interested in the consistency and smoothing effect of global solutions in $L^{p}$ Sobolev spaces.  
We also establish sharp decay estimates in $t$ and approximation formulas by the diffusion waves as $t \to \infty$ in the same framework. 

We briefly review related results for \eqref{eq:1.1}.
Ponce \cite{Ponce} proved the existence of global solutions of \eqref{eq:1.1} with $\lambda+\mu=0$ in $L^{2}$-Sobolev spaces and obtained decay estimates.
The method of proof in \cite{Ponce} is based on the energy method for the higher order derivatives of the solution and the $L^p$-$L^q$ type estimates for the
fundamental solution of the linearized equation. 
In our preceding paper \cite{K-T}, 
we considered the Cauchy problem \eqref{eq:1.1} with \eqref{eq:1.2} in the framework of $L^{2}$-Sobolev spaces; 
and we established quantitative estimates for smoothing and decay properties which reflect both the damping aspect (i.e., the parabolic aspect) and the dispersive aspect (i.e., the hyperbolic aspect) of the system \eqref{eq:1.1}.
The proof in \cite{K-T} is based on the detailed analysis of the fundamental solution of the linearized equation
that takes the spreading effect of the diffusion waves into account as in \cite{H-Z1, H-Z2, K-S, Shibata}.

Especially, we proved the asymptotic profile of the solution as $t \to \infty$ is given by the 
convolution of the Riesz transform and the diffusion waves $G_{j}^{(\beta)}(t)$ $(j=0,1)$ depending on the parameter $\beta>0$ defined by 
\begin{equation*} 
	\begin{split}
		\mathcal{G}^{(\beta)}_{0}(t,\xi) := e^{-\frac{\nu |\xi|^{2}}{2}t} \cos (\beta |\xi| t), \quad
		\mathcal{G}^{(\beta)}_{1}(t,\xi) := e^{-\frac{\nu |\xi|^{2}}{2}t} \frac{\sin (\beta |\xi| t)}{ \beta |\xi| }
	\end{split}
\end{equation*}
and 
\begin{equation*} 
	\begin{split}
		G_{j}^{(\beta)}(t,x)
		:=\mathcal{F}^{-1}[
		\mathcal{G}^{(\beta)}_{j}(t,\xi)
		],
	\end{split}
\end{equation*}
where $\mathcal{F}^{-1}$ represents the Fourier inverse transform.
The results of \cite{K-T} improve the decay estimates obtained in \cite{Ponce}.  
For the detail, see Propositions \ref{prop:2.1}-\ref{prop:2.4} below. 
We also mention the results in \cite{J-S}.
For the Cauchy problem \eqref{eq:1.1} with $\lambda+\mu=0$, 
Jonov-Sideris \cite{J-S} dealt with the nonlinear term given by the Klainerman null condition (cf. \cite{S} and \cite{A}) with small perturbation, and   
obtained the relationship between the lifespan, the coefficients $\mu$, $\nu$ and the deviation of the nonlinearity from being null in the weighted $L^{2}$-Sobolev spaces.
As a result, they also proved existence of the global solutions for the initial data which is not necessarily small in the sense of norms.

On the other hand, Hoff and Zumbrun \cite{H-Z1} considered the compressible Navier-Stokes system whose linearized system at the motionless state has a fundamental solution similar to that for \eqref{eq:1.1}. 
In \cite{H-Z1} the fundamental solution for the lineraized compressible Navier-Stokes system was investigated in detail; 
and, 
in particular, 
it was shown that a difference between the diffusion waves $G^{(\beta)}_j (t, x)$ and the heat kernel (i.e., $G^{(\beta)}_j (t, x)$ with $\beta = 0$) appears in quantitative estimates in $L^p$-norms for $p \neq 2$. 
Based on the linearized analysis, 
they derived the asymptotic approximation formula in $L^{p}$ for $1\le p \le \infty$ as $t \to \infty$
which is given by a sum of the heat kernel and diffusion waves. 
After that they also proved optimal decay estimates of the diffusion waves in $L^{p}$ for $1\le p \le \infty$ in \cite{H-Z2}, 
and then Kobayashi-Shibata \cite{K-S} improved decay estimates of the fundamental solutions to linearized compressible Navier-Stokes system, which suggest the consistency in $L^{p}$ Sobolev spaces.
See also \cite{Shibata, I1} for the analysis of the diffusion waves and \cite{I-K-M, C-I} for the consistency of the lineraized compressible Navier-Stokes system.

Motivated by these works, 
in this paper,
we investigate the consistency, smoothing effect and asymptotic behavior of global solutions to \eqref{eq:1.1} in the framework of $L^{p}$-Sobolev spaces. 

Now we explain our main results of this paper.
We will firstly show the consistency of \eqref{eq:1.1} for the global solution constructed in \cite{K-T} under the additional condition $(f_{0},f_{1}) \in \dot{W}^{3,p} \times \dot{W}^{1,p}$, where $1 \le p \le \infty$. 
Based on the estimates derived in the proof of the consistency, we will secondly prove the smoothing effect of the global solutions. 
Finally, we will show asymptotic profiles of global solutions as $t \to \infty$ in the functions spaces which global solutions belongs to. 
Here we note that the situation changes corresponding to the cases: $1<p<\infty$, $p=1$ and $p=\infty$. 
More precisely, when $1<p<\infty$, we conclude the following consistency property 
$$
	u \in \{ C([0,\infty); \dot{W}^{3,p} \cap \dot{W}^{1,p}) \cap C^{1}([0,\infty); W^{1,p}) )  \}^{3}
$$
with 
\begin{align}
		\| \nabla^{\alpha} u(t) \|_{p} & 
		\le C (1+t)^{-\frac{3}{2}(1-\frac{1}{p})+\frac{1}{p}-\frac{\alpha}{2}}, \quad 1 \le \alpha \le 3, \label{eq:1.4} \\
			\| \nabla^{\alpha}  \partial_{t} u(t) \|_{p} & 
			\le C (1+t)^{-\frac{3}{2}(1-\frac{1}{p})+\frac{1}{p}-\frac{\alpha+1}{2}}, \quad 0 \le \alpha \le 1, \label{eq:1.5} 
\end{align}
and the smoothing effect of the global solution 
$$
	u \in \{ C^{1}((0,\infty); \dot{W}^{2,p}) \cap C^{2}((0,\infty); L^{p}) \}^{3}
$$
with 
	\begin{align} 
		\| \nabla^{2} \partial_{t} u(t) \|_{p} 
		& \le C (1+t)^{-\frac{3}{2}(1-\frac{1}{p})+\frac{1}{p}-1} t^{-\frac{1}{2}}, \quad \label{eq:1.6} \\
		\| \partial_{t}^{2} u(t) \|_{p} 
		& \le C (1+t)^{-\frac{3}{2}(1-\frac{1}{p})+\frac{1}{p}-\frac{1}{2}} t^{-\frac{1}{2}}. \label{eq:1.7}
	\end{align}	

Our proof for $1<p<\infty$ relies on the direct estimation of the fundamental solutions and the Gronwall type argument.
Especially, more detailed analysis of the fundamental solutions are required to show 
$u \in \{C([0,\infty); \dot{W}^{3,p} ) \cap C^{2}((0,\infty); L^{p}) \}^{3}$ with \eqref{eq:1.4} and \eqref{eq:1.7} for $F(u)=\nabla u \nabla^{2} u$ and  
$u \in \{C^{1}((0,\infty); \dot{W}^{2,p} ) \cap C^{2}((0,\infty); L^{p}) \}^{3}$ with \eqref{eq:1.6} and \eqref{eq:1.7} for $F(u)=\nabla u \nabla \partial_{t} u$.
To overcome the difficulty,
following the method in \cite{Shibata, K-S},
we will show the estimates for the high frequency parts of the fundamental solutions to \eqref{eq:1.1} based on 
the Fourier multiplier theory and the oscillatory integral.
More precisely, we will show $L^{p}$-$L^{p}$ type estimates by the combination of the parabolic smoothing  (i.e. parabolic aspect) 
and use of the cancellation in the integration by parts by the oscillation integral (i.e. hyperbolic aspect). 
See Corollary \ref{cor:4.11} later.

On the other hand, when $p=1, \infty$, we cannot apply the same argument to prove the smoothing effect of the global solution, 
because of the well-known fact, the absence of $L^{p}$-$L^{p}$ boundedness of the Riesz transform, which appears in the fundamental solutions to \eqref{eq:1.1}.
As a result, we will conclude the smoothing effect of the global solutions as 
$$
u \in \{ C([0,\infty); \dot{W}^{1,1} ) \cap C^{1}([0,\infty); W^{1,1}) \cap C^{2}((0,\infty); L^{1}) \}^{3}
$$
satisfying \eqref{eq:1.4} with $\alpha=1$, \eqref{eq:1.5} and \eqref{eq:1.7} for $p=1$, and 
$$
	u \in \{ \dot{W}^{2,\infty}(0,\infty; L^{\infty}) \}^{3}
$$
satisfying \eqref{eq:1.7} for $p=\infty$.
In the case $p=1$, we establish estimates by using the difference of propagation speeds of the wave part of the fundamental solutions to overcome the regularity loss, 
which plays the role of a substitute of the direct estimation of the fundamental solutions.
Such kind of estimates are found in \cite{K-S} for the linearized compressible Navier-Stokes equations, where the situation is slightly different from ours.
On the other hand, when $p=\infty$, 
we take the following two steps, which is the crucial point to obtain the smoothing estimate \eqref{eq:1.7} for $p=\infty$.
Using the regularity of the initial data and interpolation theory, 
we firstly apply the smoothing effect of the global solutions for the case $1<p<\infty$,
to have decay properties and regularity of the global solution in
$$
	u \in \{ C([0,\infty); \dot{W}^{3,q} \cap \dot{W}^{1,q}) \cap C^{1}([0,\infty); W^{1,q}) \cap C^{1}((0,\infty); \dot{W}^{2,q}) \cap C^{2}((0,\infty); L^{q}) \}^{3}
$$
for  $2 \le q<\infty$.  
As a next step, we estimate the solution to show that  
$u$ belongs to $\{ \dot{W}^{2,\infty}(0,\infty; L^{\infty}) \}^{3}$ with desirable decay properties.
%
%

This paper is organized as follows.
Section 2 is devoted to preliminaries, which includes explanation of notation, detailed description of the results of \cite{K-T} for $L^{2}$-Sobolev spaces and useful facts used later. 
We state the main results of this paper precisely in section 3.
In section 4, we summarize the estimates of the fundamental solutions to the Cauchy problem \eqref{eq:1.1}.
Sections 5-7 are devoted to the study of the consistency and smoothing effect of the global solutions to \eqref{eq:1.1}.  
We deal with the case $1<p<\infty$ in section 5, $p=1$ in section 6 and $p=\infty$ in section 7, respectively.  
\section{Preliminaries}
In this section, we set up the notation, following \cite{K-T}.
After that, we summarize the estimates for the fundamental solutions of the strongly damped wave equations and
wave equations.
We also review some of the standard facts on the Riesz transform and the interpolation theory.
\subsection{Notation}
For simplicity, we denote $\mathcal{I}_{3} \in M(\R;3)$ is the identity matrix and 
\begin{equation} \label{eq:2.1}
	\mathcal{P}:=\displaystyle\frac{\xi}{|\xi|} \otimes \frac{\xi}{|\xi|}.
\end{equation}
We also define the 3-d valued constant vector depending on the initial data and the nonlinear term as follows:
\begin{equation*}
	\begin{split}
		m_{j}= {}^t\! (m_{j1}, m_{j2}, m_{j3}), \quad M[u]:={}^t\! (M_{1}[u], M_{2}[u], M_{3}[u]), 
	\end{split}
\end{equation*}
where 
\begin{equation*}
	\begin{split}
		m_{0k}:= \displaystyle\int_{\R^{3}} \nabla f_{0k}(x) dx, \quad m_{1k}:= \displaystyle\int_{\R^{3}} f_{1k}(x) dx
	\end{split}
\end{equation*}
and 
\begin{equation*}
	\begin{split}
		M_{k}[u] := \displaystyle\int_{0}^{\infty} \int_{\R^{3}} F_{k}(u)(\tau, y) dy\, d \tau 
	\end{split}
\end{equation*}
for $k=1,2,3$.
Using the above notation, we define the functions $G$, $H$ and $\tilde{G}$ by
\begin{equation*} 
	\begin{split}
		G(t,x):= 
		& \nabla^{-1} \mathcal{F}^{-1} \left[ 
		\left( \mathcal{G}^{(\sqrt{\lambda+2 \mu})}_{0}(t,\xi) -\mathcal{G}^{(\sqrt{\mu})}_{0}(t,\xi) \right)\mathcal{P} 
		+
		\mathcal{G}^{(\sqrt{\mu})}_{0}(t,\xi) 
		\right] m_{0} \\
		& +
		\mathcal{F}^{-1} \left[ 
		\left( \mathcal{G}^{(\sqrt{\lambda+2 \mu})}_{1}(t,\xi) -\mathcal{G}^{(\sqrt{\mu})}_{1}(t,\xi) \right)\mathcal{P} 
		+
		\mathcal{G}^{(\sqrt{\mu})}_{1}(t,\xi) 
		\right] (m_{1}+M[u]),
	\end{split}
\end{equation*}
\begin{equation*}
	\begin{split}
		& H(t,x):= \\ 
		& \nabla^{-1} \mathcal{F}^{-1} \left[ 
		\left( (\lambda+2 \mu)  
		\mathcal{G}^{(\sqrt{\lambda+2 \mu})}_{1}(t,\xi) 
		-\mu \mathcal{G}^{(\sqrt{\mu} )}_{1}(t,\xi) 
		\right)\mathcal{P} 
		+ \mu \mathcal{G}^{(\sqrt{\mu} )}_{1}(t,\xi)
		\right] m_{0} \\
		& +
		\mathcal{F}^{-1} \left[ 
		\left(  \mathcal{G}^{(\sqrt{\lambda+2 \mu})}_{0}(t,\xi) -\mathcal{G}^{(\sqrt{\mu} )}_{0}(t,\xi) \right) \mathcal{P} 
		+
		\mathcal{G}^{(\sqrt{\mu} )}_{0}(t,\xi)  
		\right]  (m_{1}+M[u]),
	\end{split}
\end{equation*}
and
\begin{equation*} 
	\begin{split}
		& \tilde{G}(t,x) := \\
		& -\Delta \nabla^{-1}
		\mathcal{F}^{-1} \left[
		\left(
		(\lambda+2 \mu) 
		\mathcal{G}^{(\sqrt{\lambda+2 \mu})}_{0}(t,\xi) 
		- \mu  \mathcal{G}^{(\sqrt{\mu} )}_{0}(t,\xi) 
		\right)\mathcal{P} 
		+ \mu  
		\mathcal{G}^{(\sqrt{\mu} )}_{0}(t,\xi)
		\right]  m_{0} \\
		& -\Delta
		\mathcal{F}^{-1} \left[
		\left(
		(\lambda+2 \mu) 
		\mathcal{G}^{(\sqrt{\lambda+2 \mu})}_{1}(t,\xi) 
		- \mu  \mathcal{G}^{(\sqrt{\mu} )}_{1}(t,\xi) 
		\right)\mathcal{P} 
		+ \mu  
		\mathcal{G}^{(\sqrt{\mu} )}_{1}(t,\xi) 
		\right] (m_{1}+M[u]),
	\end{split}
\end{equation*}
respectively.

For function spaces, $L^{p}=L^{p}(\R^{3})$ is a usual Lebesgue space equipped with the norm $\| f \|_{p}$ for $1 \le p \le \infty$. 
The symbol $W^{k,p}(\R^{3})$ stands for the usual Sobolev spaces
\begin{equation*}
W^{k,p}(\R^{3})
  :=\Big\{ f:\R^{3} \to \R;
        \| f \|_{W^{k,p}(\R^{3})} 
        := \| f \|_{p}+
         \|  \nabla_{x}^{k} f \|_{p}< \infty 
     \Big\}.
\end{equation*}
When $p=2$, we denote $W^{k,2}(\R^{3}) = H^{k}(\R^{3})$.
For the notation of the function spaces, the domain $\R^{3}$ is often abbreviated.
We write by $\dot{W}^{k,p}$ and $\dot{H}^{k}$
the corresponding homogeneous Sobolev spaces, respectively.  

Let us denote by $\hat{f}$ the Fourier transform of $f$
defined by
\begin{align*}
\hat{f}(\xi) := (2 \pi)^{-\frac{3}{2}}
\int_{\R^{3}} e^{-i x \cdot \xi} f(x) dx.
\end{align*}
Also, let us denote by $\mathcal{F}^{-1}[f]$ or $\check{f}$ the inverse Fourier transform.
\subsection{Consistency and smoothing effect of the global solutions in $L^{2}$ Sobolev spaces}
We recall the basic facts on the global solutions to \eqref{eq:1.1}, which are shown in \cite{K-T}.
We begin with precise description of the global existence, decay properties and smoothing effect of the $\dot{H}^{3}$ solutions when $F(u) =\nabla u \nabla^{2} u$.
\begin{prop}[\cite{K-T}] \label{prop:2.1}
	Suppose that $F(u) =\nabla u \nabla^{2} u$.
	Let $(f_{0}, f_{1}) \in \{ \dot{H}^{3} \cap \dot{W}^{1,1} \}^{3} \times  \{H^{1} \cap {L}^{1} \}^{3}$ with sufficiently small norms. 
	Then there exists a unique global solution to \eqref{eq:1.1} in the class
	$$
	 \{ C([0,\infty); \dot{H}^{3} \cap \dot{H}^{1} ) \cap C^{1}([0,\infty); H^{1}) \}^{3}
	$$
	satisfying the following time decay properties:
	\begin{equation} \label{eq:2.2}
		\begin{split}
			\| \nabla^{\alpha} u(t) \|_{2} & 
			\le C (1+t)^{-\frac{1}{4}-\frac{\alpha}{2}}, \quad 1 \le \alpha \le 3, \\
			\| \partial_{t}  \nabla^{\alpha} u(t) \|_{2} & 
			\le C (1+t)^{-\frac{3}{4}-\frac{\alpha}{2}}, \quad 0 \le \alpha \le 1,
		\end{split}
	\end{equation}
	for $t \ge 0$. Moreover the solution $u(t)$ satisifies
	\begin{equation*}
		\begin{split}
			u \in \{ C^{1}( (0,\infty); \bigcup_{2 \le p < 6} \dot{W}^{2,p}) 
			\cap W^{1, \infty}( 0,\infty; W^{1,\infty}) \cap C^{2} ((0,\infty); \bigcup_{2 \le p < 6} L^{p} ) \}^{3}
		\end{split}
	\end{equation*}
	and 
	\begin{align}
		\| \nabla^{\alpha} u(t) \|_{\infty} & 
		\le C (1+t)^{-\frac{3}{2}-\frac{\alpha}{2}}, \quad 0 \le \alpha \le 1 \label{eq:2.3}
	\end{align}
	for $t \ge 0$ and 
	\begin{align}
		\| \nabla^{2} \partial_{t} u(t) \|_{p} & \le C(1+t)^{-\frac{7}{4}+\frac{1}{p}} t^{-\frac{5}{4}+\frac{3}{2p} }, \quad 2 \le p <6, \label{eq:2.4} \\
		\| \nabla^{\alpha} \partial_{t}  u(t) \|_{\infty} & 
		\le C(1+t)^{-\frac{7}{4}} t^{-\frac{1}{4}-\frac{\alpha}{2}},\quad 0 \le \alpha \le 1, \label{eq:2.5} \\
		\| \partial_{t}^{2}  u(t) \|_{p} & 
	\le C (1+t)^{-\frac{5}{4}+\frac{1}{p}} t^{-\frac{5}{4}+\frac{3}{2p}}, \quad 2 \le p <6   \label{eq:2.6} 
	\end{align}
	for $t > 0$. 
\end{prop}
We next state the result for the case $F(u) =\nabla u \nabla \partial_{t} u$.
\begin{prop}[\cite{K-T}] \label{prop:2.2}
	Suppose that $F(u) =\nabla u \nabla \partial_{t}  u$.
	Let $(f_{0}, f_{1}) \in \{ \dot{H}^{3} \cap \dot{W}^{1,1} \}^{3} \times  \{ H^{2} \cap {L}^{1} \}^{3}$ with sufficiently small norms. 
	Then there exists a unique global solution to \eqref{eq:1.1} in the class
	$$
	\{ C([0,\infty); \dot{H}^{3} \cap \dot{H}^{1} ) \cap C^{1}([0,\infty); H^{2}) \}^{3}
	$$
	satisfying the estimates \eqref{eq:2.2} and
	\begin{equation} \label{eq:2.7}
		\begin{split}
			\| \partial_{t} \nabla^{\alpha} u(t) \|_{2} 
			\le C (1+t)^{-\frac{3}{4}-\frac{\alpha}{2}}, \quad 0 \le \alpha \le 2
		\end{split}
	\end{equation}
	for $t \ge 0$. Moreover the solution $u(t)$ has the following regularity and decay properties:
		$$u \in \{ C^{1}((0,\infty;\dot{W}^{2,6}) \cap W^{1, \infty}( 0,\infty; W^{1,\infty}) \cap C^{2} ([0,\infty); L^{2} ) \cap C^{2} ((0,\infty); L^{6} ) \}^{3}$$
	and the estimates \eqref{eq:2.3},
	\begin{align}
		\| \partial_{t}  u(t) \|_{\infty} & 
		\le C(1+t)^{-2},    \label{eq:2.8} 
	\end{align}
	for $t \ge 0$ and    
	\begin{align}
		\| \nabla^{2} \partial_{t} u(t) \|_{p} & \le C(1+t)^{-\frac{9}{4}+\frac{1}{p}} t^{-\frac{3}{4}+\frac{3}{2p}}, \quad 2 \le p \le 6, \label{eq:2.9} \\
		\| \nabla \partial_{t}  u(t) \|_{\infty} & 
		\le C(1+t)^{-\frac{9}{4}} t^{-\frac{1}{4}},    \label{eq:2.10} \\
		\| \partial_{t}^{2} u(t) \|_{p} & \le C(1+t)^{-\frac{7}{4}+\frac{1}{p}} t^{-\frac{3}{2}(\frac{1}{2}-\frac{1}{p})}, \quad 2 \le p \le 6 \label{eq:2.11} 
	\end{align}
	for $t \ge 0$ with $p=2$ and $t > 0$ with $p \neq 2$. 
\end{prop}
In Proposition \ref{prop:2.1}, 
the estimates in \eqref{eq:2.2} imply the consistency of \eqref{eq:1.1} with $F(u) =\nabla u \nabla^{2} u$ in $\{ \dot{H}^{3} \cap \dot{H}^{1} \} \times H^{1}$, while the estimates \eqref{eq:2.3}-\eqref{eq:2.6} suggest the smoothing effect of the global solution is this framework. 
A similar interpretation holds for Proposition \ref{prop:2.2}.

In \cite{K-T}, the large time behavior of the $\dot{H}^{3}$ solutions for $F(u) =\nabla u \nabla^{2} u$ is studied and is formulated as follows:
\begin{prop}[\cite{K-T}] \label{prop:2.3}
	The global solution $u(t)$ of \eqref{eq:1.1} constructed in Proposition \ref{prop:2.1} satisfies
	\begin{align}
		& \| \nabla^{\alpha} (u(t)-G(t)) \|_{2} = o(t^{-\frac{1}{4}-\frac{\alpha}{2}}), \quad 1 \le \alpha \le 3, \label{eq:2.12} \\
		& \| \nabla^{\alpha}  (u(t)-G(t)) \|_{\infty} =o(t^{-\frac{3}{2}-\frac{\alpha}{2}}), \quad 0 \le \alpha \le 1, \label{eq:2.13}  \\
		& \| \nabla^{\alpha} (\partial_{t}u(t) -H(t)) \|_{2} =o(t^{-\frac{3}{4}-\frac{\alpha}{2}}), \quad 0 \le \alpha \le 2, \label{eq:2.14}  \\
		& \| \nabla^{2} (\partial_{t} u(t) -H(t)) \|_{p} =o( t^{-3+\frac{5}{2p}} ), \quad 2 \le p <6, \label{eq:2.15}  \\
		& \| \nabla^{\alpha} (\partial_{t} u(t) -H(t))\|_{\infty}
		=o( t^{-2-\frac{\alpha}{2}}),\quad 0 \le \alpha \le 1, \label{eq:2.16}  \\
		& \| \partial_{t}^{2} u(t) -\tilde{G}(t) \|_{2} = o(t^{-\frac{5}{4}}) \label{eq:2.17} 
	\end{align}
	as $t \to \infty$.
\end{prop}
%
The $\dot{H}^{3}$ solutions for $F(u) =\nabla u \nabla \partial_{t} u$ is also considered and 
the correspondence is described as follows:
\begin{prop}[\cite{K-T}] \label{prop:2.4}
	The global solution $u(t)$ of \eqref{eq:1.1} constructed in Proposition \ref{prop:2.2} satisfies the estimates \eqref{eq:2.12}-\eqref{eq:2.17} and  
		\begin{align}
			& \| \nabla^{2} (\partial_{t} u(t) -H(t)) \|_{6} =o( t^{-\frac{31}{12}} ), \label{eq:2.18} \\
			& \| \partial_{t}^{2} u(t) -\tilde{G}(t) \|_{6} =o( t^{-\frac{25}{12}} ), \label{eq:2.19} 
	\end{align}
	as $t \to \infty$.
\end{prop}
Proposition \ref{prop:2.3} (resp. Proposition \ref{prop:2.4}) shows that estimates in Proposition \ref{prop:2.1} (resp. Proposition \ref{prop:2.2}) are sharp in $t$. 

\subsection{Estimates for linear equations}
We begin with the strongly damped wave equations:
\begin{equation} \label{eq:2.20}
\left\{
\begin{split}
& \partial_{t}^{2} w -\beta^{2} \Delta w -\nu \Delta \partial_{t} w  = 0, \quad t>0, \quad x \in \R^{3}, \\
& w(0,x)=w_{0}(x), \quad \partial_{t} w(0,x)=w_{1}(x) , \quad x \in \R^{3}, 
\end{split}
\right.
\end{equation}
where $w=w(t,x):(0,\infty)\times \R^{3}\to \R$ and $\beta>0$.
Let us denote the characteristic roots $\sigma_{\pm}^{(\beta)}$ by
\begin{equation*}
\sigma_{\pm}^{(\beta)}:=\frac{-\nu|\xi|^{2} \pm \sqrt{\nu^{2}|\xi|^{4}-4 \beta^{2} |\xi|^{2}}}{2},
\end{equation*}
Then the solution of \eqref{eq:2.20} is expressed as 
\begin{equation*}
	\begin{split}
		w(t)=K_{0}^{(\beta)}(t) w_{0} +K_{1}^{(\beta)}(t) w_{1},
	\end{split}
\end{equation*}
where 
\begin{equation} \label{eq:2.21}
	\begin{split}
		K_{0}^{(\beta)} (t) w_{0}:=
		\mathcal{F}^{-1} [\mathcal{K}_{0}^{(\beta)}(t,\xi) \hat{w}_{0}], \quad 
		K_{1}^{(\beta)} (t) w_{1}:=
		\mathcal{F}^{-1} [\mathcal{K}_{1}^{(\beta)}(t,\xi) \hat{w}_{1}]
	\end{split}
\end{equation}
and
\begin{equation*}
\begin{split}
\mathcal{K}_{0}^{(\beta)}(t,\xi):= 
\frac{
	-\sigma_{-}^{(\beta)}e^{\sigma_{+}^{(\beta)}t}+\sigma_{+}^{(\beta)} e^{\sigma_{-}^{(\beta)}t}
		}{\sigma_{+}^{(\beta)}-\sigma_{-}^{(\beta)}}, \quad 
	\mathcal{K}_{1}^{(\beta)}(t,\xi):= 
	\frac{
		e^{\sigma_{+}^{(\beta)}t}-e^{\sigma_{-}^{(\beta)}t}
	}{\sigma_{+}^{(\beta)}-\sigma_{-}^{(\beta)}}.
\end{split}
\end{equation*}
We also define the smooth cut-off functions $\chi_{j}= \chi_{j}(\xi) \in C^{\infty}(\R^{3})$ $(j=L,M,H)$ as follows:
\begin{equation*}
\begin{split}
\chi_{L}:= 
\begin{cases}
& 1 \quad (|\xi|\le \frac{c_{0}}{2}), \\
& 0 \quad (|\xi|\ge c_{0}),
\end{cases}
\end{split}
\end{equation*}
\begin{equation*}
\begin{split}
\chi_{H}:= 
\begin{cases}
& 0 \quad (|\xi|\le c_{1}), \\
& 1 \quad (|\xi|\ge 2c_{1})
\end{cases}
\end{split}
\end{equation*}
and 
\begin{equation*}
\chi_{M}=1-\chi_{L}-\chi_{H}.
\end{equation*}
Here $c_{0}$ and $c_{1}$ $(0<c_{0}<c_{1}<\infty)$ are some constants to be determined later.
As in \cite{K-T}, we can then use the evolution operators defined by 
\begin{equation} \label{eq:2.22}
\begin{split}
K_{j}^{(\beta)}(t)g & := \mathcal{F}^{-1}[
\mathcal{K}_{j}^{(\beta)}(t, \xi)\hat{g} 
], \\
K_{jk}^{(\beta)}(t)g & := \mathcal{F}^{-1}[
\mathcal{K}_{j}^{(\beta)}(t,\xi)\chi_{k} \hat{g} 
], \\
G_{jk}^{(\beta)}(t)g & := \mathcal{F}^{-1}[
\mathcal{G}_{j}^{(\beta)}(t,\xi)\chi_{k} \hat{g} 
]
\end{split}
\end{equation}
for $j=0,1$ and $k=L,M,H$.
Firstly we mention the decay properties of the fundamental solutions \eqref{eq:2.21}.
\begin{lem}[\cite{Ponce}, \cite{Shibata}, \cite{K-S}] \label{Lem:2.5}
	{\rm (i)} Let $1 \le q \le p \le \infty$, $\ell \ge \tilde{\ell} \ge 0$ and $\alpha \ge \tilde{\alpha} \ge 0$. 
	Then it holds that
	\begin{align} 
	& 
	\left\|
	\partial_{t}^{\ell} 
	\nabla^{\alpha}
	K_{0L}^{(\beta)}(t)g
	\right\|_{p} 
	\le C(1+t)^{-\frac{3}{2}(\frac{1}{q}-\frac{1}{p})-(\frac{1}{q}-\frac{1}{p})+\frac{1}{2} -\frac{\ell-\tilde{\ell}+ \alpha-\tilde{\alpha}}{2}}
	\| \nabla^{\tilde{\alpha}+\tilde{\ell}} g \|_{q}, \label{eq:2.23}  \\
	& \left\| 
	\partial_{t}^{\ell} 
	\nabla^{\alpha}
	K_{1L}^{(\beta)}(t) g
	\right\|_{p} 
	\le C(1+t)^{-\frac{3}{2}(\frac{1}{q}-\frac{1}{p})-(\frac{1}{q}-\frac{1}{p})+1 -\frac{\ell-\tilde{\ell}+ \alpha-\tilde{\alpha}}{2}}
	\| \nabla^{\tilde{\alpha}+\tilde{\ell}} g \|_{q}, \label{eq:2.24} \\ 
		& 
		\left\|
		\partial_{t}^{\ell} 
		\nabla^{\alpha}
		G_{0L}^{(\beta)}(t)g
		\right\|_{p} 
		\le C(1+t)^{-\frac{3}{2}(\frac{1}{q}-\frac{1}{p})-(\frac{1}{q}-\frac{1}{p})+\frac{1}{2} -\frac{\ell-\tilde{\ell}+ \alpha-\tilde{\alpha}}{2}}
		\| \nabla^{\tilde{\alpha}+\tilde{\ell}} g \|_{q}, \label{eq:2.25}  \\
		& \left\| 
		\partial_{t}^{\ell} 
		\nabla^{\alpha}
		G_{1L}^{(\beta)}(t) g
		\right\|_{p} 
		\le C(1+t)^{-\frac{3}{2}(\frac{1}{q}-\frac{1}{p})-(\frac{1}{q}-\frac{1}{p})+1 -\frac{\ell-\tilde{\ell}+ \alpha-\tilde{\alpha}}{2}}
		\| \nabla^{\tilde{\alpha}+\tilde{\ell}} g \|_{q} \label{eq:2.26} 
	\end{align}
	for $t \ge 0$. \\
	{\rm (ii)}
	Let $\ell \ge \tilde{\ell} \ge 0$, $\alpha \ge \tilde{\alpha} \ge 0$ and $t>0$. 
	Then it holds that 
	\begin{equation}
	\begin{split}
		 \sum_{k=M.H}  \| \partial_{t}^{\ell} \nabla^{\alpha} K_{0k}^{(\beta)}(t) g \|_{p}
		\le C e^{-ct} (\| \nabla^{\alpha} g \|_{p} + t^{-\frac{\alpha-\tilde{\alpha}}{2}-(\ell-\frac{\ell}{2})}\| \nabla^{\tilde{\alpha}+\tilde{\ell}} g \|_{p} ) \label{eq:2.27}  
	\end{split}	
	\end{equation}
for $1 \le p \le \infty$
and 
\begin{equation}
	\begin{split}
		 \sum_{k=M.H}  \| \partial_{t}^{\ell} \nabla^{\alpha} K_{1k}^{(\beta)}(t) g \|_{p} 
		 \le C e^{-ct} (\| \nabla^{(\alpha-2)_{+}} g \|_{p} + t^{-\frac{\alpha-\tilde{\alpha}}{2}-(\ell-\frac{\ell}{2})+1}\| \nabla^{\tilde{\alpha}+\tilde{\ell}} g \|_{p} ) \label{eq:2.28} 
	\end{split}	
\end{equation} 
for $1<p<\infty$. \\
	{\rm (iii)}
	Let $\alpha, \ell \ge 0$, $1 \le p \le \infty$ and $t>0$.
	Then it holds that 
	\begin{equation} \label{eq:2.29}
		\begin{split}
			\sum_{k=M.H} \|  \partial^{\ell}_{t} \nabla^{\alpha} \mathcal{R}_{a} \mathcal{R}_{b} \mathcal{F}^{-1} [\mathcal{G}_{0}^{(\beta)}(t,\xi) \chi_{k} ]\|_{p} 
			 \le C e^{-ct} t^{-\frac{3}{2}(1-\frac{1}{p})-\frac{\alpha+\ell}{2}}
		\end{split}
	\end{equation}
	and 
	\begin{equation} \label{eq:2.30}
		\begin{split}
			\sum_{k=M.H}  \| \partial^{\ell}_{t} \nabla^{\alpha} \mathcal{R}_{a} \mathcal{R}_{b} \mathcal{F}^{-1} [\mathcal{G}_{1}^{(\beta)}(t,\xi) \chi_{k} ]\|_{p}  
			\le C e^{-ct} t^{-\frac{3}{2}(1-\frac{1}{p})-\frac{\alpha+\ell}{2}}.
		\end{split}
	\end{equation}
\end{lem}
Secondly, we show the estimates for the solution of wave equations.
Especially, the estimate \eqref{eq:2.33} plays a crucial role to obtain the asymptotic profiles of the solutions as $t \to \infty$ in $L^{1}$ Sobolev spaces.
\begin{lem}  \label{Lem:2.6}
	Let $1 \le p \le \infty$, $\alpha$, $\ell \ge 0$ and $\gamma>0$. There exists $C>0$ such that 
	\begin{equation} \label{eq:2.31}
		\| \partial_{t}^{\ell} W^{(\beta)}_{0}(t)g \|_{p} \le C ( \| \nabla^{\ell} g \|_{p}+t \| \nabla^{\ell+1} g \|_{p}), 
	\end{equation} 
	\begin{equation} \label{eq:2.32}
		\| \partial_{t}^{\ell}  W^{(\beta)}_{1}(t)g \|_{p} \le C t \| \nabla^{\ell} g \|_{p}, 
	\end{equation}
	\begin{equation} \label{eq:2.33}
	\| \nabla^{\alpha}( W^{(\beta)}_{0}(t)g -W^{(\gamma)}_{0}(t)g) \|_{p} \le C t \| \nabla^{\alpha+1} g \|_{p} 
	\end{equation}
	for $t >0$, where 
	\begin{equation*} 
		W^{(\beta)}_{0}(t)g := \mathcal{F}^{-1}[\cos (t \beta |\xi|) \hat{g}], \quad 
		W^{(\beta)}_{1}(t)g := \mathcal{F}^{-1}\left[\frac{\sin (t \beta |\xi|)}{\beta |\xi|} \hat{g} \right].
	\end{equation*}
\end{lem}
\begin{proof}
	We only show the third estimate \eqref{eq:2.33},
	since estimates \eqref{eq:2.31}-\eqref{eq:2.32} are well-known and the direct consequence of the representation formula of the fundamental solutions. 
	Using the representation formula (cf.\cite{S-S})
	\begin{equation*}
	 W^{(\beta)}_{0}(t)g= \frac{1}{4 \pi} \int_{\mathbb{S}^{2}} g(x+t\beta y)dS_{y} 
		+ \frac{t}{4 \pi} \int_{\mathbb{S}^{2}}  y \cdot \nabla g(x+t\beta y)dS_{y},
	\end{equation*} 
where $\cdot$ represents the inner product in $\R^{3}$. 
we see
	\begin{equation*}
	\begin{split}
	& \ \nabla^{\alpha}( W^{(\beta)}_{0}(t)g-W^{(\gamma)}_{0}(t)g ) \\
	& = \frac{t(\beta- \gamma)}{4 \pi} \int_{\mathbb{S}^{2}} y \cdot \nabla^{\alpha+\ell+1}g(x+ty(\theta \beta +(1-\theta) \gamma) )dS_{y} \\
	& + \frac{t}{4 \pi} \int_{\mathbb{S}^{2}} y \cdot \nabla^{\alpha+1} ( g(x+t\beta y) -g(x+t\gamma y)) dS_{y},
	\end{split}	
	\end{equation*}
	where we apply the mean value theorem to have 
	$$
 g(x+t\beta y) -g(x+t\gamma y)=  t (\beta- \gamma)y \cdot \nabla g(x+ty(\theta \beta +(1-\theta) \gamma) )
	$$
	for some $\theta \in [0,1]$.
	Then we conclude the desired estimate \eqref{eq:2.33}.
\end{proof}
Thirdly, we collect the basic estimates for the Riesz transform. 
\begin{lem}  \label{Lem:2.7}
	{\rm (i)} Let $1<p<\infty$. 
	There exists $C>0$ such that 
	\begin{equation} \label{eq:2.34}
		\begin{split}
			\| \mathcal{R}_{a} g\|_{p} \le C \| g \|_{p},
		\end{split}
	\end{equation}
	where 
	\begin{equation*} 
		\begin{split}
			\mathcal{R}_{a} g:= \mathcal{F}^{-1} \left[\frac{\xi_{a}}{|\xi|} \hat{g} \right].
		\end{split}
	\end{equation*}
	{\rm (ii)} 
	Let $\ell \ge 0$, $\alpha \ge 0$ and $\ell+\alpha \ge 1$. 
	There exists $C>0$ such that 
	\begin{equation} \label{eq:2.35}
		\begin{split}
			\| \partial_{t}^{\ell} \nabla^{\alpha} \mathcal{R}_{a}\mathcal{R}_{b} \mathcal{F}^{-1}[e^{-\frac{\nu t |\xi|^{2}}{2}} \chi_{L}] \|_{1} \le C(1+t)^{-\frac{\alpha}{2}-\ell}
		\end{split}
	\end{equation}
	for $t \ge 0$.
\end{lem}
	For the proof of \eqref{eq:2.34}, see e.g. \cite{Gr}.
	The estimate \eqref{eq:2.35} is proved in \cite{K-S}.
Finally, we recall well-know embedding results.
\begin{lem} \label{Lem:2.8}
	There exists a constant $C>0$ such that 
	\begin{align}
	& \| g \|_{L^{1}(\R^{3})} \le C \| g \|_{L^{2}(\R^{3})}^{\frac{1}{4}} \| x^{2} u \|_{L^{2}(\R^{3})}^{\frac{3}{4}}
			=C \| \hat{g} \|_{L^{2}(\R^{3})}^{\frac{1}{4}} \| \nabla_{\xi}^{2} \hat{g} \|_{L^{2}(\R^{3})}^{\frac{3}{4}},  \label{eq:2.36} \\
			& \| g \|_{L^{\infty}(\R^{3})} \le C \| g \|_{L^{2}(\R^{3})}^{\frac{2}{5}} \| \nabla^{2} g \|_{L^{2}(\R^{3})}^{\frac{3}{5}}, \label{eq:2.37} \\
			& \| \nabla g \|_{L^{2p}(\R^{3})} \le C \| g \|_{L^{\infty}(\R^{3})}^{\frac{1}{2}} \| \nabla^{2} g \|_{L^{p}(\R^{3})}^{\frac{1}{2}}, \quad 1\le p< \infty, \label{eq:2.38} \\ 
	& \| \nabla g \|_{L^{\infty}(\R^{3})} \le C \| g \|_{L^{\infty}(\R^{3})}^{\frac{q}{2q-3}} \| \nabla^{2} g \|_{L^{q}(\R^{3})}^{1-\frac{q}{2q-3}}, \quad 3 <q< \infty, \label{eq:2.39} 
	\end{align}
	where $C$ is independent of $g$.
\end{lem}
For the proof of \eqref{eq:2.36}, see e.g. \cite{BTW}.
	The estimate \eqref{eq:2.37}-\eqref{eq:2.39} are proved in \cite{C}.
\section{Main results}
%
In this section, we state the main results of this paper.

Propositions \ref{prop:2.1}-\ref{prop:2.4} provide the large time behavior of the global solutions in the framework of $L^{2}$-Sobolev spaces.
In contrast to \cite{K-T},  
our aim in this paper is to discuss the consistency and smoothing effect of the global solutions in the $L^{p}$ framework, where $1 \le p \le \infty$. 
We firstly consider the case $1<p<\infty$.
We have the following result when $F(u)=\nabla u \nabla^{2} u$.
\begin{thm} \label{thm:3.1}
	Let $1<p<\infty$.
	In addition to the assumption on Proposition \ref{prop:2.1}, assume also that 
	$(f_{0}, f_{1}) \in \{ \dot{W}^{3,p} \}^{3} \times  \{ \dot{W}^{1,p}  \}^{3}$. 
	Then the global solution to \eqref{eq:1.1} constructed in Proposition \ref{prop:2.1} satisfies 
	$$
	u \in \{ C([0,\infty); \dot{W}^{3,p} \cap \dot{W}^{1,p} ) \cap C^{1}([0,\infty); W^{1,p}) \cap C^{1}((0,\infty); \dot{W}^{2,p}) \cap C^{2}((0,\infty); L^{p}) \}^{3},
	$$
	together with the following properties:
	\begin{align} 
			\| \nabla^{\alpha} u(t) \|_{p} & 
			\le C (1+t)^{-\frac{5}{2}(1-\frac{1}{p})+1-\frac{\alpha}{2}}, \quad 1 \le \alpha \le 3, \label{eq:3.1} \\
			\| \nabla^{\alpha}  \partial_{t} u(t) \|_{p} & 
			\le C (1+t)^{-\frac{5}{2}(1-\frac{1}{p})+\frac{1}{2}-\frac{\alpha}{2}}, \quad 0 \le \alpha \le 1 \label{eq:3.2}
	\end{align}
	for $t \ge 0$,
	\begin{align} 
		\| \nabla^{2} \partial_{t} u(t) \|_{p} 
		& \le C (1+t)^{-\frac{5}{2}(1-\frac{1}{p})} t^{-\frac{1}{2}}, \quad \label{eq:3.3} \\
		\| \partial_{t}^{2} u(t) \|_{p} 
		& \le C (1+t)^{-\frac{5}{2}(1-\frac{1}{p})+\frac{1}{2}} t^{-\frac{1}{2}} \label{eq:3.4}
	\end{align}
	for $t > 0$, and
	\begin{align}
		& \| \nabla^{\alpha} (u(t)-G(t)) \|_{p} = o(t^{-\frac{5}{2}(1-\frac{1}{p})+1-\frac{\alpha}{2}}), \quad 1 \le \alpha \le 3, \label{eq:3.5} \\
		& \| \nabla^{\alpha} (\partial_{t} u(t) -H(t)) \|_{p} =o(t^{-\frac{5}{2}(1-\frac{1}{p})+\frac{1}{2}-\frac{\alpha}{2}}), \quad 0 \le \alpha \le 2, \label{eq:3.6}  \\
		& \| \partial_{t}^{2} u(t) -\tilde{G}(t) \|_{p} =o(t^{-\frac{5}{2}(1-\frac{1}{p})})  \label{eq:3.7} 
	\end{align}
	as $t \to \infty$.
\end{thm}
We next state the result for the case $F(u)=\nabla u \nabla \partial_{t} u$.
\begin{thm} \label{thm:3.2}
	Let $1<p<\infty$.
	In addition to the assumption on Proposition \ref{prop:2.2}, assume also that 
	$(f_{0}, f_{1}) \in \{ \dot{W}^{3,p} \}^{3} \times  \{ \dot{W}^{1,p}  \}^{3}$. 
	Then the global solution to \eqref{eq:1.1} constructed in Proposition \ref{prop:2.2} satisfies 
	$$
	u \in \{ C([0,\infty); \dot{W}^{3,p} \cap \dot{W}^{1,p} ) \cap C^{1}([0,\infty); W^{1,p}) \cap C^{1}((0,\infty); \dot{W}^{2,p}) \cap C^{2}((0,\infty); L^{p}) \}^{3},
	$$
	together with the estimates \eqref{eq:3.1}-\eqref{eq:3.7}. 
\end{thm}

It is worth pointing out that although the nonlinear interaction is given by quasi-linear and the linear principal part is hyperbolic, 
we establish the smoothing effect of the global solution by reducing the problem to the integral equation and applying the smoothing effect of the strong damping, 
without using the energy method based on the integration by parts.
Especially, high frequency estimates stated in Corollary \ref{cor:4.11} below are essential to avoid the regularity-loss of the nonlinear terms in $L^{p}$-Sobolev spaces. 

Focusing on the difference of the propagation speed of the hyperbolic parts in the fundamental solutions, we deal with the case $p=1$.
Then the estimate \eqref{eq:2.33} plays an essential role in the proof. 
We first state the result when $F(u)= \nabla u \nabla^{2} u$.
\begin{thm} \label{thm:3.3}
Under the assumption on Proposition \ref{prop:2.1}, assume also that 
$(f_{0}, f_{1}) \in \{ \dot{W}^{3,1} \}^{3} \times  \{ \dot{W}^{1,1}  \}^{3}$. 
Then the global solution to \eqref{eq:1.1} constructed in Proposition \ref{prop:2.1} satisfies 
$$
u \in \{ C([0,\infty); \dot{W}^{1,1} ) \cap C^{1}([0,\infty); W^{1,1}) \cap C^{2}((0,\infty); L^{1}) \}^{3}
$$
with the following properties:
\begin{align} 
	\| \nabla u(t) \|_{1} & 
	\le C (1+t)^{\frac{1}{2}}, \label{eq:3.8} \\
	\| \nabla^{\alpha} \partial_{t} u(t) \|_{1} & 
	\le C (1+t)^{\frac{1}{2}-\frac{\alpha}{2}}, \quad 0 \le \alpha \le 1 \label{eq:3.9}
\end{align}
for $t \ge 0$,
\begin{align} 
	\| \partial_{t}^{2} u(t) \|_{1} 
	\le C (1+t)^{\frac{1}{2}} t^{-\frac{1}{2}} \quad \label{eq:3.10}
\end{align}
for $t > 0$, and
\begin{align}
	& \| \nabla (u(t)-G(t)) \|_{1} = o(t^{\frac{1}{2}}), \label{eq:3.11} \\
	& \| \nabla^{\alpha} (\partial_{t} u(t) -H(t)) \|_{1} =o(t^{\frac{1}{2}-\frac{\alpha}{2}}), \quad 0 \le \alpha \le 1, \label{eq:3.12}  \\
	& \| \partial_{t}^{2} u(t) -\tilde{G}(t) \|_{1} =o(1)  \label{eq:3.13} 
\end{align}
as $t \to \infty$.
\end{thm}
%
We have the following result when $F(u)=\nabla u \nabla \partial{t} u$.
\begin{thm} \label{thm:3.4}
Under the assumption on Proposition \ref{prop:2.2}, assume also that 
$(f_{0}, f_{1}) \in \{ \dot{W}^{3,1} \}^{3} \times  \{ \dot{W}^{1,1}  \}^{3}$. 
Then the global solution to \eqref{eq:1.1} constructed in Proposition \ref{prop:2.2} satisfies 
$$
u \in \{ C([0,\infty); \dot{W}^{1,1} ) \cap C^{1}([0,\infty); W^{1,1}) \cap C^{2}([0,\infty); L^{1}) \}^{3}
$$
with the estimates \eqref{eq:3.8}-\eqref{eq:3.13}.
\end{thm}
Finally we discuss the case $p=\infty$. 
In this case, the estimate \eqref{eq:2.33} does not work well since we cannot obtain the decay properties directly.
Our alternative is to pay attention to the regularity of the initial data, which assures the application of Theorem \ref{thm:3.1} (resp. Theorem \ref{thm:3.2}).
As a consequences, our estimation for the global solution becomes easier since it has sufficient regularity and we obtain the following results: 
\begin{thm} \label{thm:3.5}
	Under the assumption on Proposition \ref{prop:2.1}, assume also that 
	$(f_{0}, f_{1}) \in \{ \dot{W}^{3,\infty} \}^{3} \times  \{ \dot{W}^{1,\infty}  \}^{3}$. 
	Then the global solution to \eqref{eq:1.1} constructed in Proposition \ref{prop:2.1} satisfies 
	$$
	u \in \{ \dot{W}^{2,\infty}(0,\infty; L^{\infty}) \}^{3}
	$$
	with the following time decay properties:
	\begin{align} 
		\| \partial_{t}^{2} u(t) \|_{\infty} 
		\le C (1+t)^{-2} t^{-\frac{1}{2}} \quad \label{eq:3.14}
	\end{align}
	for $t > 0$, and
	\begin{align}
		& \| \partial_{t}^{2} u(t) -\tilde{G}(t) \|_{\infty} =o(t^{-\frac{5}{2}})  \label{eq:3.15} 
	\end{align}
	as $t \to \infty$.
\end{thm}
\begin{thm} \label{thm:3.6}
	Under the assumption on Proposition \ref{prop:2.2}, assume also that 
	$(f_{0}, f_{1}) \in \{ \dot{W}^{3,\infty} \}^{3} \times  \{ \dot{W}^{1,\infty}  \}^{3}$. 
	Then the global solution to \eqref{eq:1.1} constructed in Proposition \ref{prop:2.2} satisfies 
	$$
	u(t) \in \{ \dot{W}^{2,\infty}(0,\infty; L^{\infty}) \}^{3}
	$$
	with the estimates \eqref{eq:3.14}-\eqref{eq:3.15}.
\end{thm}
%
\section{Basic estimates for the fundamental solutions}
In this section, we summarize the results from \cite{K-T}, since our proof of main results deeply depend on them.
At first, we mention the estimates for the low frequency parts of the fundamental solutions to \eqref{eq:1.1}.
For this purpose, we introduce the notation 
\begin{align*}
	\mathbb{K}_{0}^{(\beta)}(t,x) & := \mathcal{R}_{a} \mathcal{R}_{b} \mathcal{F}^{-1}[\mathcal{K}_{0L}^{(\beta)}(t,\xi)\chi_{L} ],  \\
	\mathbb{K}_{1}^{(\beta)}(t,x) & :=\mathcal{R}_{a} \mathcal{R}_{b} \mathcal{F}^{-1}[\mathcal{K}_{1L}^{(\beta)}(t,\xi)\chi_{L} ],  \\ 
	\mathbb{G}_{j}^{(\beta)}(t,x) & :=\mathcal{R}_{a} \mathcal{R}_{b} \mathcal{F}^{-1}[\mathcal{G}_{jL}^{(\beta)}(t,\xi)\chi_{L} ] 
\end{align*}
for $j=0,1$ and $a,b=1,2,3$.
The decay properties of them are described as follows:
\begin{lem} \label{Lem:4.1}
	Let $\alpha\ge \tilde{\alpha} \ge 0$, $\ell \ge \tilde{\ell} \ge 0$, $m \ge 0$, $1 \le q \le p\le \infty$ and $t \ge 0$.
	Then it holds that 
	\begin{align} 
			& \left\| 
			\partial^{\ell}_{t} \nabla^{\alpha} \mathbb{K}_{0}^{(\beta)}(t) \ast g
			\right\|_{p} 
			\le C(1+t)^{-\frac{3}{2}(\frac{1}{q}-\frac{1}{p})-(\frac{1}{q}-\frac{1}{p})+\frac{1}{2}- \frac{\ell-\tilde{\ell}+ \alpha-\tilde{\alpha}}{2}} \| \nabla^{\tilde{\alpha}+\tilde{\ell}} g \|_{q}
		\label{eq:4.1} \\
		 & \left\| 
		 \partial^{\ell}_{t} \nabla^{\alpha}\mathbb{K}_{1}^{(\beta)}(t) \ast g
		 \right\|_{p} 
		 \le C(1+t)^{-\frac{3}{2}(\frac{1}{q}-\frac{1}{p})-(\frac{1}{q}-\frac{1}{p})+1- \frac{\ell-\tilde{\ell}+ \alpha-\tilde{\alpha}}{2}} \| \nabla^{\tilde{\alpha}+\tilde{\ell}} g \|_{q}, \label{eq:4.2} \\
		 & \left\| \partial^{\ell}_{t} \nabla^{\alpha} 
		\mathbb{G}_{0}^{(\beta)}(t) 
		 \right\|_{p} 
		 \le
		 C (1+t)^{
		 	-\frac{3}{2}(1-\frac{1}{p})-(1-\frac{1}{p})+\frac{1}{2}-\frac{\ell+\alpha}{2}}, \label{eq:4.3} \\
	 	& \left\| \partial^{\ell}_{t} \nabla^{\alpha} 
	 	\mathbb{G}_{1}^{(\beta)}(t)
	 	\right\|_{p} 
	 	\le C (1+t)^{-\frac{3}{2}(1-\frac{1}{p})-(1-\frac{1}{p})+1-\frac{\ell+\alpha}{2}, \label{eq:4.4}
	 	}
	\end{align}
	where $1<p \le \infty$ for $\ell+\alpha=0$ and $1 \le p \le \infty$ for $\ell+\alpha \ge 1$.
\end{lem}
Next we recall the expansion formulas of $\mathbb{K}_{j}^{(\beta)}(t) \ast g$ and $K_{j}^{(\beta)}(t) g$ for $j=0,1$ as $t \to \infty$.
\begin{lem} \label{Lem:4.2}
	Let $\alpha\ge \tilde{\alpha} \ge 0$, $\ell \ge \tilde{\ell} \ge 0$, $m \ge 0$, $1 \le q \le p\le \infty$ and $t \ge 0$.
	Then it holds that 
	\begin{equation} \label{eq:4.5}
		\begin{split}
			& 
			\left\| \nabla^{\alpha} \left(
			\partial^{\ell}_{t} K_{0L}^{(\beta)}(t)  g -(-1)^{\frac{\ell}{2}} \beta^{\ell} \nabla^{\ell} G_{0L}^{(\beta)}(t) \ast g  
			\right)
			\right\|_{p} \\
			& +
			\left\| \nabla^{\alpha} \left(
			\partial^{\ell}_{t} \mathbb{K}_{00}^{(\beta)}(t) \ast g -(-1)^{\frac{\ell}{2}} \beta^{\ell} \nabla^{\ell} \mathbb{G}_{0}^{(\beta)}(t) \ast g  
			\right)
			\right\|_{p} \\
			& \le C(1+t)^{-\frac{3}{2}(\frac{1}{q}-\frac{1}{p})-(\frac{1}{q}-\frac{1}{p})- \frac{\ell-\tilde{\ell}+ \alpha-\tilde{\alpha}}{2}} \| \nabla^{\tilde{\alpha}+\tilde{\ell}} g \|_{q}
		\end{split}
	\end{equation}
	for $\ell=2m$,
	\begin{equation} \label{eq:4.6}
		\begin{split}
			& \left\| \nabla^{\alpha} \left(
			\partial^{\ell}_{t} K_{0L}^{(\beta)}(t) \ast g -(-1)^{\frac{\ell+1}{2}} \beta^{\ell+1} \nabla^{\ell+1} G_{1L}^{(\beta)}(t)\ast g 
			\right)
			\right\|_{p} \\
			& +\left\| \nabla^{\alpha} \left(
			\partial^{\ell}_{t} \mathbb{K}_{0}^{(\beta)}(t) \ast g -(-1)^{\frac{\ell+1}{2}} \beta^{\ell+1} \nabla^{\ell+1} \mathbb{G}_{1}^{(\beta)}(t)\ast g 
			\right)
			\right\|_{p} \\
			& \le C(1+t)^{-\frac{3}{2}(\frac{1}{q}-\frac{1}{p})-(\frac{1}{q}-\frac{1}{p})- \frac{\ell-\tilde{\ell}+ \alpha-\tilde{\alpha}}{2}} \| \nabla^{\tilde{\alpha}+\tilde{\ell}} g \|_{q}
		\end{split}
	\end{equation}
	for $\ell=2m+1$,
	\begin{equation} \label{eq:4.7}
		\begin{split}
			& \left\| \nabla^{\alpha} \left(
			\partial^{\ell}_{t} K_{1L}^{(\beta)}(t) \ast g - (-1)^{\frac{\ell}{2}} \beta^{\ell} \nabla^{\ell} G_{1L}^{(\beta)}(t) \ast g
			\right)
			\right\|_{p} \\ 
			& +\left\| \nabla^{\alpha} \left(
			\partial^{\ell}_{t} \mathbb{K}_{1}^{(\beta)}(t) \ast g - (-1)^{\frac{\ell}{2}} \beta^{\ell} \nabla^{\ell} \mathbb{G}_{1}^{(\beta)}(t) \ast g
			\right)
			\right\|_{p} \\ 
			& \le C(1+t)^{-\frac{3}{2}(\frac{1}{q}-\frac{1}{p})-(\frac{1}{q}-\frac{1}{p})+\frac{1}{2}- \frac{\ell-\tilde{\ell}+ \alpha-\tilde{\alpha}}{2}} \| \nabla^{\tilde{\alpha}+\tilde{\ell}} g \|_{q}
		\end{split}
	\end{equation}
	for $\ell=2m$ and 
	\begin{equation} \label{eq:4.8}
		\begin{split}
			& \left\| \nabla^{\alpha} \left(
			\partial^{\ell}_{t} K_{1L}^{(\beta)}(t) \ast g - m_{g} (-1)^{\frac{\ell-1}{2}} \beta^{\ell-1} \nabla^{\ell-1} G_{0L}^{(\beta)}(t) \ast g
			\right)
			\right\|_{p} \\
			& + \left\| \nabla^{\alpha} \left(
			\partial^{\ell}_{t} \mathbb{K}_{1}^{(\beta)}(t) \ast g - m_{g} (-1)^{\frac{\ell-1}{2}} \beta^{\ell-1} \nabla^{\ell-1} \mathbb{G}_{0}^{(\beta)}(t) \ast g
			\right)
			\right\|_{p} \\
			& \le C(1+t)^{-\frac{3}{2}(\frac{1}{q}-\frac{1}{p})-(\frac{1}{q}-\frac{1}{p})+\frac{1}{2}-\frac{\ell-\tilde{\ell}+ \alpha-\tilde{\alpha}}{2}} \| \nabla^{\tilde{\alpha}+\tilde{\ell}} g \|_{q}
		\end{split}
	\end{equation}
	for $\ell=2m+1$.
\end{lem}
The following lemma states the large time behavior of $\mathbb{G}_{0}^{(\beta)}(t) \ast g$ and $\mathbb{G}_{1}^{(\beta)}(t) \ast g$.
\begin{lem} \label{Lem:4.3}
	Let $\alpha, \ell, m \ge 0$ and $g \in L^{1}$.
	Then it holds that 
	\begin{equation} \label{eq:4.9}
		\begin{split}
			\left\| \nabla^{\alpha} \left(
			\partial^{\ell}_{t} \mathbb{G}_{0}^{(\beta)}(t) \ast g - m_{g}(-1)^{\frac{\ell}{2}} \beta^{\ell} \nabla^{\ell} \mathbb{G}_{0}^{(\beta)}(t) 
			\right)
			\right\|_{p} 
			= o(t^{-\frac{3}{2}(1-\frac{1}{p})-(1-\frac{1}{p})+\frac{1}{2}-\frac{\ell+\alpha}{2}})
		\end{split}
	\end{equation}
	for $\ell=2m$,
	\begin{equation} \label{eq:4.10}
		\begin{split}
			\left\| \nabla^{\alpha} \left(
			\partial^{\ell}_{t} \mathbb{G}_{0}^{(\beta)}(t) \ast g - m_{g} (-1)^{\frac{\ell+1}{2}} \beta^{\ell+1} \nabla^{\ell+1} \mathbb{G}_{1}^{(\beta)}(t) 
			\right)
			\right\|_{p} 
			= o(t^{-\frac{3}{2}(1-\frac{1}{p})-(1-\frac{1}{p})+\frac{1}{2}-\frac{\ell+\alpha}{2}})
		\end{split}
	\end{equation}
	for $\ell=2m+1$,
	\begin{equation} \label{eq:4.11}
		\begin{split}
			\left\| \nabla^{\alpha} \left(
			\partial^{\ell}_{t} \mathbb{G}_{1}^{(\beta)}(t) \ast g -  m_{g} (-1)^{\frac{\ell}{2}} \beta^{\ell} \nabla^{\ell} \mathbb{G}_{1}^{(\beta)}(t) 
			\right)
			\right\|_{p} 
			= o(t^{-\frac{3}{2}(1-\frac{1}{p})-(1-\frac{1}{p})+1-\frac{\ell+\alpha}{2} })
		\end{split}
	\end{equation}
	for $\ell=2m$ and 
	\begin{equation} \label{eq:4.12}
		\begin{split}
			\left\| \nabla^{\alpha} \left(
			\partial^{\ell}_{t} \mathbb{G}_{1}^{(\beta)}(t) \ast g - m_{g} (-1)^{\frac{\ell-1}{2}} \beta^{\ell-1} \nabla^{\ell-1} \mathbb{G}_{0}^{(\beta)}(t) 
			\right)
			\right\|_{p} 
			= o(t^{-\frac{3}{2}(1-\frac{1}{p})-(1-\frac{1}{p})+1-\frac{\ell+\alpha}{2} })
		\end{split}
	\end{equation}
	for $\ell=2m+1$, 
	as $t \to \infty$, where $1<p \le \infty$ for $\ell+\alpha=0$ and $1 \le p \le \infty$ for $\ell+\alpha \ge 1$.
	Here $m_{g}$ is defined by 
	\begin{equation} \label{eq:4.13}
		\begin{split}
	m_{g}:=\displaystyle\int_{\R^{3}} g(x) dx.
\end{split}
\end{equation}
\end{lem}
By a similar way, we also have the large time behavior of 
$G_{0L}^{(\beta)}(t) \ast g$ and $G_{1}^{(\beta)}(t) \ast g$.
We note that in this case, 
we can deal with $p=1$ and $\ell+\alpha=0$.
\begin{lem} \label{Lem:4.4}
	Let $\alpha, \ell, m \ge 0$, $1 \le p \le \infty$ and $g \in L^{1}$.
	Then it holds that 
	\begin{equation} \label{eq:4.14}
		\begin{split}
			\left\| \nabla^{\alpha} \left(
			\partial^{\ell}_{t} G_{0L}^{(\beta)}(t) \ast g - m_{g}(-1)^{\frac{\ell}{2}} \beta^{\ell} \nabla^{\ell} G_{0L}^{(\beta)}(t) 
			\right)
			\right\|_{p} 
			= o(t^{-\frac{3}{2}(1-\frac{1}{p})-(1-\frac{1}{p})+\frac{1}{2}-\frac{\ell+\alpha}{2}})
		\end{split}
	\end{equation}
	for $\ell=2m$,
	\begin{equation} \label{eq:4.15}
		\begin{split}
			\left\| \nabla^{\alpha} \left(
			\partial^{\ell}_{t} G_{0L}^{(\beta)}(t) \ast g - m_{g} (-1)^{\frac{\ell+1}{2}} \beta^{\ell+1} \nabla^{\ell+1} G_{1L}^{(\beta)}(t) 
			\right)
			\right\|_{p} 
			= o(t^{-\frac{3}{2}(1-\frac{1}{p})-(1-\frac{1}{p})+\frac{1}{2}-\frac{\ell+\alpha}{2}})
		\end{split}
	\end{equation}
	for $\ell=2m+1$,
	\begin{equation} \label{eq:4.16}
		\begin{split}
			\left\| \nabla^{\alpha} \left(
			\partial^{\ell}_{t} G_{1L}^{(\beta)}(t) \ast g -  m_{g} (-1)^{\frac{\ell}{2}} \beta^{\ell} \nabla^{\ell} G_{1L}^{(\beta)}(t) 
			\right)
			\right\|_{p} 
			= o(t^{-\frac{3}{2}(1-\frac{1}{p})-(1-\frac{1}{p})+1-\frac{\ell+\alpha}{2} })
		\end{split}
	\end{equation}
	for $\ell=2m$ and 
	\begin{equation} \label{eq:4.17}
		\begin{split}
			\left\| \nabla^{\alpha} \left(
			\partial^{\ell}_{t} G_{1L}^{(\beta)}(t) \ast g - m_{g} (-1)^{\frac{\ell-1}{2}} \beta^{\ell-1} \nabla^{\ell-1} G_{0L}^{(\beta)}(t) 
			\right)
			\right\|_{p} 
			= o(t^{-\frac{3}{2}(1-\frac{1}{p})-(1-\frac{1}{p})+1-\frac{\ell+\alpha}{2} })
		\end{split}
	\end{equation}
	for $\ell=2m+1$, 
	as $t \to \infty$.
\end{lem}

As we see from \eqref{eq:4.10}, we cannot expect the $L^{1}$-$L^{1}$ estimates for $\mathbb{G}_{0}^{(\beta)}(t)\ast g$,
because of the Riesz transform (cf. \cite{K-S}).
On the other hand, the following proposition yields information about the $L^{1}$ estimates for $(\mathbb{G}_{0}^{(\beta)}(t) -\mathbb{G}_{0}^{(\gamma)}(t))\ast g$
with $\beta, \gamma>0$ and $\beta \neq \gamma$.
\begin{prop} \label{Prop:4.5}
	Let $\beta, \gamma>0$ with $\beta \neq \gamma$.
	Then it holds that 
		\begin{align}
			& \left\| 
			(\mathbb{G}_{0}^{(\beta)}(t) -\mathbb{G}_{0}^{(\gamma)}(t))\ast g 
			\right\|_{p} \le C t^{\frac{1}{2}} \| g\|_{p} , \label{eq:4.18} 
		\end{align}
		for $1 \le p \le \infty$ and $t \ge 0$, and 
		\begin{align}
			& \left\| 
			(\mathbb{G}_{0}^{(\beta)}(t) -\mathbb{G}_{0}^{(\gamma)}(t))\ast g - m_{g} (\mathbb{G}_{0}^{(\beta)}(t) -\mathbb{G}_{0}^{(\gamma)}(t))
			\right\|_{1} 
			= o(t^{\frac{1}{2}}) \label{eq:4.19}
		\end{align}
	as $t \to \infty$ for $g \in L^{1}$.
\end{prop}
\begin{proof}
	We firstly prove the estimate \eqref{eq:4.18}.
	Noting that 
	\begin{equation*}
		\begin{split}
			(\mathbb{G}_{0}^{(\beta)}(t) -\mathbb{G}_{0}^{(\gamma)}(t))\ast g
			 = (W^{(\beta)}_{0}(t) -W^{(\gamma)}_{0}(t)) \mathcal{R}_{a}\mathcal{R}_{b} \mathcal{F}^{-1}[e^{-\frac{\nu t |\xi|^{2}}{2}} \chi_{L}] \ast g, 
		\end{split}
	\end{equation*}
we apply the estimates \eqref{eq:2.32} and \eqref{eq:2.34} to see that 
\begin{equation*}
	\begin{split}
		\|(\mathbb{G}_{0}^{(\beta)}(t) -\mathbb{G}_{0}^{(\gamma)}(t))\ast g \|_{p}
		\le C t \| \nabla \mathcal{R}_{a}\mathcal{R}_{b} \mathcal{F}^{-1}[e^{-\frac{\nu t |\xi|^{2}}{2}} \chi_{L}] \ast g \|_{p} \le C t^{\frac{1}{2}} \| g \|_{p},
	\end{split}
\end{equation*}
which is the desired estimate \eqref{eq:4.18}.
Next we show the estimate \eqref{eq:4.19}.
	For the simplicity of the notation, we define $H^{(\beta, \gamma)}_{0}(t,x)$ as 
	\begin{equation*}
		\begin{split}
			H^{(\beta, \gamma)}_{0}(t,x):= \mathbb{G}_{0}^{(\beta)}(t,x) -\mathbb{G}_{0}^{(\gamma)}(t,x).
		\end{split}
	\end{equation*}	
	Here we can rephrase \eqref{eq:4.19} as 
	\begin{equation*}
		\begin{split}
			\| H^{(\beta, \gamma)}_{0}(t) \|_{1} \le Ct^{\frac{1}{2}}.
		\end{split}
	\end{equation*}
	We also easily have 
	\begin{equation*}
		\begin{split}
			\| \nabla H^{(\beta, \gamma)}_{0}(t) \|_{1} \le C \| \nabla \mathbb{G}_{0}^{(\beta)}(t) \|_{1}
			+ C\| \nabla \mathbb{G}_{0}^{(\gamma)}(t) \|_{1}\le C
		\end{split}
	\end{equation*}
	by \eqref{eq:4.3}.
	Now we observe that 
	\begin{equation*}
		\begin{split}
			H(t) \ast g -m_{g} H(t,x)  
			& = \int_{|y|\le t^{\frac{1}{4}}} (H(t,x-y)-H(t,x) ) g(y) dy \\
			& +\int_{|y|\ge t^{\frac{1}{4}}} H(t,x-y) g(y) dy -\int_{|y|\ge t^{\frac{1}{4}}} H(t,x) g(y) dy.
		\end{split}
	\end{equation*}
	Therefore when $t \ge 1$, the mean value theorem gives 
	$$
	|H^{(\beta, \gamma)}_{0}(t,x)-H^{(\beta, \gamma)}_{0}(t,y) | \le C |y| |\nabla H^{(\beta, \gamma)}_{0}(t,x-\theta y)|
	$$
	for some $\theta \in [0,1]$ and 
	\begin{equation} \label{eq:4.20}
		\begin{split}
			& \| H(t) \ast g -m_{g} H(t,x)   \|_{1} \\ 
			& \le C t^{\frac{1}{4}} \int_{|y|\le t^{\frac{1}{4}}} \| \nabla H^{(\beta, \gamma)}_{0}(t) \|_{1} |g(y)| dy \\
			& +\int_{|y|\ge t^{\frac{1}{4}}} \| H^{(\beta, \gamma)}_{0}(t) \|_{1} |g(y)| dy +\int_{|y|\ge t^{\frac{1}{4}}} \|  H^{(\beta, \gamma)}_{0}(t) \|_{1} |g(y)| dy \\
			& \le C t^{\frac{1}{4}} \| g \|_{1} + C t^{\frac{1}{2}} \int_{|y|\ge t^{\frac{1}{4}}}  |g(y)| dy. 
		\end{split}
	\end{equation}
	Since $g \in L^{1}$, we see $\displaystyle\lim_{t \to \infty} \int_{|y|\ge t^{\frac{1}{4}}}  |g(y)| dy=0$.
	Therefore the estimate \eqref{eq:4.20} implies the desired estimate.
	We complete the proof of the proposition.
\end{proof}
Combining Lemmas \ref{Lem:4.2}-\ref{Lem:4.3}, we can obtain the approximation formulas of $\mathbb{K}_{0}^{(\beta)}(t) \ast g$ and $\mathbb{K}_{1}^{(\beta)}(t) \ast g$ 
by the diffusion waves with the Riesz transform. 
\begin{cor} \label{cor:4.6}
	Under the assumption on Lemma \ref{Lem:4.3}, 
	the following estimates hold.
	\begin{equation} \label{eq:4.21}
		\begin{split}
			& \left\| \nabla^{\alpha} \left(
			\partial^{\ell}_{t} \mathbb{K}_{0}^{(\beta)}(t) \ast g - m_{g}(-1)^{\frac{\ell}{2}} \beta^{\ell} \nabla^{\ell} \mathcal{R}_{a} \mathcal{R}_{b} G_{0}^{(\beta)}(t) 
			\right)
			\right\|_{p}  =o(t^{-\frac{5}{2}(1-\frac{1}{p})+\frac{1}{2}-\frac{\ell+ \alpha}{2}} )
		\end{split}
	\end{equation}
	for $\ell=2m$,
	\begin{equation} \label{eq:4.22}
		\begin{split}
			& \left\| \nabla^{\alpha} \left(
			\partial^{\ell}_{t} \mathbb{K}_{0}^{(\beta)}(t) \ast g -m_{g} (-1)^{\frac{\ell+1}{2}} \beta^{\ell+1} \nabla^{\ell+1} \mathcal{R}_{a} \mathcal{R}_{b} G_{1}^{(\beta)}(t) 
			\right)
			\right\|_{p} =o(t^{-\frac{5}{2}(1-\frac{1}{p})+\frac{1}{2}-\frac{\ell+ \alpha}{2}} )
		\end{split}
	\end{equation}
	for $\ell=2m+1$,
	\begin{equation} \label{eq:4.23}
		\begin{split}
			& \left\| \nabla^{\alpha} \left(
			\partial^{\ell}_{t} \mathbb{K}_{1}^{(\beta)}(t) \ast g - m_{g}(-1)^{\frac{\ell}{2}} \beta^{\ell} \nabla^{\ell} \mathcal{R}_{a} \mathcal{R}_{b} G_{1}^{(\beta)}(t) 
			\right)
			\right\|_{p} =o(t^{-\frac{5}{2}(1-\frac{1}{p})+1-\frac{\ell+ \alpha}{2}} )
		\end{split}
	\end{equation}
	for $\ell=2m$ and 
	\begin{equation} \label{eq:4.24}
		\begin{split}
			& \left\| \nabla^{\alpha} \left(
			\partial^{\ell}_{t} \mathbb{K}_{1}^{(\beta)}(t) \ast g - m_{g} (-1)^{\frac{\ell-1}{2}} \beta^{\ell-1} \nabla^{\ell-1} \mathcal{R}_{a} \mathcal{R}_{b} G_{0}^{(\beta)}(t) 
			\right)
			\right\|_{p} =o(t^{-\frac{5}{2}(1-\frac{1}{p})+1-\frac{\ell+ \alpha}{2}} )
		\end{split}
	\end{equation}
	for $\ell=2m+1$, as $t \to \infty$.
\end{cor}
As we expect, we have a similar conclusion to $K_{0L}^{(\beta)}(t)$ and $K_{1L}^{(\beta)}(t) g$, 
which is formulated as follows. 
\begin{cor} \label{cor:4.7}
	Under the assumption on Lemma \ref{Lem:4.4}, 
	the following estimates hold.
	\begin{equation} \label{eq:4.25}
		\begin{split}
			& \left\| \nabla^{\alpha} \left(
			\partial^{\ell}_{t} K_{0L}^{(\beta)}(t) \ast g - m_{g}(-1)^{\frac{\ell}{2}} \beta^{\ell} \nabla^{\ell} G_{0}^{(\beta)}(t) 
			\right)
			\right\|_{p}  =o(t^{-\frac{5}{2}(1-\frac{1}{p})+\frac{1}{2}-\frac{\ell+ \alpha}{2}} )
		\end{split}
	\end{equation}
	for $\ell=2m$,
	\begin{equation} \label{eq:4.26}
		\begin{split}
			& \left\| \nabla^{\alpha} \left(
			\partial^{\ell}_{t} K_{0L}^{(\beta)}(t) \ast g -m_{g} (-1)^{\frac{\ell+1}{2}} \beta^{\ell+1} \nabla^{\ell+1} G_{1}^{(\beta)}(t) 
			\right)
			\right\|_{p} =o(t^{-\frac{5}{2}(1-\frac{1}{p})+\frac{1}{2}-\frac{\ell+ \alpha}{2}} )
		\end{split}
	\end{equation}
	for $\ell=2m+1$,
	\begin{equation} \label{eq:4.27}
		\begin{split}
			& \left\| \nabla^{\alpha} \left(
			\partial^{\ell}_{t} K_{1L}^{(\beta)}(t) \ast g - m_{g}(-1)^{\frac{\ell}{2}} \beta^{\ell} \nabla^{\ell} G_{1}^{(\beta)}(t) 
			\right)
			\right\|_{p} =o(t^{-\frac{5}{2}(1-\frac{1}{p})+1-\frac{\ell+ \alpha}{2}} )
		\end{split}
	\end{equation}
	for $\ell=2m$ and 
	\begin{equation} \label{eq:4.28}
		\begin{split}
			& \left\| \nabla^{\alpha} \left(
			\partial^{\ell}_{t} K_{1L}^{(\beta)}(t) \ast g - m_{g} (-1)^{\frac{\ell-1}{2}} \beta^{\ell-1} \nabla^{\ell-1} G_{0}^{(\beta)}(t) 
			\right)
			\right\|_{p} =o(t^{-\frac{5}{2}(1-\frac{1}{p})+1-\frac{\ell+ \alpha}{2}} )
		\end{split}
	\end{equation}
	for $\ell=2m+1$, as $t \to \infty$.
\end{cor}
The following lemma plays an essential role to obtain the asymptotic profiles of the nonlinear term as $t \to \infty$.
\begin{lem} \label{Lem:4.8}
	Let $\alpha, \ell, m \ge 0$.
	Suppose that $f \in L^{1}(0,\infty;L^{1}(\R^{3}))$ with $\| f(t) \|_{1} \le C (1+t)^{-2}$.
	Then the following estimates hold as $t \to \infty$.\\
	{\rm (i)}
	\begin{equation} \label{eq:4.29}
		\begin{split}
			& \left\| \nabla^{\alpha} \left(\int_{0}^{\frac{t}{2}} \partial_{t}^{\ell} \mathbb{K}^{(\beta)}_{1}(t-\tau) \ast f(\tau) d \tau 
			-(-1)^{\frac{\ell}{2}} \beta^{\ell} \nabla^{\ell} 
			\int_{0}^{\frac{t}{2}} 
			\int_{\R^{n}} f(\tau,y) dy d \tau\,  
			\mathbb{G}_{1}(t)
			\right)
			\right\|_{p} \\
			& = o(t^{-\frac{5}{2}(1-\frac{1}{p})+1-\frac{\alpha+\ell}{2}})
		\end{split}
	\end{equation}
	for $\ell=2m$ and 
	\begin{equation} \label{eq:4.30}
		\begin{split}
			& \left\| \nabla^{\alpha} \left(\int_{0}^{\frac{t}{2}} \partial_{t}^{\ell} \mathbb{K}^{(\beta)}_{1}(t-\tau) \ast f(\tau) d \tau 
			-(-1)^{\frac{\ell-1}{2}} \beta^{\ell-1} \nabla^{\ell-1} 
			\int_{0}^{\frac{t}{2}} 
			\int_{\R^{n}} f(\tau,y) dy d \tau\,  
			\mathbb{G}_{0}(t)
			\right)
			\right\|_{p} \\
			& = o(t^{-\frac{5}{2}(1-\frac{1}{p})+1-\frac{\alpha+\ell}{2}})
		\end{split}
	\end{equation}
	for $\ell=2m+1$, where $1 <p \le \infty$ for $\ell+\alpha=0$ and $1 \le p \le \infty$ for $\ell+\alpha \ge 1$.\\
	{\rm (ii)}
	\begin{equation} \label{eq:4.31}
		\begin{split}
			& \left\| \nabla^{\alpha} \left(\int_{0}^{\frac{t}{2}} \partial_{t}^{\ell} K^{(\beta)}_{1L}(t-\tau) f(\tau) d \tau 
			-(-1)^{\frac{\ell}{2}} \beta^{\ell} \nabla^{\ell} 
			\int_{0}^{\frac{t}{2}} 
			\int_{\R^{n}} f(\tau,y) dy d \tau\,  
			G_{1L}(t)
			\right)
			\right\|_{p} \\
			& = o(t^{-\frac{5}{2}(1-\frac{1}{p})+1-\frac{\alpha+\ell}{2}})
		\end{split}
	\end{equation}
	for $\ell=2m$ and 
	\begin{equation} \label{eq:4.32}
		\begin{split}
			& \left\| \nabla^{\alpha} \left(\int_{0}^{\frac{t}{2}} \partial_{t}^{\ell} K^{(\beta)}_{1L}(t-\tau) \ast f(\tau) d \tau 
			-(-1)^{\frac{\ell-1}{2}} \beta^{\ell-1} \nabla^{\ell-1} 
			\int_{0}^{\frac{t}{2}} 
			\int_{\R^{n}} f(\tau,y) dy d \tau\,  
			G_{0L}(t)
			\right)
			\right\|_{p} \\
			& = o(t^{-\frac{5}{2}(1-\frac{1}{p})+1-\frac{\alpha+\ell}{2}})
		\end{split}
	\end{equation}
	for $\ell=2m+1$, where $1 \le p \le \infty$.
\end{lem}
We conclude this section with the estimates for the middle and high frequency parts.
The following results are firstly mentioned in \cite{K-T}, and as announced there, we give the detailed proof.
For this purpose, following \cite{Shibata} (see also \cite{K-S}), 
we introduce the evolution operators  
\begin{equation*}
	\begin{split}
		J_{H}^{(\beta)}(t)g  & :=
		\mathcal{F}^{-1} \left[
		\frac{\sigma_{+}^{(\beta)} e^{\sigma_{-}^{(\beta)} t}}{\sigma_{+}^{(\beta)}-\sigma_{-}^{(\beta)}}
		\chi_{H} \frac{\xi_{a}\xi_{b}}{|\xi|^{2}}
		\hat{g}
		\right], \\
		J_{0H +}^{(\beta)}(t)g 
		& := \mathcal{F}^{-1} \left[
		e^{\sigma_{\pm}^{(\beta)} t} \chi_{H} \frac{\xi_{a}\xi_{b}}{|\xi|^{2}}
		\hat{g}
		\right], \quad 
		J_{1H +}^{(\beta)}(t)g  :=
		\mathcal{F}^{-1} \left[
		\frac{e^{\sigma_{\pm}^{(\beta)} t}}{\sigma_{+}^{(\beta)}-\sigma_{-}^{(\beta)}}
		\chi_{H} \frac{\xi_{a}\xi_{b}}{|\xi|^{2}}
		\hat{g}
		\right]
	\end{split}
\end{equation*}
for $a, b=1,2,3$.
Then we can decompose $K_{jH}^{(\beta)}(t)\mathcal{R}_{a} \mathcal{R}_{b}g$ $(j=0,1)$ as follows:
\begin{equation} \label{eq:4.33}
	\begin{split}
		K_{0H}^{(\beta)}(t)\mathcal{R}_{a} \mathcal{R}_{b}g 
		& = J_{0H+}^{(\beta)}(t) g 
		-\partial_{t} J_{1H+}^{(\beta)}(t) g 
		+J_{H}^{(\beta)}(t) g, \\
		K_{1H}^{(\beta)}(t)\mathcal{R}_{a} \mathcal{R}_{b}g 
		& = J_{1H+}^{(\beta)}(t) g 
		- J_{1H-}^{(\beta)}(t) g.
	\end{split}
\end{equation}
Now we claim the $L^{p}$-$L^{p}$ type estimates for  $J_{H}^{(\beta)}(t)g$, $J_{0H \pm}^{(\beta)}(t)g$ and $J_{1H -}^{(\beta)}(t)g $. 
\begin{lem} \label{Lem:4.9}
	Let $\alpha \ge \tilde{\alpha} \ge 0$, $\ell \ge 2 \tilde{\ell} \ge 0$ and $1 \le q \le p \le \infty$. 
	Then it holds that
	\begin{align}
		& \| \partial_{t}^{\ell} \nabla^{\alpha} (J_{0H+}^{(\beta)}(t) g -e^{-\frac{\beta^2}{\nu} t}  \mathcal{R}_{a} \mathcal{R}_{b} \mathcal{F}^{-1}[\chi_{H} \hat{g}]) \|_{p} 
		\le C e^{-ct} \| \nabla^{\alpha} g \|_{p}, \label{eq:4.34} \\
		& \| \partial_{t}^{\ell} \nabla^{\alpha} J_{1H+}^{(\beta)}(t) g \|_{p} \le C e^{-ct} \| \nabla^{\alpha} g \|_{p} \label{eq:4.35} 
	\end{align}
	for $t \ge 0$ and
	\begin{align}
		& \| \partial_{t}^{\ell} \nabla^{\alpha} J_{H}^{(\beta)}(t) g \|_{p} \le C e^{-ct} t^{-\frac{3}{2}(\frac{1}{q}-\frac{1}{p})-\frac{\alpha-\tilde{\alpha}}{2}-(\ell-\frac{\tilde{\ell}}{2})+1} 
		\| \nabla^{\tilde{\alpha}+\tilde{\ell}} g \|_{p},
		\label{eq:4.36} \\
		& \| \partial_{t}^{\ell} \nabla^{\alpha} J_{1H-}^{(\beta)}(t) g \|_{p} \le C e^{-ct} t^{-\frac{3}{2}(\frac{1}{q}-\frac{1}{p})-\frac{\alpha-\tilde{\alpha}}{2}-(\ell-\frac{\tilde{\ell}}{2})+1} 
		\| \nabla^{\tilde{\alpha}+\tilde{\ell}} g \|_{p}.
		\label{eq:4.37}
	\end{align}
	for $t > 0$. 
\end{lem}	
\begin{rem}
	We remark that our proof of Lemma \ref{Lem:4.9} is easily extended to $n$ dimensional case.
\end{rem}
\begin{proof}
	At first, we show the estimate \eqref{eq:4.35}. 
	Noting 
	\begin{equation*}
		\begin{split}
			\frac{1}{i^{|\gamma|} x^{\gamma}} e^{i x \cdot \xi} = \partial_{\xi}^{\gamma}  e^{i x \cdot \xi}
		\end{split}
	\end{equation*}
	for $x \neq 0$ and $\gamma \in \mathbb{Z}_{+}^{3}$, 
	we have the following decomposition:
	\begin{equation} \label{eq:4.38}
		\begin{split}
			\partial_{t}^{\ell} \nabla^{\alpha} J_{1H+}^{(\beta)}(t) g  & =\mathcal{F}^{-1} \left[
			(\sigma_{+}^{(\beta)})^{\ell}
			\frac{e^{\sigma_{+}^{(\beta)} t}}{\sigma_{+}^{(\beta)}-\sigma_{-}^{(\beta)}}
			\chi_{H} \frac{\xi_{a}\xi_{b}}{|\xi|^{2}} \hat{g}
			\right] \\
			& =
			(2 \pi)^{-\frac{3}{2}}
			\left(
			\lim_{\varepsilon \to 0} 
			\int_{\R^{n}} e^{i x \cdot \xi-\varepsilon |\xi|^{2}} 
			(\sigma_{+}^{(\beta)})^{\ell}
			\frac{e^{\sigma_{+}^{(\beta)} t}}{\sigma_{+}^{(\beta)}-\sigma_{-}^{(\beta)}}
			\chi_{H} \frac{\xi_{a}\xi_{b}}{|\xi|^{2}} d \xi
			\right)  \ast \nabla^{\alpha} g
			\\
			& = \left( 
			\frac{C}{x^{\gamma}} \lim_{\varepsilon \to 0} 
			\int_{\R^{n}} e^{i x \cdot \xi-\varepsilon |\xi|^{2}} 
			\partial_{\xi}^{\gamma} \left(
			(\sigma_{+}^{(\beta)})^{\ell}
			\frac{e^{\sigma_{+}^{(\beta)} t}}{\sigma_{+}^{(\beta)}-\sigma_{-}^{(\beta)}}
			\chi_{H} \frac{\xi_{a}\xi_{b}}{|\xi|^{2}}
			\right)
			d \xi 
			\right) \ast \nabla^{\alpha} g
			\\
			& = J_{11}(t,x)\ast \nabla^{\alpha} g+J_{12}(t,x)\ast \nabla^{\alpha} g,
		\end{split}
	\end{equation}
	where
	\begin{equation*} 
		\begin{split}
			J_{11}(t,x) & := \frac{C}{x^{\gamma}} \lim_{\varepsilon \to 0} 
			\int_{\R^{n}} e^{i x \cdot \xi-\varepsilon |\xi|^{2}} \chi_{H} 
			\partial_{\xi}^{\gamma} \left(
			(\sigma_{+}^{(\beta)})^{\ell}
			\frac{e^{\sigma_{+}^{(\beta)} t}}{\sigma_{+}^{(\beta)}-\sigma_{-}^{(\beta)}}
			\frac{\xi_{a}\xi_{b}}{|\xi|^{2}}
			\right)
			d \xi, \\
			J_{12}(t,x) & := \frac{C}{x^{\gamma}}  
			\int_{\R^{n}} e^{i x \cdot \xi} 
			\sum_{\substack{ \tilde{\gamma}_{1} +\tilde{\gamma}_{2}=\gamma \\ |\gamma_{2}| \ge 1} }
			\partial_{\xi}^{\tilde{\gamma}_{1}} \left(
			(\sigma_{+}^{(\beta)})^{\ell}
			\frac{e^{\sigma_{+}^{(\beta)} t}}{\sigma_{+}^{(\beta)}-\sigma_{-}^{(\beta)}}
			\frac{\xi_{a}\xi_{b}}{|\xi|^{2}}
			\right)
			\partial_{\xi}^{\tilde{\gamma}_{2}} \chi_{H} 
			d \xi. 
		\end{split}
	\end{equation*}
	Now, direct calculations show
	\begin{equation} \label{eq:4.39}
		\begin{split}
			\partial_{\xi}^{\gamma} \left(
			(\sigma_{+}^{(\beta)})^{\ell}
			\frac{e^{\sigma_{+}^{(\beta)} t}}{\sigma_{+}^{(\beta)}-\sigma_{-}^{(\beta)}}
			\frac{\xi_{a}\xi_{b}}{|\xi|^{2}}
			\right) = \sum_{\gamma_{1} +\gamma_{2}=\gamma } C_{\gamma}
			\partial_{\xi}^{\gamma_{1}} \left(
			\frac{(\sigma_{+}^{(\beta)})^{\ell}}{\sigma_{+}^{(\beta)}-\sigma_{-}^{(\beta)}}
			\frac{\xi_{a}\xi_{b}}{|\xi|^{2}}
			\right) 
			\partial_{\xi}^{\gamma_{2}} e^{\sigma_{+}^{(\beta)} t}
		\end{split}
	\end{equation}
	and 
	\begin{equation} \label{eq:4.40}
		\begin{split}
			\partial_{\xi}^{\gamma_{1}} \left(
			\frac{(\sigma_{+}^{(\beta)})^{\ell}}{\sigma_{+}^{(\beta)}-\sigma_{-}^{(\beta)}}
			\frac{\xi_{a}\xi_{b}}{|\xi|^{2}}
			\right)=O(|\xi|^{-|\gamma_{1}|-2}) 
		\end{split}
	\end{equation}
	as $|\xi| \to \infty$.
	On the other hand,
	observing that
	\begin{equation*}
		\begin{split}
			\left|\prod_{k=1}^{m} \partial_{\xi}^{\tilde{\gamma}_{k}} (t \sigma_{+}^{(\beta)}) \right| 
			\le C \prod_{k=1}^{m} (t |\xi|^{-|\tilde{\gamma}_{k}|}) =Ct^{m} |\xi|^{-|\gamma_{2}|}
		\end{split}
	\end{equation*}
	for $\xi \in \supp \chi_{H}$ and $\displaystyle\sum_{k=1}^{m} \tilde{\gamma}_{k} =\gamma_{2} $, 
	we have
	\begin{equation} \label{eq:4.41}
		\begin{split}
			\chi_{H}|\partial_{\xi}^{\gamma_{2}} e^{t \sigma_{+}^{(\beta)}}| 
			& =
			\chi_{H} \sum_{m=1}^{|\gamma_{2}|}
			e^{t \sigma_{+}^{(\beta)}} 
			\left| \sum_{\substack{ \sum_{k=1}^{m} \tilde{\gamma}_{k} =\gamma_{2} \\ \tilde{\gamma}_{k} \ge 1} }
			\prod_{k=1}^{m} \partial_{\xi}^{\tilde{\gamma}_{k}} (t \sigma_{+}^{(\beta)}) \right| \\
			&
			\le C e^{-ct} \chi_{H}
			\sum_{m=1}^{|\gamma_{2}|}
			t^{m} |\xi|^{-|\gamma_{2}|} \le C e^{-ct} \chi_{H} |\xi|^{-|\gamma_{2}|}.
		\end{split}
	\end{equation}
	Then it follows from the estimates \eqref{eq:4.39}-\eqref{eq:4.41} that 
	\begin{equation} \label{eq:4.42}
		\begin{split}
			\left|\partial_{\xi}^{\gamma} \left(
			(\sigma_{+}^{(\beta)})^{\ell}
			\frac{e^{\sigma_{+}^{(\beta)} t}}{\sigma_{+}^{(\beta)}-\sigma_{-}^{(\beta)}}
			\frac{\xi_{a}\xi_{b}}{|\xi|^{2}}
			\right) \right| 
			\le C e^{-c t} |\xi|^{-|\gamma|-2}.
		\end{split}
	\end{equation}
	The estimate \eqref{eq:4.42} implies that
	\begin{equation*} 
		\begin{split}
			|J_{11}(t,x)| 
			=
			\begin{cases}
				& e^{-ct}O(|x|^{-2})  \quad \text{for}\ |x| \le 1, \\
				& e^{-ct} O(|x|^{-4})  \quad \text{for}\ |x| \ge 1,
			\end{cases}
		\end{split}
	\end{equation*}
	where we choose $\gamma \in \mathbb{Z}_{+}^{3}$ as $|\gamma|=2$ for $|x|\le 1$ and $|\gamma|=4$ for $|x|\ge 1$.
	Thus we immediately have 
	\begin{equation} \label{eq:4.43}
		\begin{split}
			\|J_{11}(t)\|_{1}  \le  e^{-ct}.
		\end{split}
	\end{equation}
	To obtain the estimate for $J_{12}(t,x)$, we apply the integration by parts to see that 
	\begin{equation*}
		\begin{split}
			\left|
			\int_{\R^{3}} e^{i x \cdot \xi} 
			\sum_{\substack{ \tilde{\gamma}_{1} +\tilde{\gamma}_{2}=\gamma \\ |\gamma_{2}| \ge 1} }
			\partial_{\xi}^{\tilde{\gamma}_{1}} \left(
			(\sigma_{+}^{(\beta)})^{\ell}
			\frac{e^{\sigma_{+}^{(\beta)} t}}{\sigma_{+}^{(\beta)}-\sigma_{-}^{(\beta)}}
			\frac{\xi_{a}\xi_{b}}{|\xi|^{2}}
			\right)
			\partial_{\xi}^{\tilde{\gamma}_{2}} \chi_{H} 
			d \xi 
			\right| \le C e^{-ct}, 
		\end{split}
	\end{equation*}
	since $\supp \partial_{\xi}^{\tilde{\gamma}_{2}} \chi_{H}$ with $\tilde{\gamma}_{2} \neq 0$ is compact in $\R^{3}$.
	Therefore we choose $\gamma \in \mathbb{Z}^{3}$ satisfying $|\gamma|=2$ for $|x|\le 1$ and $|\gamma|=4$ for $|x|\ge 1$ again 
	to have 
	\begin{equation*}
		\begin{split}
			|J_{12}(t,x)| 
			=
			\begin{cases}
				& e^{-ct}O(|x|^{-2})  \quad \text{for}\ |x| \le 1, \\
				& e^{-ct} O(|x|^{-4})  \quad \text{for}\ |x| \ge 1,
			\end{cases}
		\end{split}
	\end{equation*}
	which shows 
	\begin{equation} \label{eq:4.44}
		\begin{split}
			\|J_{12}(t)\|_{1}  \le  e^{-ct}.
		\end{split}
	\end{equation}
	Combining the estimates \eqref{eq:4.43} and \eqref{eq:4.44}, 
	we obtain the estimate 
	\begin{equation*}
		\begin{split}
			\| \partial_{t}^{\ell} \nabla^{\alpha} J_{1H+}^{(\beta)}(t) g \|_{p} 
			\le C (
			\| J_{11}(t)\|_{1}
			+
			\|J_{12}(t)\|_{1}
			) \| \nabla^{\alpha} g \|_{p} 
			\le C e^{-ct} \| \nabla^{\tilde{\alpha}} g \|_{p},
		\end{split}
	\end{equation*}
	which is the desired estimate \eqref{eq:4.35}.

	Secondly, we prove the estimate \eqref{eq:4.34}. 
	Here we only give the proof for the case $\ell \ge 1$,
	since the proof for $\ell=0$ is slightly easier.
	Now we assume  $\ell \ge 1$.
	Then we have 
	\begin{equation} \label{eq:4.45}
		\begin{split}
			& \partial_{t}^{\ell} \nabla^{\alpha} (J_{0H+}^{(\beta)}(t) g -e^{-\frac{\beta^2}{\nu} t}  \mathcal{R}_{a} \mathcal{R}_{b} \mathcal{F}^{-1}[\chi_{H} \hat{g}]) \\
			& = \sum_{m=0}^{\ell}
			C_{\ell} \partial_{t}^{m} e^{-\frac{\beta^{2}}{\nu}t}
			\mathcal{F}^{-1} \left[
			\partial_{t}^{\ell-m} (e^{(\sigma_{+}^{(\beta)}+\frac{\beta^{2}}{\nu})t}-1)
			\frac{\xi_{a} \xi_{b}}{|\xi|^{2}} \chi_{H} 
			\right] \ast \nabla^{\alpha} g \\
			& = \sum_{m=0}^{\ell-1}
			C_{\ell} \left(
			-\frac{\beta^{2}}{\nu}
			\right)^{m} e^{-\frac{\beta^{2}}{\nu}t}
			\mathcal{F}^{-1} \left[
			\left( \sigma_{+}^{(\beta)}+\frac{\beta^{2}}{\nu} \right)^{\ell-m} e^{(\sigma_{+}^{(\beta)}+\frac{\beta^{2}}{\nu})t}
			\frac{\xi_{a} \xi_{b}}{|\xi|^{2}} \chi_{H} 
			\right] \ast \nabla^{\alpha} g \\
			& + \left(
			-\frac{\beta^{2}}{\nu}
			\right)^{\ell} e^{-\frac{\beta^{2}}{\nu}t}
			\mathcal{F}^{-1} \left[
			(e^{(\sigma_{+}^{(\beta)}+\frac{\beta^{2}}{\nu})t}-1)
			\frac{\xi_{a} \xi_{b}}{|\xi|^{2}} \chi_{H} 
			\right] \ast \nabla^{\alpha} g.
		\end{split}
	\end{equation}
	An easy computation shows 
	\begin{equation} \label{eq:4.46}
		\begin{split}
			\sigma_{+}^{(\beta)}+\frac{\beta^{2}}{\nu}=O(|\xi|^{-2})
		\end{split}
	\end{equation}
	and then
	\begin{equation} \label{eq:4.47}
		\begin{split}
			\partial_{\xi}^{\gamma}
			\left(
			\sigma_{+}^{(\beta)}+\frac{\beta^{2}}{\nu}
			\right)
			=O(|\xi|^{-|\gamma|-2})
		\end{split}
	\end{equation}
	as $|\xi| \to \infty$ for $\gamma \in \mathbb{Z}_{+}^{3}$.
	Therefore, it follows from
	\begin{equation*}
		\begin{split}
			\left|\prod_{k=1}^{m} t \partial_{\xi}^{\tilde{\gamma}_{k}} \left( \sigma_{+}^{(\beta)} + \frac{\beta^{2}}{\nu}\right) \right| 
			\le C \prod_{k=1}^{m} (t |\xi|^{-2-|\tilde{\gamma}_{k}|}) =Ct^{m} |\xi|^{-2m-|\gamma|}
		\end{split}
	\end{equation*}
	for $\xi \in \supp \chi_{H}$ and $\displaystyle\sum_{k=1}^{m} \tilde{\gamma}_{k} =\gamma$ that
	\begin{equation} \label{eq:4.48}
		\begin{split}
			|\partial_{\xi}^{\gamma} e^{t (\sigma_{+}^{(\beta)}+\frac{\beta^{2}}{\nu})}| 
			& =
			\sum_{m=1}^{|\gamma|}
			e^{t (\sigma_{+}^{(\beta)}+\frac{\beta^{2}}{\nu})}
			\left|\prod_{k=1}^{m} t \partial_{\xi}^{\tilde{\gamma}_{k}} \left( \sigma_{+}^{(\beta)} + \frac{\beta^{2}}{\nu}\right) \right| 
			\le C e^{-ct} |\xi|^{-2m-|\gamma|}
		\end{split}
	\end{equation}
	by \eqref{eq:4.47}.
	We can apply the same argument as that for the proof of \eqref{eq:4.34} to estimate 
	the first factor in the right hand side of \eqref{eq:4.45}.
	Namely we have 
	\begin{equation} \label{eq:4.49}
		\begin{split}
			& \left\|
			\sum_{m=0}^{\ell-1}
			C_{\ell} \left(
			-\frac{\beta^{2}}{\nu}
			\right)^{m} e^{-\frac{\beta^{2}}{\nu}t}
			\mathcal{F}^{-1} \left[
			\left( \sigma_{+}^{(\beta)}+\frac{\beta^{2}}{\nu} \right)^{\ell-m} e^{(\sigma_{+}^{(\beta)}+\frac{\beta^{2}}{\nu})t}
			\frac{\xi_{a} \xi_{b}}{|\xi|^{2}} \chi_{H} 
			\right] \ast \nabla^{\alpha} g \right\|_{p} \\
			& \le C
			\sum_{m=0}^{\ell-1}
			\left\|
			\mathcal{F}^{-1} \left[
			\left( \sigma_{+}^{(\beta)}+\frac{\beta^{2}}{\nu} \right)^{\ell-m} e^{\sigma_{+}^{(\beta)}t}
			\frac{\xi_{a} \xi_{b}}{|\xi|^{2}} \chi_{H} 
			\right] \right\|_{1} \|\nabla^{\alpha} g \|_{p} \\
			& \le C e^{-ct} \|\nabla^{\alpha} g \|_{p}.	
		\end{split}
	\end{equation}

	It remains to show the estimate for the term $\displaystyle\mathcal{F}^{-1} \left[
	(e^{(\sigma_{+}^{(\beta)}+\frac{\beta^{2}}{\nu})t}-1)
	\frac{\xi_{a} \xi_{b}}{|\xi|^{2}} \chi_{H} 
	\right]$.
	Based on the fact that 
	\begin{equation*}
		\begin{split}
			& \partial_{\xi}^{\gamma} \left(
			(e^{(\sigma_{+}^{(\beta)}+\frac{\beta^{2}}{\nu})t}-1)
			\frac{\xi_{a} \xi_{b}}{|\xi|^{2}} \chi_{H} 
			\right) \\
			& =
			\chi_{H} 
			(e^{(\sigma_{+}^{(\beta)}+\frac{\beta^{2}}{\nu})t}-1)
			\partial_{\xi}^{\gamma} \left(
			\frac{\xi_{a} \xi_{b}}{|\xi|^{2}} 
			\right) 
			+\chi_{H} 
			\sum_{\substack{ \gamma_{1} +\gamma_{2}=\gamma \\ |\gamma_{2}| \ge 1} }
			\partial_{\xi}^{\gamma_{1}} \left(
			\frac{\xi_{a} \xi_{b}}{|\xi|^{2}} 
			\right)
			\partial_{\xi}^{\gamma_{2}} e^{(\sigma_{+}^{(\beta)}+\frac{\beta^{2}}{\nu})t}  \\
			& +\sum_{\substack{ \gamma_{1} +\gamma_{2}=\gamma \\ |\gamma_{2}| \ge 1} } 
			\partial_{\xi}^{\gamma_{1}} \left(
			(e^{(\sigma_{+}^{(\beta)}+\frac{\beta^{2}}{\nu})t}-1)
			\frac{\xi_{a} \xi_{b}}{|\xi|^{2}} 
			\right) \partial_{\xi}^{\gamma_{2}} \chi_{H},
		\end{split}
	\end{equation*}
	we have the following decomposition
	\begin{equation} \label{eq:4.50}
		\begin{split}
			& e^{-\frac{\beta^{2}}{\nu}t} \mathcal{F}^{-1} \left[
			(e^{(\sigma_{+}^{(\beta)}+\frac{\beta^{2}}{\nu})t}-1)
			\frac{\xi_{a} \xi_{b}}{|\xi|^{2}} \chi_{H} 
			\right] \\
			& = \frac{c_{n}}{x^{\gamma}}e^{-\frac{\beta^{2}}{\nu}t} 
			\lim_{\varepsilon \to 0} 
			\int_{\R^{n}} e^{i x \cdot \xi-\varepsilon |\xi|^{2}}
			\partial_{\xi}^{\gamma} \left(
			(e^{(\sigma_{+}^{(\beta)}+\frac{\beta^{2}}{\nu})t}-1)
			\frac{\xi_{a} \xi_{b}}{|\xi|^{2}} \chi_{H} 
			\right) d \xi \\
			&
			=J_{01}(t,x)+J_{02}(t,x)+J_{03}(t,x), 
		\end{split}
	\end{equation}
	where 
	\begin{equation*}
		\begin{split}
			J_{01}(t,x)& :=  \frac{c_{3}}{x^{\gamma}}e^{-\frac{\beta^{2}}{\nu}t} 
			\lim_{\varepsilon \to 0} 
			\int_{\R^{3}} e^{i x \cdot \xi-\varepsilon |\xi|^{2}}
			(e^{(\sigma_{+}^{(\beta)}+\frac{\beta^{2}}{\nu})t}-1) \chi_{H} 
			\partial_{\xi}^{\gamma} \left(
			\frac{\xi_{a} \xi_{b}}{|\xi|^{2}} 
			\right) d \xi, \\ 
			J_{02}(t,x)& :=  \frac{c_{3}}{x^{\gamma}}e^{-\frac{\beta^{2}}{\nu}t} 
			\lim_{\varepsilon \to 0} 
			\int_{\R^{3}} e^{i x \cdot \xi-\varepsilon |\xi|^{2}}
			C_{\gamma} \chi_{H} 
			\sum_{\substack{ \gamma_{1} +\gamma_{2}=\gamma \\ |\gamma_{2}| \ge 1} }
			\partial_{\xi}^{\gamma_{1}} \left(
			\frac{\xi_{a} \xi_{b}}{|\xi|^{2}} 
			\right)
			\partial_{\xi}^{\gamma_{2}} e^{(\sigma_{+}^{(\beta)}+\frac{\beta^{2}}{\nu})t}
			d \xi, \\ 
			J_{03}(t,x)& :=  \frac{c_{3}}{x^{\gamma}}e^{-\frac{\beta^{2}}{\nu}t} 
			\int_{\R^{3}} e^{i x \cdot \xi}
			\sum_{\substack{ \gamma_{1} +\gamma_{2}=\gamma \\ |\gamma_{2}| \ge 1} } 
			C_{\gamma}
			\partial_{\xi}^{\gamma_{1}} \left(
			(e^{(\sigma_{+}^{(\beta)}+\frac{\beta^{2}}{\nu})t}-1)
			\frac{\xi_{a} \xi_{b}}{|\xi|^{2}} 
			\right) \partial_{\xi}^{\gamma_{2}} \chi_{H}
			d \xi.
		\end{split}
	\end{equation*}

	Now we estimate $J_{01}(t,x)$.
	The mean value theorem gives 
	\begin{equation*}
		\begin{split}
			e^{-\frac{\beta^{2}}{\nu}t} 
			(e^{(\sigma_{+}^{(\beta)}+\frac{\beta^{2}}{\nu})t}-1)
			= \left( \sigma_{+}^{(\beta)}+\frac{\beta^{2}}{\nu} \right) e^{\{ \theta \sigma_{+}^{(\beta)}-(1-\theta)\frac{\beta^{2}}{\nu} \}t}
		\end{split}
	\end{equation*}
	for some $\theta \in [0,1]$. Namely we see 
	\begin{equation*}
		\begin{split}
			e^{-\frac{\beta^{2}}{\nu}t} 
			(e^{(\sigma_{+}^{(\beta)}+\frac{\beta^{2}}{\nu})t}-1)
			= e^{-ct} O(|\xi|^{-2})
		\end{split}
	\end{equation*}
	as $|\xi| \to \infty$ by \eqref{eq:4.46}.
	Combining with 
	\begin{equation} \label{eq:4.51}
		\partial_{\xi}^{\gamma} \left(
		\frac{\xi_{a} \xi_{b}}{|\xi|^{2}} 
		\right) =O(|\xi|^{-|\gamma|})
	\end{equation}
	as $|\xi| \to \infty$ for $\gamma \in \mathbb{Z}_{+}^{3}$, 
	we obtain 
	\begin{equation*}
		\begin{split}
			|J_{01}(t,x)|  \le  \frac{C}{\left| x^{\gamma} \right|} e^{-ct} \left|
			\lim_{\varepsilon \to 0} 
			\int_{\R^{3}} e^{i x \cdot \xi-\varepsilon |\xi|^{2}} |\xi|^{-|\gamma|-2}
			\chi_{H}  d \xi 
			\right|
			=
			\begin{cases}
				& e^{-ct}O(|x|^{-2})  \quad \text{for}\ |x| \le 1, \\
				& e^{-ct} O(|x|^{-4})  \quad \text{for}\ |x| \ge 1,
			\end{cases}
		\end{split}
	\end{equation*}
	where we choose $\gamma \in \mathbb{Z}_{+}^{3}$ satisfying $|\gamma|=2$ for $|x|\le 1$ and $|\gamma|=4$ for $|x|\ge 1$.
	Therefore we have 
	\begin{equation} \label{eq:4.52}
		\begin{split}
			\|J_{01}(t)\|_{1}  \le  e^{-ct}.
		\end{split}
	\end{equation}
	Similarly we see that 
	\begin{equation*}
		\begin{split}
			J_{02}(t,x) =  \frac{c_{3}}{x^{\gamma}}
			\lim_{\varepsilon \to 0} 
			\int_{\R^{3}} e^{i x \cdot \xi-\varepsilon |\xi|^{2}}
			\chi_{H}  e^{\sigma_{+}^{(\beta)}t}
			\sum_{\substack{ \gamma_{1} +\gamma_{2}=\gamma \\ |\gamma_{2}| \ge 1} }
			\sum_{m=1}^{|\gamma_{2}|} t^{m}O(|\xi|^{-2m-|\gamma|})
			d \xi
		\end{split}
	\end{equation*}
	by \eqref{eq:4.48} and \eqref{eq:4.51}.
	Thus,
	with the choice of $\gamma \in \mathbb{Z}^{3}$ as $|\gamma|=2$ for $|x|\le 1$ and $|\gamma|=4$ for $|x|\ge 1$,
	we have 
	\begin{equation*}
		\begin{split}
			|J_{02}(t,x)| & \le \frac{C}{|x^{\gamma}|} e^{-ct} 
			\left| 
			\lim_{\varepsilon \to 0} 
			\int_{\R^{3}} e^{i x \cdot \xi-\varepsilon |\xi|^{2}} |\xi|^{-|\gamma|-2}
			\chi_{H}  d \xi
			\right|
			=
			\begin{cases}
				& e^{-ct}O(|x|^{-2})  \quad \text{for}\ |x| \le 1, \\
				& e^{-ct} O(|x|^{-4})  \quad \text{for}\ |x| \ge 1,
			\end{cases}
		\end{split}
	\end{equation*}
	which implies 
	\begin{equation} \label{eq:4.53}
		\begin{split}
			\|J_{02}(t)\|_{1}  \le  e^{-ct}.
		\end{split}
	\end{equation}
	Now, we note that $\displaystyle\cup_{1\le |\gamma_{2}| \le |\gamma|} \supp \partial_{\xi}^{\gamma_{2}} \chi_{H}$ is compact in $\R^{3}$.
	Then it is easy to see that 
	\begin{equation*}
		\begin{split}
			|J_{03}(t,x)| & \le \frac{C}{|x^{\gamma}|} e^{-ct} 
			\int_{\R^{3}}
			\left|
			\sum_{\substack{ \gamma_{1} +\gamma_{2}=\gamma \\ |\gamma_{2}| \ge 1} } 
			C_{\gamma}
			\partial_{\xi}^{\gamma_{1}} \left(
			(e^{(\sigma_{+}^{(\beta)}+\frac{\beta^{2}}{\nu})t}-1)
			\frac{\xi_{a} \xi_{b}}{|\xi|^{2}} 
			\right) \partial_{\xi}^{\gamma_{2}} \chi_{H}
			\right|
			d \xi \\
			& =e^{-ct}O(|x|^{-|\gamma|}) \quad \text{as} \quad |x| \to \infty
		\end{split}
	\end{equation*}
	for all $\gamma \in \mathbb{Z}^{3}$, 
	which implies 
	\begin{equation} \label{eq:4.54}
		\begin{split}
			\|J_{03}(t)\|_{1}  \le  e^{-ct}.
		\end{split}
	\end{equation}
	Summing up \eqref{eq:4.45}, \eqref{eq:4.49} and \eqref{eq:4.52}-\eqref{eq:4.54}, we obtain the estimate \eqref{eq:4.34}.
	%
	%
	
	In the third, we prove the estimate \eqref{eq:4.36}.
	For the simplicity of the notation, we denote 
	\begin{equation*}
		\begin{split}
			\mathcal{J}_{H}^{(\ell,\alpha)}(t,\xi):=  \frac{(\sigma_{-}^{(\beta)})^{\ell} (i \xi)^{\alpha} \sigma_{+}^{(\beta)} e^{\sigma_{-}^{(\beta)} t} }{\sigma_{+}^{(\beta)}-\sigma_{-}^{(\beta)}}
			\chi_{H} \frac{\xi_{a}\xi_{b}}{|\xi|^{2}}.
		\end{split}
	\end{equation*}
	It is easy to see 
	\begin{equation*}
		\begin{split}
			|\mathcal{J}_{H}^{(\ell,\alpha)}(t,\xi)| \le C e^{-ct} e^{-c|\xi|^{2} t} |\xi|^{2 (\ell-1)+|\alpha|} \chi_{H}
		\end{split}
	\end{equation*}
	and then 
	\begin{equation} \label{eq:4.55}
		\begin{split}
			\| \mathcal{J}_{H}^{(\ell,\alpha)}(t) \|_{p} \le C e^{-ct} t^{-\frac{3}{2p} -(\ell-1) -\frac{|\alpha|}{2}}
		\end{split}
	\end{equation}
	for $1 \le p\le 2$.
	On the other hand, noting that 
	\begin{equation*}
		\begin{split}
			\left|\prod_{k=1}^{m} \partial_{\xi}^{\tilde{\gamma}_{k}} (t \sigma_{-}^{(\beta)}) \right| 
			\le C \prod_{k=1}^{m} (t |\xi|^{2-|\tilde{\gamma}_{k}|}) =Ct^{m} |\xi|^{2m-|\gamma_{2}|}
		\end{split}
	\end{equation*}
	for $\xi \in \supp \chi_{H}$ and $\sum_{k=1}^{m} \tilde{\gamma}_{k} =\gamma_{2} $,
	we have 
	\begin{equation*}
		\begin{split}
			|\partial_{\xi}^{\gamma} e^{t \sigma_{-}^{(\beta)}}| 
			& =
			\sum_{m=1}^{|\gamma_{2}|}
			e^{t \sigma_{-}^{(\beta)}} 
			\left| \sum_{\substack{ \sum_{k=1}^{m} \tilde{\gamma}_{k} =\gamma_{2} \\ \tilde{\gamma}_{k} \ge 1} }
			\prod_{k=1}^{m} \partial_{\xi}^{\tilde{\gamma}_{k}} (t \sigma_{-}^{(\beta)}) \right| \\
			& 
			\le C e^{-ct} e^{-c|\xi|^{2} t} 
			\sum_{m=1}^{|\gamma_{2}|}
			t^{m} |\xi|^{2m-|\gamma_{2}|} \le C e^{-ct} e^{-c|\xi|^{2} t} |\xi|^{-|\gamma_{2}|}.
		\end{split}
	\end{equation*}
	Thus a direct calculation shows  
	\begin{equation*}
		\begin{split}
			& |\partial_{\xi}^{\gamma}\mathcal{J}_{H}^{(\ell,\alpha)}(t,\xi) | \\
			& \le \sum_{\substack{ \gamma_{1} +\gamma_{2}=\gamma \\ |\gamma_{2}| \ge 1} }
			C_{\gamma_{1},\gamma_{2}} \left| 
			\partial_{\xi}^{\gamma_{1}}
			\left(
			 \frac{(\sigma_{-}^{(\beta)})^{\ell} (i \xi)^{\alpha} \sigma_{+}^{(\beta)} e^{\sigma_{-}^{(\beta)} t} }{\sigma_{+}^{(\beta)}-\sigma_{-}^{(\beta)}}
			\frac{\xi_{a}\xi_{b}}{|\xi|^{2}}
			\right)
			\partial_{\xi}^{\gamma_{2}}
			\chi_{H} \right| \\
			& +
			\chi_{H}
			\sum_{\gamma_{1} +\gamma_{2}=\gamma }
			C_{\gamma_{1},\gamma_{2}} 
			\left| 
			\partial_{\xi}^{\gamma_{1}}
			\left(
			 \frac{(\sigma_{-}^{(\beta)})^{\ell} (i \xi)^{\alpha} \sigma_{+}^{(\beta)}  }{\sigma_{+}^{(\beta)}-\sigma_{-}^{(\beta)}}
			\frac{\xi_{a}\xi_{b}}{|\xi|^{2}}
			\right)
			\partial_{\xi}^{\gamma_{2}}
			e^{\sigma_{\pm}^{(\beta)} t}
			\right| \\
			& \le C e^{-ct} \sum_{0 <|\gamma_{2}|<|\gamma| } |\partial_{\xi}^{\gamma_{2}} \chi_{H}| + C e^{-ct} e^{-c|\xi|^{2} t} |\xi|^{2 (\ell-1)+|\alpha|-|\gamma|},
		\end{split}
	\end{equation*}
	and we obtain the estimate 
	\begin{equation} \label{eq:4.56}
		\begin{split}
			\| 
			\partial_{\xi}^{\gamma} \mathcal{J}_{0H-}^{(\ell,\alpha)}(t)
			\|_{2} \le C e^{-ct} t^{-\frac{3}{4} -(\ell-1) -\frac{|\alpha|-|\gamma|}{2}}.   
		\end{split}
	\end{equation}
	Choosing $|\gamma|=2$ and applying the estimate \eqref{eq:2.36} with \eqref{eq:4.55} and \eqref{eq:4.56}, 
	we see that
	\begin{equation*}
		\begin{split}
			& \| \partial_{t}^{\ell} \nabla^{\alpha} J_{H}^{(\beta)}(t) g \|_{p} \\
			& \le 
			C \| \mathcal{F}^{-1} [\mathcal{J}_{H}^{(\ell-\frac{\tilde{\ell}}{2},\alpha-\tilde{\alpha})}(t,\xi) ] \|_{1} \|\nabla^{\tilde{\alpha}+\tilde{\ell}} g \|_{p} \\
			& \le 
			C \| \mathcal{J}_{H}^{(\ell-\frac{\tilde{\ell}}{2},\alpha-\tilde{\alpha})}(t) \|_{2}^{\frac{1}{4}} 
			\left( \sum_{|\gamma|=2}
			\| \partial_{\xi}^{\gamma} \mathcal{J}_{H}^{(\ell-\frac{\tilde{\ell}}{2},\alpha-\tilde{\alpha})}(t) \|_{2}
			\right)^{\frac{3}{4}}
			\|\nabla^{\tilde{\alpha}+\tilde{\ell}} g \|_{p} \\
			& \le 
			C e^{-ct} t^{-\frac{\alpha-\tilde{\alpha}}{2}-(\ell-\frac{\tilde{\ell}}{2})+1}\| \nabla^{\tilde{\alpha}+\tilde{\ell}} g \|_{p}.
		\end{split}
	\end{equation*}
	We also have 
	\begin{equation*}
		\begin{split}
			\| \partial_{t}^{\ell} \nabla^{\alpha} J_{H}^{(\beta)}(t) g \|_{\infty} 
			& \le 
			C \| \mathcal{J}_{H}^{(\ell-\frac{\tilde{\ell}}{2},\alpha-\tilde{\alpha})}(t,\xi)  \|_{1} \|\nabla^{\tilde{\alpha}+\tilde{\ell}} g \|_{1} \\
			& \le 
			C e^{-ct} t^{-\frac{3}{2}-\frac{\alpha-\tilde{\alpha}}{2}-(\ell-\frac{\tilde{\ell}}{2})+1}\| \nabla^{\tilde{\alpha}+\tilde{\ell}} g \|_{1}.
		\end{split}
	\end{equation*}
	By the interpolation, we conclude the desired estimate \eqref{eq:4.36}.
	The estimate \eqref{eq:4.37} is shown in a similar way to the proof of \eqref{eq:4.36}.
	We complete the proof of Lemma \ref{Lem:4.9}.
\end{proof}
Recalling the identities \eqref{eq:4.33}, we can formulate the estimates for the high frequency parts of the fundamental solutions to \eqref{eq:1.1}.
\begin{cor} \label{cor:4.11}
	Let $\alpha \ge \tilde{\alpha} \ge 0$, $\ell \ge 2 \tilde{\ell} \ge 0$ and $t>0$. 
	Then, the following estimates hold:
	\begin{equation}
		\begin{split}
			& \sum_{k=M,H} ( \| \partial_{t}^{\ell} \nabla^{\alpha} K_{0k}^{(\beta)}(t)  g  \|_{p} 
			+\| \partial_{t}^{\ell} \nabla^{\alpha} K_{0k}^{(\beta)}(t) \mathcal{R}_{a}  \mathcal{R}_{b} g  \|_{p} )  \\
			& \le C e^{-ct} (\| \nabla^{\alpha_{1}} g \|_{p} +t^{-\frac{3}{2}(\frac{1}{q}-\frac{1}{p})-\frac{\alpha-\tilde{\alpha}}{2}-(\ell-\frac{\tilde{\ell}}{2})+1}
			\| \nabla^{\tilde{\alpha}+\tilde{\ell}} g \|_{q}), \quad \alpha_{1} \ge \alpha, \label{eq:4.57} 
		\end{split}
	\end{equation}
	\begin{equation}
		\begin{split}
			& \sum_{k=M,H} ( \| \partial_{t}^{\ell} \nabla^{\alpha} K_{1k}^{(\beta)}(t)  g  \|_{p} 
			+\| \partial_{t}^{\ell} \nabla^{\alpha} K_{1k}^{(\beta)}(t) \mathcal{R}_{a}  \mathcal{R}_{b} g  \|_{p} ) \\
			& \le C e^{-ct} (\| \nabla^{\alpha_{1}} g \|_{p} 
			+t^{-\frac{3}{2}(\frac{1}{q}-\frac{1}{p})-\frac{\alpha-\tilde{\alpha}}{2}-(\ell-\frac{\tilde{\ell}}{2})+1}\| \nabla^{\tilde{\alpha}+\tilde{\ell}} g \|_{q}),
			\quad \alpha_{1} \ge (\alpha-2)_{+}
			\label{eq:4.58}
		\end{split}
	\end{equation}
	for $1 < p< \infty$ and $1 \le q \le p$ and 
	\begin{equation}
		\begin{split}
			& \sum_{k=M,H} ( \| \partial_{t}^{\ell} \nabla^{\alpha} K_{0k}^{(\beta)}(t)  g  \|_{p} 
			+\| \partial_{t}^{\ell} \nabla^{\alpha} K_{0k}^{(\beta)}(t) \mathcal{R}_{a}  \mathcal{R}_{b} g  \|_{p} )  \\
			& \le C e^{-ct} (\| \nabla^{\alpha+2} g \|_{p} +t^{-\frac{3}{2}(\frac{1}{q}-\frac{1}{p})-\frac{\alpha-\tilde{\alpha}}{2}-(\ell-\frac{\tilde{\ell}}{2})+1}
			\| \nabla^{\tilde{\alpha}+\tilde{\ell}} g \|_{q}), \label{eq:4.59} 
		\end{split}
	\end{equation}
	\begin{equation}
		\begin{split}
			& \sum_{k=M,H} ( \| \partial_{t}^{\ell} \nabla^{\alpha} K_{1k}^{(\beta)}(t)  g  \|_{p} 
			+\| \partial_{t}^{\ell} \nabla^{\alpha} K_{1k}^{(\beta)}(t) \mathcal{R}_{a}  \mathcal{R}_{b} g  \|_{p} )  \\
			& \le C e^{-ct} (\| \nabla^{\alpha} g \|_{p} 
			+t^{-\frac{3}{2}(\frac{1}{q}-\frac{1}{p})-\frac{\alpha-\tilde{\alpha}}{2}-(\ell-\frac{\tilde{\ell}}{2})+1}\| \nabla^{\tilde{\alpha}+\tilde{\ell}} g \|_{q})
			\label{eq:4.60}
		\end{split}
	\end{equation}
	for $1 \le q \le p \le \infty$. 
\end{cor}
\begin{proof}
	We first show the estimate \eqref{eq:4.57}.
	Noting that 
	\begin{equation*}
		\begin{split}
			\partial_{t}^{\ell} \nabla^{\alpha} K_{0k}^{(\beta)}(t)\mathcal{R}_{a} \mathcal{R}_{b}g =
			\mathcal{R}_{a} \mathcal{R}_{b} \partial_{t}^{\ell} \nabla^{\alpha} K_{0k}^{(\beta)}(t)g
		\end{split}
	\end{equation*}
	for $k=M,H$,
	we apply the estimates \eqref{eq:2.34} and \eqref{eq:2.27} to have \eqref{eq:4.57}.
	We can obtain the estimate \eqref{eq:4.58} by a similar way. 
	
	We next prove the estimate \eqref{eq:4.59}.
	Noting the decomposition \eqref{eq:4.33}, we have 
	\begin{equation*}
		\begin{split}
			 \partial_{t}^{\ell} \nabla^{\alpha} K_{0H}^{(\beta)}(t)\mathcal{R}_{a} \mathcal{R}_{b}g
		 = \partial_{t}^{\ell} \nabla^{\alpha} J_{0H+}^{(\beta)}(t) g 
			-\partial_{t}^{\ell+1} \nabla^{\alpha} J_{1H+}^{(\beta)}(t) g 
			+\partial_{t}^{\ell} \nabla^{\alpha} J_{H}^{(\beta)}(t) g.
		\end{split}
	\end{equation*}
	Then, we apply the estimates \eqref{eq:4.34}-\eqref{eq:4.36} to have 
	\begin{equation} \label{eq:4.61}
		\begin{split}
			\| \partial_{t}^{\ell} \nabla^{\alpha} K_{0H}^{(\beta)}(t) \mathcal{R}_{a}  \mathcal{R}_{b}g \|_{p} 
			& \le C e^{-ct} \| \nabla^{\alpha} g \|_{p} + Ce^{-\frac{\beta^2}{\nu} t} \| \nabla^{\alpha}   \mathcal{R}_{a} \mathcal{R}_{b} \mathcal{F}^{-1}[\chi_{H} \hat{g}] \|_{p} \\
			& +C e^{-ct} t^{-\frac{3}{2}(\frac{1}{q}-\frac{1}{p})-\frac{\alpha-\tilde{\alpha}}{2}-(\ell-\frac{\tilde{\ell}}{2})+1}
			\| \nabla^{\tilde{\alpha}+\tilde{\ell}} g \|_{p}.
		\end{split}
	\end{equation}
	On the other hand, using the fact that $\chi_{H}=1-\chi_{L}-\chi_{M}$, 
	we have 
	\begin{equation} \label{eq:4.62}
		\begin{split}
			e^{-\frac{\beta^2}{\nu} t}  \| \nabla^{\alpha}  \mathcal{R}_{a} \mathcal{R}_{b} \mathcal{F}^{-1}[\chi_{H} \hat{g}] \|_{p} 
			& \le 
			C e^{-\frac{\beta^2}{\nu} t} \| \mathcal{F}^{-1}[(\chi_{H}|\xi|^{-2}])( \xi^{\alpha} \xi_{a} \xi_{b} \hat{g})] \|_{p} \\
			& \le C e^{-ct} \| \nabla^{\alpha+2} g \|_{p} 
		\end{split}
	\end{equation}
	for $1 \le p \le  \infty$, since the Fourier multiplier $\chi_{H}|\xi|^{-2}$ is a bounded operator on $L^{p}$ for $1 \le p \le \infty$.
	It is also well-known that the middle frequency part is smooth enough and decays exponentially (cf. \cite{K-S}). 
	Namely we have 
	\begin{equation} \label{eq:4.63}
		\begin{split}
			& \| \partial_{t}^{\ell} \nabla^{\alpha} K_{0M}^{(\beta)}(t) \mathcal{R}_{a}  \mathcal{R}_{b}g \|_{p}
			+\| \partial_{t}^{\ell} \nabla^{\alpha} K_{1M}^{(\beta)}(t) \mathcal{R}_{a}  \mathcal{R}_{b}g \|_{p} \\
			& + \| \partial_{t}^{\ell} \nabla^{\alpha} K_{0M}^{(\beta)}(t) g \|_{p}
			+\| \partial_{t}^{\ell} \nabla^{\alpha} K_{1M}^{(\beta)}(t) g \|_{p}
			\le C e^{-ct} \| \nabla^{\tilde{\alpha}} g \|_{q} 
		\end{split}
	\end{equation}
	for $1 \le q \le p \le \infty$ and $\alpha, \tilde{\alpha} \ge 0$.
	Therefore we can conclude the estimate \eqref{eq:4.59} by the combination of \eqref{eq:4.61}-\eqref{eq:4.63}. 
	The estimate \eqref{eq:4.60} is easily proved by \eqref{eq:4.33}, \eqref{eq:4.35} and \eqref{eq:4.57}.
	We complete the proof of Corollary \ref{cor:4.11}. 
\end{proof}
\section{Proof of main results ($1< p < \infty$)} 
For the proof of main results, we firstly reformulate \eqref{eq:1.1} into the integral equation.
\begin{prop}[\cite{K-T}, \cite{Takeda}] \label{prop:5.1}
	Let $u$ be a solution of \eqref{eq:1.1}. Then it holds that
	\begin{equation*} 
		\begin{split}
			\hat{u}(t,\xi) & =\mathcal{K}_{0}^{(\sqrt{\lambda+2 \mu})}(t,\xi) \mathcal{P} \hat{f}_{0}(\xi) 
			+ \mathcal{K}_{0}^{(\sqrt{\mu})}(t,\xi) (\mathcal{I}_{3}-\mathcal{P}) \hat{f}_{0}(\xi) \\
			& +\mathcal{K}_{1}^{(\sqrt{\lambda+2 \mu})}(t,\xi) \mathcal{P} \hat{f}_{1}(\xi) 
			+ \mathcal{K}_{1}^{(\sqrt{\mu})}(t,\xi) (\mathcal{I}_{3}-\mathcal{P}) \hat{f}_{1}(\xi) \\
			& + \int_{0}^{t} \left\{ 
			\mathcal{K}_{1}^{(\sqrt{\lambda+2 \mu})}(t-\tau,\xi) \mathcal{P} 
			+ \mathcal{K}_{1}^{(\sqrt{\mu})}(t-\tau,\xi) (\mathcal{I}_{3}-\mathcal{P}) 
			\right\} \hat{F}(u)(\tau, \xi) d \tau.
		\end{split}
	\end{equation*}
\end{prop}
From Proposition \ref{prop:5.1} we have the expression of the solution $u(t)$ by 
\begin{equation} \label{eq:5.1}
	\begin{split}
		u(t) & =u_{lin}(t)+u_{N}[u](t),
	\end{split}
\end{equation}
where
\begin{equation*} 
	\begin{split}
		u_{lin}(t) & :=(K_{0}^{(\sqrt{\lambda+2 \mu})}(t)-  K_{0}^{(\sqrt{\mu})}(t)) \mathcal{F}^{-1} [\mathcal{P} \hat{f}_{0}] 
		+ K_{0}^{(\sqrt{\mu})}(t)f_{0} \\
		& +(K_{1}^{(\sqrt{\lambda+2 \mu})}(t)-  K_{1}^{(\sqrt{\mu})}(t)) \mathcal{F}^{-1} [\mathcal{P} \hat{f}_{1}] 
		+ K_{1}^{(\sqrt{\mu})}(t)f_{1},
	\end{split}
\end{equation*}
and 
\begin{equation*} 
	\begin{split}
		u_{N}(t) & := \int_{0}^{t}
		(K_{1}^{(\sqrt{\lambda+2 \mu})}(t-\tau)-K_{1}^{(\sqrt{\mu})}(t-\tau)) \mathcal{F}^{-1} [\mathcal{P} \hat{F}(u)(\tau) ] d \tau \\
		& + \int_{0}^{t} K_{1}^{(\sqrt{\mu})}(t-\tau) F(u)(\tau, \xi) d \tau.
	\end{split}
\end{equation*}
We also recall the estimates for the nonlinear term, which are obtained in \cite{K-T} under the assumption on Proposition \ref{prop:2.1} for $F(u)=\nabla u \nabla^{2} u$:
\begin{align} 
	& \| F(u)(t) \|_{1} \le C(1+t)^{-2}, \label{eq:5.2} \\
	& \| F(u)(t) \|_{p}  \le C(1+t)^{-\frac{7}{2}+\frac{3}{2p}}, \quad 2 \le p \le 6, \label{eq:5.3} \\
	& \| \nabla F(u)(t) \|_{2} \le C(1+t)^{-\frac{13}{4}}. \label{eq:5.4}
\end{align}
Here we note that we can obtain the estimates \eqref{eq:5.2}-\eqref{eq:5.4} under the assumption of Proposition \ref{prop:2.2}, which will be used in the proof of Theorems \ref{thm:3.2}, \ref{thm:3.4} and \ref{thm:3.6}.
\subsection{Proof of Theorem \ref{thm:3.1}}
We firstly collect the estimates for the linear part, 
which are easy consequences of the estimates \eqref{eq:2.23}, \eqref{eq:2.24}, \eqref{eq:2.27}, \eqref{eq:2.28}, \eqref{eq:4.1}, \eqref{eq:4.2}, \eqref{eq:4.57} and \eqref{eq:4.58} with the notation 
$$\| f_{0}, f_{1}\|_{Y_{p}}:= \| f_{0} \|_{\dot{H}^{3} \cap \dot{W}^{1,1} \cap \dot{W}^{3,p}} + \| f_{1} \|_{H^{1} \cap L^{1} \cap \dot{W}^{1,p}}.$$ 
Namely, we have 
\begin{equation} \label{eq:5.5}
	\begin{split}
		& \sup_{t \ge 0} \biggl\{  \sum_{\alpha=1,3} (1+t)^{\frac{5}{2}(1-\frac{1}{p})-1+\frac{\alpha}{2}} \| \nabla^{\alpha} u_{lin}(t) \|_{p} 
		+
		\sum_{\alpha=0,1} (1+t)^{\frac{5}{2}(1-\frac{1}{p})-\frac{1}{2}+\frac{\alpha}{2}} \| \nabla^{\alpha} \partial_{t} u_{lin}(t) \|_{p} \\
		& \qquad +  (1+t)^{\frac{5}{2}(1-\frac{1}{p})} t^{\frac{1}{2}} \| \nabla^{2} \partial_{t} u_{lin}(t) \|_{p}  +  (1+t)^{\frac{5}{2}(1-\frac{1}{p})-\frac{1}{2}} t^{\frac{1}{2}} \| \partial_{t}^{2} u_{lin}(t) \|_{p} 
		\biggr\} \\
		& \le C \| f_{0}, f_{1}\|_{Y_{p}}.
	\end{split}
\end{equation}

Now we turn to the proof of the estimates in Theorem \ref{thm:3.1}.
Let $F(u)=\nabla u \nabla^{2} u$.
We prove the estimate \eqref{eq:3.1} with $\alpha=3$. 
At first, 
we decompose the nonlinear term as follows:
\begin{equation} \label{eq:5.6} 
	\begin{split}
		& \left\| \nabla^{3}  \int_{0}^{t} K^{(\beta)}_{1}(t-\tau) \mathcal{R}_{a} \mathcal{R}_{b} F_{j}(u)(\tau) d \tau \right\|_{p} 
		+ \left\| \nabla^{3}  \int_{0}^{t} K^{(\beta)}_{1}(t-\tau)  F_{j}(u)(\tau) d \tau \right\|_{p} \\
		& \le \sum_{k=L,M,H}  \int_{0}^{t} \left\| \nabla^{3} K^{(\beta)}_{1k}(t-\tau)  F_{j}(u)(\tau) \right\|_{p}  d \tau,
	\end{split}
\end{equation}
since $1<p< \infty$ and \eqref{eq:2.34}.
When $1<p<2$, 
we see that

\begin{equation} \label{eq:5.7}
	\begin{split}
		\| F(u)(t)  \|_{p} \le \| F(u) \|_{1}^{\frac{2}{p}-1} \|  F(u) \|_{2}^{2-\frac{2}{p}}  \le C(1+t)^{-\frac{5}{2}(1-\frac{1}{p})-\frac{1}{p}-1}
	\end{split}
\end{equation}
by the H\"older inequality, where $\frac{2}{2-p} >2$ for $1 < p< 2$, \eqref{eq:5.2} and \eqref{eq:5.3}.
Noting that 
\begin{equation} \label{eq:5.8}
	\begin{split}
\| \nabla F(u)(t)  \|_{1} \le C \| \nabla^{2} u \|_{2}^{2} + C \| \nabla u \|_{\infty} \| \nabla^{3} u \|_{2} \le C(1+t)^{-\frac{5}{2}}
\end{split}
\end{equation}
by \eqref{eq:2.2} and \eqref{eq:2.3},
we apply the same argument with \eqref{eq:5.4} to have
\begin{equation} \label{eq:5.9}
	\begin{split}
		\| \nabla F(u)(t)  \|_{p} \le \| \nabla F(u) \|_{1}^{\frac{2}{p}-1} \| \nabla F(u) \|_{2}^{2-\frac{2}{p}}  \le C(1+t)^{-\frac{5}{2}(1-\frac{1}{p})-\frac{1}{p}-\frac{3}{2}}
	\end{split}
\end{equation}
for $1<p<2$.
Then when $1<p<2$,
we employ the estimates \eqref{eq:2.24}, \eqref{eq:5.2} and \eqref{eq:5.7} to obtain
\begin{equation} \label{eq:5.10}
	\begin{split}
		& \int_{0}^{t} \left\| \nabla^{3} K^{(\beta)}_{1L}(t-\tau)  F_{j}(u)(\tau) \right\|_{p}  d \tau \\
		& \le C  \int_{0}^{\frac{t}{2}} (1+t-\tau)^{-\frac{5}{2}(1-\frac{1}{p})-\frac{1}{2}} 
		\| F_{j}(u)(\tau) \|_{1} d \tau 
		+C  \int_{\frac{t}{2}}^{t} (1+t-\tau)^{-\frac{1}{2}} 
		\| F_{j}(u)(\tau) \|_{p} d \tau   \\
		& \le C  \int_{0}^{\frac{t}{2}} (1+t-\tau)^{-\frac{5}{2}(1-\frac{1}{p})-\frac{1}{2}} 
		(1+\tau)^{-2} d \tau \\
		&
		+C  \int_{\frac{t}{2}}^{t} (1+t-\tau)^{-\frac{1}{2}} 
		(1+\tau)^{-\frac{5}{2}(1-\frac{1}{p})-\frac{1}{p}-1} d \tau \\
		& \le C(1+t)^{-\frac{5}{2}(1-\frac{1}{p})-\frac{1}{2}},
	\end{split}
\end{equation}
where we used the fact that 
$
-2 < -\frac{7}{4} <-\frac{5}{2}(1-\frac{1}{p})-\frac{1}{2}
$. 
When $p \ge 2$, we apply \eqref{eq:2.24}, \eqref{eq:5.1} and \eqref{eq:5.3} again to see that
\begin{equation} \label{eq:5.11}
	\begin{split}
		& \int_{0}^{t} \left\| \nabla^{3} K^{(\beta)}_{1L}(t-\tau)  F_{j}(u)(\tau) \right\|_{p}  d \tau \\
		& \le C  \int_{0}^{\frac{t}{2}} (1+t-\tau)^{-\frac{5}{2}(1-\frac{1}{p})-\frac{1}{2}} 
		\| F_{j}(u)(\tau) \|_{1} d \tau  \\
		&
		+C \int_{\frac{t}{2}}^{t} (1+t-\tau)^{-\frac{5}{2}(\frac{1}{2}-\frac{1}{p})} 
		\| \nabla F_{j}(u)(\tau) \|_{2} d \tau   \\
		& \le C  \int_{0}^{\frac{t}{2}} (1+t-\tau)^{-\frac{5}{2}(1-\frac{1}{p})-\frac{1}{2}} 
		(1+\tau)^{-2} d \tau + 
		C  \int_{\frac{t}{2}}^{t} (1+t-\tau)^{-\frac{5}{2}(\frac{1}{2}-\frac{1}{p})} 
		(1+\tau)^{-\frac{13}{4}} d \tau
		\\
		& \le C(1+t)^{-\frac{5}{2}(1-\frac{1}{p})-\frac{1}{2}}.
	\end{split}
\end{equation}
We thus obtain 
\begin{equation} \label{eq:5.12}
	\begin{split}
		\int_{0}^{t} \left\| \nabla^{3} K^{(\beta)}_{1L}(t-\tau)  F_{j}(u)(\tau) \right\|_{p}  d \tau 
		\le C(1+t)^{-\frac{5}{2}(1-\frac{1}{p})-\frac{1}{2}}
	\end{split}
\end{equation}
for $1<p<\infty$ by \eqref{eq:5.10} and \eqref{eq:5.11}.

For the middle and high frequency parts, 
using the estimates \eqref{eq:4.58} and 
\begin{equation*} 
	\begin{split}
		\| \nabla F(u)(t) \|_{p} & \le C \| \nabla^{2} u(t)\|_{2p}^{2} + C \| \nabla u \|_{\infty} \| \nabla^{3} u \|_{p} \\
		& \le C \| \nabla u \|_{\infty} \| \nabla^{3} u \|_{p} \le C(1+t)^{-2} \| \nabla^{3} u(t) \|_{p}
	\end{split}
\end{equation*}
by \eqref{eq:2.38} and \eqref{eq:2.3} with $\alpha=1$,
we see that
\begin{equation} \label{eq:5.13}
	\begin{split}
		& \sum_{k=M,H}
		\int_{0}^{t} \left\| \nabla^{3} K^{(\beta)}_{1k}(t-\tau)  F_{j}(u)(\tau) \right\|_{p}  d \tau \\
		&\le
		C  \int_{0}^{t} e^{-c(t-\tau)} \| \nabla F(u)(\tau)\|_{p} d \tau  
		 \le C  \int_{0}^{t} e^{-c(t-\tau)}(1+\tau)^{-2} \| \nabla^{3} u(\tau) \|_{p} d \tau.
	\end{split}
\end{equation}
Taking $\nabla^{3}$ to the both sides of \eqref{eq:5.1} and 
combining the estimates \eqref{eq:5.5}, \eqref{eq:5.12} and \eqref{eq:5.13}, 
we arrive at the estimate 
\begin{equation} \label{eq:5.14}
	\begin{split}
		\| \nabla^{3} u(t) \|_{p} \le C_{0}(1+t)^{-\frac{5}{2}(1-\frac{1}{p})-\frac{1}{2}}
		+  C_{1} \int_{0}^{t} e^{-c_{2}(t-\tau)}(1+\tau)^{-2} \| \nabla^{3} u(\tau) \|_{p} d \tau,
	\end{split}
\end{equation}
where $C_{0}$, $C_{1}$ and $c_{2}$ are positive constants.
Applying the Gronwall type argument to \eqref{eq:5.14}, 
we can obtain \eqref{eq:3.1} with $\alpha=3$.
Indeed, denoting $A(t):= e^{c_{2} t} \| \nabla^{3} u(t) \|_{p}$ and
$$
F(t):=C_{1} \int_{0}^{t} e^{c_{2}\tau}(1+\tau)^{-2} \| \nabla^{3} u(\tau) \|_{p} d \tau, 
$$
we can rephrase \eqref{eq:5.14} as
\begin{equation} \label{eq:5.15}
\begin{split}
	A(t) \le C_{0}(1+t)^{-\frac{5}{2}(1-\frac{1}{p})-\frac{1}{2}} e^{c_{2}t} +F(t).
\end{split}
\end{equation}
Noting that  
\begin{equation*}
	\begin{split}
		F'(t)= C_{1}e^{c_{2} t} (1+t)^{-2} \| \nabla^{3} u(t) \|_{p} \le C_{1}(1+t)^{-2}
		(
		C_{0}(1+t)^{-\frac{5}{2}(1-\frac{1}{p})-\frac{1}{2}} e^{c_{2}t} +F(t)
		)
	\end{split}
\end{equation*}
by \eqref{eq:5.15},
we have
\begin{equation*}
	\begin{split}
		\frac{d}{dt} 
		\left(
		F(t) e^{\frac{C_{1}}{1+t}}
		\right)
		\le C e^{c_{2} t+\frac{C_{1}}{1+t}}(1+t)^{-\frac{5}{2}(1-\frac{1}{p})-\frac{5}{2}},
	\end{split}
\end{equation*}
and thus, 
\begin{equation*}
	\begin{split}
		F(t) e^{\frac{C_{1}}{1+t}}
		& \le C \int_{0}^{t}e^{c_{2} \tau +\frac{C_{1}}{1+\tau}}(1+\tau)^{-\frac{5}{2}(1-\frac{1}{p})-\frac{5}{2}} d \tau \\
		& \le C e^{c_{2} \frac{t}{2}} \int_{0}^{\frac{t}{2}}
		(1+\tau)^{-\frac{5}{2}(1-\frac{1}{p})-\frac{5}{2}} d \tau
		 + C e^{c_{2} t} \int_{\frac{t}{2}}^{t}
			(1+\tau)^{-\frac{5}{2}(1-\frac{1}{p})-\frac{5}{2}} d \tau \\
		& \le C e^{c_{2} \frac{t}{2}} 
		+ C e^{c_{2} t} 
		(1+t)^{-\frac{5}{2}(1-\frac{1}{p})-\frac{3}{2}},
	\end{split}
\end{equation*}
where we used the fact that $F(0)=0$.
Namely, we get
\begin{equation*}
	\begin{split}
		F(t) \le C 
		e^{c_{2} \frac{t}{2}} 
		+ C e^{c_{2} t} 
		(1+t)^{-\frac{5}{2}(1-\frac{1}{p})-\frac{3}{2}}
	\end{split}
\end{equation*}
and 
\begin{equation*}
	\begin{split}
		A(t)=e^{c_{2} t} \| \nabla^{3} u(t) \|_{p} \le C(1+t)^{-\frac{5}{2}(1-\frac{1}{p})-\frac{1}{2}} e^{c_{2}t} +C e^{\frac{c_{2}t}{2}}
		+ C e^{c_{2} t} 
		(1+t)^{-\frac{5}{2}(1-\frac{1}{p})-\frac{3}{2}}
	\end{split}
\end{equation*}
by \eqref{eq:5.15} again.
Therefore we conclude that
\begin{equation*}
	\begin{split}
		\| \nabla^{3} u(t) \|_{p} \le C (1+t)^{-\frac{5}{2}(1-\frac{1}{p})-\frac{1}{2}} +C e^{-\frac{c_{2}t}{2}}
		+ C  
		(1+t)^{-\frac{5}{2}(1-\frac{1}{p})-\frac{3}{2}},
	\end{split}
\end{equation*}
which implies the desired estimate \eqref{eq:3.1} with $\alpha=3$.
%

Similar arguments apply to obtain the estimate \eqref{eq:3.1} with $\alpha=1$ and $1<p \le 2$.  
On the other hand, we immediately have the estimate \eqref{eq:3.1} for $p \ge 2$ by the estimates \eqref{eq:2.2}, \eqref{eq:2.3} and the interpolation $\| \nabla u(t) \|_{p} \le \| \nabla u(t) \|_{\infty}^{1-\frac{2}{p}} \| \nabla u(t) \|_{2}^{\frac{2}{p}}$.
Therefore we can conclude the proof of the estimate \eqref{eq:3.1} with $\alpha=1$.

Once we obtain \eqref{eq:3.1} with $1 \le \alpha \le 3$, we easily have 
%
\begin{equation} \label{eq:5.16}
	\begin{split}
		\| F_{j}(u)(t) \|_{p} \le C \| \nabla u(t)  \|_{\infty} \| \nabla^{2} u(t)  \|_{p} 
		 \le C(1+t)^{-\frac{5}{2}(1-\frac{1}{p})-2}
	\end{split}
\end{equation}
and 
\begin{equation} \label{eq:5.17}
	\begin{split}
		\| \nabla F(u)(t) \|_{p} \le C(1+t)^{-2} \| \nabla^{3} u(t) \|_{p} \le C(1+t)^{-\frac{5}{2}(1-\frac{1}{p})-\frac{7}{2}} 
	\end{split}
\end{equation}
for $1 < p< \infty$.
Thus we can apply the same procedure of the estimate \eqref{eq:3.1} again to obtain the estimates \eqref{eq:3.2} and \eqref{eq:3.3}, using \eqref{eq:5.1} and \eqref{eq:5.5}.  

Next, we show the estimate \eqref{eq:3.4}.
Now, 
noting that $\partial_{t} \mathcal{K}^{(\beta)}_{1}(0,\xi)=1$,
\begin{equation} \label{eq:5.18}
	\begin{split}
& \partial_{t}^{2} \int_{0}^{t} K^{(\beta)}_{1}(t-\tau) \mathcal{R}_{a} \mathcal{R}_{b} F_{j}(u)(\tau) d \tau \\
& 
= \mathcal{R}_{a} \mathcal{R}_{b} F_{j}(u)(t) +\int_{0}^{t} \partial_{t}^{2}  K^{(\beta)}_{1}(t-\tau) \mathcal{R}_{a} \mathcal{R}_{b} F_{j}(u)(\tau) d \tau 
\end{split}
\end{equation}
and
\begin{equation} \label{eq:5.19}
\partial_{t}^{2} \int_{0}^{t} K^{(\beta)}_{1}(t-\tau) F_{j}(u)(\tau) d \tau
=  F_{j}(u)(t) +\int_{0}^{t} \partial_{t}^{2}  K^{(\beta)}_{1}(t-\tau)  F_{j}(u)(\tau) d \tau, 
\end{equation}
we use the estimates \eqref{eq:2.34}, \eqref{eq:2.24}, \eqref{eq:4.2}, \eqref{eq:4.58}, \eqref{eq:5.2}, \eqref{eq:5.16} and \eqref{eq:5.17} to obtain
\begin{equation} \label{eq:5.20}
	\begin{split}
		& \left\| \partial_{t}^{2} \int_{0}^{t} K^{(\beta)}_{1}(t-\tau) \mathcal{R}_{a} \mathcal{R}_{b} F_{j}(u)(\tau) d \tau \right\|_{p} 
		+ \left\| \partial_{t}^{2} \int_{0}^{t} K^{(\beta)}_{1}(t-\tau)  F_{j}(u)(\tau) d \tau \right\|_{p} 
		\\
		& \le C \| F_{j}(u)(t) \|_{p} \\
		& +C \int_{0}^{\frac{t}{2}} (1+t-\tau)^{-\frac{5}{2}(1-\frac{1}{p})} \| F(u)(\tau)\|_{1} d \tau 
		+C \int_{\frac{t}{2}}^{t}  \| F(u)(\tau)\|_{p} d \tau \\
		& +C  \int_{0}^{t} e^{-c(t-\tau)}(1+(t-\tau)^{-\frac{1}{2}}) \| \nabla F(u)(\tau)\|_{p} d \tau \\
		& \le 
		C (1+t)^{-\frac{5}{2}(1-\frac{1}{p})-2}+ C  \int_{0}^{\frac{t}{2}} (1+t-\tau)^{-\frac{5}{2}(1-\frac{1}{p})} (1+\tau)^{-2} d \tau \\
		& + C  \int_{\frac{t}{2}}^{t}  (1+\tau)^{-\frac{5}{2}(1-\frac{1}{p})-2}  d \tau
		+
		C  \int_{0}^{t} e^{-c(t-\tau)} (1+(t-\tau)^{-\frac{1}{2}}) (1+\tau)^{-\frac{5}{2}(1-\frac{1}{p})-\frac{7}{2}}  d \tau  \\
		& \le (1+t)^{-\frac{5}{2}(1-\frac{1}{p})}.
	\end{split}
\end{equation}
Therefore we get the estimate \eqref{eq:3.4} by \eqref{eq:5.1}, \eqref{eq:5.5} and \eqref{eq:5.20}. 
%

It remains to prove estimates \eqref{eq:3.5}-\eqref{eq:3.7}.
We only prove the estimate \eqref{eq:3.5}, since we can obtain estimates \eqref{eq:3.5}-\eqref{eq:3.7} by a similar way.
We denote 
\begin{equation} \label{eq:5.21}
	\begin{split}
		G(t,x) & = G_{0,lin}(t,x) +G_{1,lin}(t,x)+G_{N}(t,x),
	\end{split}
\end{equation}
where 
\begin{equation*}
	\begin{split}
		G_{0,lin}(t,x)& := 
		\nabla^{-1} \mathcal{F}^{-1} \left[ 
		\left( \mathcal{G}^{(\sqrt{\lambda+2 \mu})}_{0}(t,\xi) -\mathcal{G}^{(\sqrt{\mu})}_{0}(t,\xi) \right)\mathcal{P} m_{0}
		+
		\mathcal{G}^{(\sqrt{\mu})}_{0}(t,\xi) m_{0}
		\right], \\
		G_{1,lin}(t,x)& := 
		\mathcal{F}^{-1} \left[ 
		\left( \mathcal{G}^{(\sqrt{\lambda+2 \mu})}_{1}(t,\xi) -\mathcal{G}^{(\sqrt{\mu})}_{1}(t,\xi) \right)\mathcal{P} m_{1}
		+
		\mathcal{G}^{(\sqrt{\mu})}_{1}(t,\xi) m_{1}
		\right], \\ 
		G_{N}(t,x)& := 
		\mathcal{F}^{-1} \left[ 
		\left( \mathcal{G}^{(\sqrt{\lambda+2 \mu})}_{1}(t,\xi) -\mathcal{G}^{(\sqrt{\mu})}_{1}(t,\xi) \right)\mathcal{P} M[u]
		+
		\mathcal{G}^{(\sqrt{\mu})}_{1}(t,\xi) M[u]
		\right]
	\end{split}
\end{equation*}
and $\mathcal{P}$ is defined by \eqref{eq:2.1}.
Using the decomposition \eqref{eq:5.21}, we claim that 
\begin{align}
	& \| \nabla^{\alpha} (u_{lin}(t)-G_{lin}(t)) \|_{p} = o(t^{-\frac{5}{2}(1-\frac{1}{p})+1-\frac{\alpha}{2}}), \quad 1 \le \alpha \le 3, \label{eq:5.22} \\
	&  \| \nabla^{\alpha} (u_{N}(t)-G_{N}(t)) \|_{p} = o(t^{-\frac{5}{2}(1-\frac{1}{p})+1-\frac{\alpha}{2}}), \quad 1 \le \alpha \le 3, \label{eq:5.23}  
\end{align}
as $t \to \infty$.
The estimate \eqref{eq:5.22} is shown by the combination of the linear estimates \eqref{eq:4.21}, \eqref{eq:4.23}, \eqref{eq:4.25}, \eqref{eq:4.27}, \eqref{eq:4.57} and \eqref{eq:4.58}.
The proof is completed by showing the estimate \eqref{eq:5.23}.
For this purpose, we recall the useful estimates from \cite{K-T} and the decomposition of $u_{N}(t)-G_{N}(t)$:
\begin{equation} \label{eq:5.24}
	\begin{split}
		& \| \partial_{t}^{\ell} \nabla^{\alpha} (G_{1M}(t) +G_{1H}(t)) \|_{p}
		\left|
		\int_{0}^{\frac{t}{2}} \int_{\R^{3}} F(u)(\tau, y) dy d \tau 
		\right| \\
		& +
		\| \partial_{t}^{\ell} \nabla^{\alpha} 
		\mathcal{F}^{-1} [\mathcal{G}^{(\beta)}_{1}(t,\xi)\mathcal{P}(\chi_{M} +\chi_{H}) ] \|_{p}
		\left|
		\int_{0}^{\frac{t}{2}} \int_{\R^{3}} F(u)(\tau, y) dy d \tau 
		\right| \\
		& \le Ce^{-ct} t^{-\frac{3}{2}(1-\frac{1}{p})-\frac{\alpha+\ell}{2}}
	\end{split}
\end{equation}
for $1 \le p \le \infty$ and $\alpha, \ell \ge 0$ and 
\begin{equation} \label{eq:5.25}
	\begin{split}
		& \| \partial_{t}^{\ell} \nabla^{\alpha} G_{1}(t) \|_{p}
		\left|
		\int_{\frac{t}{2}}^{\infty} \int_{\R^{3}} F(u)(\tau, y) dy d \tau 
		\right| \\
		& +
		\| \partial_{t}^{\ell} \nabla^{\alpha} 
		\mathcal{F}^{-1} [\mathcal{G}^{(\beta)}_{1}(t,\xi) \mathcal{P} ] \|_{p}
		\left|
		\int_{\frac{t}{2}}^{\infty} \int_{\R^{3}} F(u)(\tau, y) dy d \tau 
		\right| \le C t^{-\frac{5}{2}(1-\frac{1}{p})-\frac{\alpha+\ell}{2}},
	\end{split}
\end{equation}
where $1<p \le \infty$ for $\ell+\alpha=0$ and $1 \le p \le \infty$ for $\ell+\alpha \ge 1$.

The integrand $u_{N}(t)-G_{N}(t)$ is decomposed into 3 parts; 
\begin{equation*} 
	\begin{split}
		u_{N}(t)-G_{N}(t) & = J_{0, \mathcal{P}, \lambda+2 \mu}- J_{0,\mathcal{P},\mu}+ J_{0,\phi, \mu},
	\end{split}
\end{equation*}
where
\begin{equation*} 
	\begin{split}
		&J_{0, \mathcal{P}, \beta} 
		:= \\
		& \int_{0}^{\frac{t}{2}}
		K_{1L}^{(\sqrt{\beta})}(t-\tau) \mathcal{F}^{-1} [\mathcal{P} \hat{F}(u)(\tau) ] d \tau 
		-
		\mathcal{F}^{-1}
		\left[
		\mathcal{G}_{1}^{(\sqrt{\beta})}(t,\xi) \chi_{L} \mathcal{P}
		\right]\, \int_{0}^{\frac{t}{2}} \int_{\R^{3}} F(u)(\tau, y) dy d \tau 
		\\
		& - \mathcal{F}^{-1}
		\left[
		\mathcal{G}_{1}^{(\sqrt{\beta})}(t,\xi) (\chi_{M}+\chi_{H}) \mathcal{P}
		\right]\, \int_{0}^{\frac{t}{2}} \int_{\R^{3}} F(u)(\tau, y) dy d \tau \\
		& +\int_{\frac{t}{2}}^{t}
		K_{1L}^{(\sqrt{\beta})}(t-\tau) \mathcal{F}^{-1} [\mathcal{P} \hat{F}(u)(\tau) ] d \tau
		+\sum_{k=M,H} \int_{0}^{t}
		K_{1k}^{(\sqrt{\beta})}(t-\tau) \mathcal{F}^{-1} [\mathcal{P} \hat{F}(u)(\tau) ] d \tau \\
		& -
		\mathcal{F}^{-1}
		\left[
		\mathcal{G}_{1}^{(\sqrt{\beta})}(t,\xi) \mathcal{P}
		\right]\, \int_{\frac{t}{2}}^{\infty} \int_{\R^{3}} F(u)(\tau, y) dy d \tau 
	\end{split}
\end{equation*}
and 
\begin{equation*} 
	\begin{split}
		J_{0, \phi, \beta} 
		& :=
		\int_{0}^{\frac{t}{2}}
		K_{1L}^{(\sqrt{\beta})}(t-\tau) F(u)(\tau)  d \tau  -
		G_{1L}^{(\sqrt{\beta})}(t)  \int_{0}^{\frac{t}{2}} \int_{\R^{3}} F(u)(\tau, y) dy d \tau 
		\\
		& -
		(G_{1M}^{(\sqrt{\beta})}(t)+G_{1M}^{(\sqrt{\beta})}(t))  \int_{0}^{\frac{t}{2}} \int_{\R^{3}} F(u)(\tau, y) dy d \tau \\
		& +\int_{\frac{t}{2}}^{t}
		K_{1L}^{(\sqrt{\beta})}(t-\tau) F(u)(\tau) d \tau
		+\sum_{k=M,H} \int_{0}^{t}
		K_{1k}^{(\sqrt{\beta})}(t-\tau) F(u)(\tau) d \tau \\
		& -
		G_{1}^{(\sqrt{\beta})}(t) \int_{\frac{t}{2}}^{\infty} \int_{\R^{3}} F(u)(\tau, y) dy d \tau. 
	\end{split}
\end{equation*}
Then we see that 
\begin{equation} \label{eq:5.26}
	\begin{split}
		\| \nabla^{\alpha} J_{0, \mathcal{P}, \beta} \|_{p} 
		& \le o(t^{-\frac{5}{2}(1-\frac{1}{p})+1-\frac{\alpha}{2}} ) +C \int_{\frac{t}{2}}^{t}
		(1+t-\tau)^{1-\frac{\alpha}{2}} \| F(u)(\tau) \|_{p}  d \tau\\
		& +\sum_{k=M,H} \int_{0}^{t}
		e^{-c(t-\tau)}(1+(t-\tau)^{1-\frac{\alpha}{2}}) \| \nabla F(u)(\tau) \|_{p} d \tau
		+
		C t^{-\frac{5}{2}(1-\frac{1}{p})-\frac{\alpha}{2}} \\
		& \le o(t^{-\frac{5}{2}(1-\frac{1}{p})+1-\frac{\alpha}{2}}) +C \int_{\frac{t}{2}}^{t}
		(1+t-\tau)^{1-\frac{\alpha}{2}} (1+\tau)^{-\frac{5}{2}(1-\frac{1}{p})-\frac{1}{p}-\frac{3}{2}} d \tau \\
		& +C\int_{0}^{t}
		e^{-c(t-\tau)} (1+\tau)^{-\frac{5}{2}(1-\frac{1}{p})-\frac{7}{2}}  d \tau
		=o(t^{-\frac{5}{2}(1-\frac{1}{p})+1-\frac{\alpha}{2}} ) 
	\end{split}
\end{equation}
as $t \to \infty$ for $1 \le \alpha \le 3$, by \eqref{eq:4.2}, \eqref{eq:4.29}, \eqref{eq:4.58}, \eqref{eq:5.16}, \eqref{eq:5.24} and \eqref{eq:5.25}.
Likewise, we can get 
\begin{equation} \label{eq:5.27}
	\begin{split}
		\| \nabla^{\alpha} J_{0, \phi, \beta} \|_{2} =o(t^{-\frac{5}{2}(1-\frac{1}{p})+1-\frac{\alpha}{2}} ) 
	\end{split}
\end{equation}
as $t \to \infty$ for $1 \le \alpha \le 3$.
%
Combining \eqref{eq:5.26} and \eqref{eq:5.27}, 
we have
\begin{equation*} 
	\begin{split}
		\| \nabla^{\alpha} (u_{N}(t)-G_{N}(t))\|_{p} & \le  
		\| \nabla^{\alpha} J_{0, \mathcal{P}, \lambda+2 \mu}\|_{p}+\| \nabla^{\alpha} J_{0,\mathcal{P},\mu} \|_{p}+
		+\| \nabla^{\alpha}  J_{0,\phi, \lambda+2 \mu} \|_{p} \\
		& =o(t^{-\frac{5}{2}(1-\frac{1}{p})+1-\frac{\alpha}{2}} ) 
	\end{split}
\end{equation*}
as $t \to \infty$ for $1 \le \alpha \le 3$, which is the desired estimate \eqref{eq:5.23}.
We complete the proof of Theorem \ref{thm:3.1}.
\subsection{Proof of Theorem \ref{thm:3.2}}
Let $F(u)=\nabla u \nabla \partial_{t} u$.
At first, we note that the linear estimate \eqref{eq:5.5} obtained in the proof of Theorem \ref{thm:3.1} is still valid.
Then it suffices to show the estimates for nonlinear term.  
Now we observe that 
\begin{equation} \label{eq:5.28}
	\begin{split}
		& \left\| \nabla^{3}  \int_{0}^{t} K^{(\beta)}_{1}(t-\tau) \mathcal{R}_{a} \mathcal{R}_{b} F_{j}(u)(\tau) d \tau \right\|_{p} 
		+ \left\| \nabla^{3}  \int_{0}^{t} K^{(\beta)}_{1}(t-\tau)  F_{j}(u)(\tau) d \tau \right\|_{p} \\
		& \left\| \nabla^{2} \partial_{t} \int_{0}^{t} K^{(\beta)}_{1}(t-\tau) \mathcal{R}_{a} \mathcal{R}_{b} F_{j}(u)(\tau) d \tau \right\|_{p} 
		+ \left\| \nabla^{2} \partial_{t} \int_{0}^{t} K^{(\beta)}_{1}(t-\tau)  F_{j}(u)(\tau) d \tau \right\|_{p} \\
		& \le \sum_{k=L,M,H}  \int_{0}^{t} \left( 
		\left\| \nabla^{3} K^{(\beta)}_{1k}(t-\tau)  F_{j}(u)(\tau) \right\|_{p}  
		+
		\left\| \nabla^{2} \partial_{t} K^{(\beta)}_{1k}(t-\tau)  F_{j}(u)(\tau) \right\|_{p}  
		\right)
		d \tau \\
		& \le C  \int_{0}^{\frac{t}{2}} (1+t-\tau)^{-\frac{5}{2}(1-\frac{1}{p})-\frac{1}{2}} 
		\| F_{j}(u)(\tau) \|_{1} d \tau \\
		&
		+ \int_{\frac{t}{2}}^{t} \left( 
		\left\| \nabla^{3} K^{(\beta)}_{1L}(t-\tau)  F_{j}(u)(\tau) \right\|_{p}  
		+
		\left\| \nabla^{2} \partial_{t} K^{(\beta)}_{1L}(t-\tau)  F_{j}(u)(\tau) \right\|_{p}  
		\right)
		d \tau  \\
		& + C  \int_{0}^{t} e^{-c(t-\tau)} \| \nabla F_{j}(u)(\tau )  \|_{p} d \tau 
		+ \int_{0}^{t} 
		\left\| \nabla^{2} \partial_{t} K^{(\beta)}_{1H}(t-\tau)  F_{j}(u)(\tau) \right\|_{p}  
		d \tau 
	\end{split}
\end{equation}
by \eqref{eq:2.34}, \eqref{eq:2.24}, \eqref{eq:4.2} and \eqref{eq:4.58}, and 
\begin{equation} \label{eq:5.29}
	\begin{split}
		\| \nabla F(u) \|_{p} & \le C \| \nabla^{2} u \partial_{t} \nabla u \|_{p}+C  \| \nabla u \|_{\infty} \| \partial_{t} \nabla^{2} u \|_{p} \\   
		& \le C \| \nabla^{2} u \|_{2p} \| \nabla \partial_{t} u \|_{2p} + C \| \nabla u \|_{\infty} \| \nabla^{2} \partial_{t} u \|_{p} \\
		& \le C
		\| \nabla u \|_{\infty}^{\frac{1}{2}} 
		\| \nabla^{3} u \|_{p}^{\frac{1}{2}}
		\| \partial_{t} u \|_{\infty}^{\frac{1}{2}} 
		\| \nabla^{2} \partial_{t} u \|_{p}^{\frac{1}{2}}  
		+ C \| \nabla u \|_{\infty} \| \nabla^{2} \partial_{t} u \|_{p} \\
		& \le C(1+t)^{-2}(\| \nabla^{3} u \|_{p} + \| \nabla^{2} \partial_{t} u \|_{p})
	\end{split}
\end{equation}
by the H\"older inequality, and \eqref{eq:2.38}, \eqref{eq:2.3} with $\alpha=1$ and \eqref{eq:2.8} for $1 < p < \infty$.

On the other hand, 
when $1 \le  p \le 2$, 
we see that
\begin{equation} \label{eq:5.30}
	\begin{split}
		\| \nabla F(u)(t) \|_{p} \le \| \nabla F(u) \|_{1}^{\frac{2}{p}-1} \| \nabla F(u) \|_{2}^{2-\frac{2}{p}}  \le C(1+t)^{-\frac{5}{2}(1-\frac{1}{p})-\frac{1}{p}-\frac{3}{2}},
	\end{split}
\end{equation}
where we used the fact that 
\begin{equation*} 
	\begin{split}
		\| \nabla F(u)(t) \|_{1} \le \| \nabla^{2} u \|_{2} \| \nabla \partial_{t} u\|_{2} 
		+
		\| \nabla u \|_{2} \| \nabla^{2} \partial_{t} u\|_{2} \le C(1+t)^{-\frac{5}{2}}
	\end{split}
\end{equation*}
and \eqref{eq:5.4}.
Then observing that 
$
-2 < -\frac{7}{4} <-\frac{5}{2}(1-\frac{1}{p})-\frac{1}{2}
$
for $1<p<2$,
we can obtain the estimate of the RHS of \eqref{eq:5.28} as follows:
\begin{equation} \label{eq:5.31}
	\begin{split}
		& (\text{RHS} \ \text{of}\ \eqref{eq:5.28}) \\
		& \le C  \int_{0}^{\frac{t}{2}} (1+t-\tau)^{-\frac{5}{2}(1-\frac{1}{p})-\frac{1}{2}} 
		(1+\tau)^{-2} d \tau 
		+C  \int_{\frac{t}{2}}^{t} \| \nabla F_{j}(u)(\tau) \|_{p}  d \tau   \\
		&  +C  \int_{0}^{t} e^{-c(t-\tau)}
		(1+(t-\tau)^{-\frac{1}{2}})
		\| \nabla F_{j}(u)(\tau) \|_{p}  d \tau   \\ 
		& \le C  \int_{0}^{\frac{t}{2}} (1+t-\tau)^{-\frac{5}{2}(1-\frac{1}{p})-\frac{1}{2}} 
		(1+\tau)^{-2} d \tau 
		+C  \int_{\frac{t}{2}}^{t} (1+\tau)^{-\frac{5}{2}(1-\frac{1}{p})-\frac{1}{p}-\frac{3}{2}}  d \tau   \\
		&  +C  \int_{0}^{t} e^{-c(t-\tau)}
		(1+(t-\tau)^{-\frac{1}{2}})
		(1+\tau)^{-\frac{5}{2}(1-\frac{1}{p})-\frac{1}{p}-\frac{3}{2}}  d \tau   \\ 
		& \le C(1+t)^{-\frac{5}{2}(1-\frac{1}{p})-\frac{1}{2}}
	\end{split}
\end{equation}
by \eqref{eq:2.24}, \eqref{eq:4.58}, \eqref{eq:5.2} and \eqref{eq:5.30}.
Therefore we conclude the estimates \eqref{eq:3.1} with $\alpha=3$ and \eqref{eq:3.3} for $1<p<2$ by \eqref{eq:5.5} and \eqref{eq:5.31}.

When $2<p<6$, 
noting that  
$-1<-\frac{5}{2}(\frac{1}{2}-\frac{1}{p})-\frac{1}{2}<0$,
we also have 
\begin{equation} \label{eq:5.32}
	\begin{split}
		& (\text{RHS} \ \text{of}\ \eqref{eq:5.28}) \\
		& \le C  \int_{0}^{\frac{t}{2}} (1+t-\tau)^{-\frac{5}{2}(1-\frac{1}{p})-\frac{1}{2}} 
		(1+\tau)^{-2} d \tau \\
		& +C  \int_{\frac{t}{2}}^{t} (1+t-\tau)^{-\frac{5}{2}(\frac{1}{2}-\frac{1}{p})} \| \nabla F_{j}(u)(\tau) \|_{2}  d \tau   \\
		&  +C  \int_{0}^{t} e^{-c(t-\tau)}
		(1+\tau)^{-2}  (\| \nabla^{3} u(\tau) \|_{p} + \| \nabla^{2} \partial_{t} u(\tau) \|_{p})d \tau   \\ 
		&  +C  \int_{0}^{t} e^{-c(t-\tau)}
		(t-\tau)^{-\frac{5}{2}(\frac{1}{2}-\frac{1}{p})-\frac{1}{2}} \| \nabla F_{j}(u)(\tau) \|_{2} d \tau   \\ 
		& \le C  (1+t)^{-\frac{5}{2}(1-\frac{1}{p})-\frac{1}{2}} 
		+C  \int_{\frac{t}{2}}^{t} (1+t-\tau)^{-\frac{5}{2}(\frac{1}{2}-\frac{1}{p})} (1+\tau)^{-\frac{13}{4}}  d \tau   \\
		&  +C  \int_{0}^{t} e^{-c(t-\tau)}
		(1+\tau)^{-2}  (\| \nabla^{3} u(\tau) \|_{p} + \| \nabla^{2} \partial_{t} u(\tau) \|_{p})d \tau   \\ 
		&  +C  \int_{0}^{t} e^{-c(t-\tau)}
		(t-\tau)^{-\frac{5}{2}(\frac{1}{2}-\frac{1}{p})-\frac{1}{2}} (1+\tau)^{-\frac{13}{4}} d \tau   \\ 
		& \le C(1+t)^{-\frac{5}{2}(1-\frac{1}{p})-\frac{1}{2}} + C  \int_{0}^{t} e^{-c(t-\tau)}
		(1+\tau)^{-2}  (\| \nabla^{3} u(\tau) \|_{p} + \| \nabla^{2} \partial_{t} u(\tau) \|_{p})d \tau
	\end{split}
\end{equation}
by \eqref{eq:2.24}, \eqref{eq:5.2}, \eqref{eq:5.4} and \eqref{eq:5.29}.
Combining the estimates \eqref{eq:5.5} and \eqref{eq:5.32}, 
we arrive at the estimate    
\begin{equation} \label{eq:5.33}
	\begin{split}
	& \| \nabla^{3} u(t) \|_{p} + \| \nabla^{2} \partial_{t} u(t) \|_{p} \\
		& \le C(1+t)^{-\frac{5}{2}(1-\frac{1}{p})}t^{-\frac{1}{2}} + C  \int_{0}^{t} e^{-c(t-\tau)}
		(1+\tau)^{-2}  (\| \nabla^{3} u(\tau) \|_{p} + \| \nabla^{2} \partial_{t} u(\tau) \|_{p})d \tau.
	\end{split}
\end{equation}
Then we have 
\begin{equation} \label{eq:5.34}
	\begin{split}
		& \| \nabla^{3} u(t) \|_{p} + \| \nabla^{2} \partial_{t} u(t) \|_{p} \le C(1+t)^{-\frac{5}{2}(1-\frac{1}{p})}t^{-\frac{1}{2}}
	\end{split}
\end{equation}
by the same argument as that for the proof of \eqref{eq:3.1} with $\alpha=3$ in Theorem \ref{thm:3.1}.
The estimate \eqref{eq:5.34} shows \eqref{eq:3.3} for $2 < p < 6$.
Moreover we apply the same argument as that to $\| \nabla^{3} u(t) \|_{p}$ with \eqref{eq:5.34} to have  
\begin{equation} \label{eq:5.35}
	\begin{split}
		& \| \nabla^{3} u(t) \|_{p} \\
		 & \le  \| \nabla^{3} u_{lin}(t) \|_{p} + \| \nabla^{3} u_{N}(t) \|_{p} \\
		& \le C(1+t)^{-\frac{5}{2}(1-\frac{1}{p})-\frac{1}{2}} + C  \int_{0}^{t} e^{-c(t-\tau)}
		(1+\tau)^{-2}  (\| \nabla^{3} u(\tau) \|_{p} + \| \nabla^{2} \partial_{t} u(\tau) \|_{p})d \tau \\
		& \le C(1+t)^{-\frac{5}{2}(1-\frac{1}{p})-\frac{1}{2}} + C  \int_{0}^{t} e^{-c(t-\tau)}
		(1+\tau)^{-\frac{5}{2}(1-\frac{1}{p})-2} \tau^{-\frac{1}{2}} d \tau \\
		& \le C(1+t)^{-\frac{5}{2}(1-\frac{1}{p})-\frac{1}{2}},
	\end{split}
\end{equation}
which is the desired estimate \eqref{eq:3.1} with $\alpha=3$ for $2 < p<6$.
Here we used the fact that
\begin{equation*} 
	\begin{split}
	    & \int_{0}^{t} e^{-c(t-\tau)}
		(1+\tau)^{-\frac{5}{2}(1-\frac{1}{p})-2} \tau^{-\frac{1}{2}} d \tau \\
		& \le Ce^{-ct} \int_{0}^{\frac{t}{2}} 
		 \tau^{-\frac{1}{2}} d \tau + C (1+t)^{-\frac{5}{2}(1-\frac{1}{p})-2} t^{-\frac{1}{2}}  \int_{\frac{t}{2}}^{t} 
		 d \tau  
		\le C(1+t)^{-\frac{5}{2}(1-\frac{1}{p})-\frac{3}{2}}.
	\end{split}
\end{equation*}
Next we deal with the case $6 \le p< \infty$.
In this case, we firstly claim that 
\begin{equation} \label{eq:5.36}
	\begin{split}
		& \| \nabla^{3} u(t) \|_{5} \le C(1+t)^{-\frac{5}{2}}, \\
		& \| \nabla^{2} \partial_{t} u(t) \|_{5} \le C(1+t)^{-2}t^{-\frac{1}{2}}.
	\end{split}
\end{equation}
Indeed, the linear solution is estimated as 
\begin{equation*} 
	\begin{split}
		& \| \nabla^{3} u_{lin}(t) \|_{5} \le C(1+t)^{-\frac{5}{2}}, \\
		& \| \nabla^{2} \partial_{t} u_{lin}(t) \|_{5} \le C(1+t)^{-2}t^{-\frac{1}{2}}
	\end{split}
\end{equation*}
by \eqref{eq:5.5} and 
$\| \nabla f \|_{5} \le \| \nabla f \|_{2}^{\frac{2(p-5)}{5(p-2)}} \| \nabla f \|_{p}^{\frac{3p}{5(p-2)}}$.
Then we can apply the estimates \eqref{eq:5.34}, \eqref{eq:5.35} with $p=5$ to have \eqref{eq:5.36}.
Moreover we have 
\begin{equation} \label{eq:5.37}
	\begin{split}
		& \| \nabla F(u)(t) \|_{5} \le C(1+t)^{-4}t^{-\frac{1}{2}}
	\end{split}
\end{equation}
by \eqref{eq:5.29} and \eqref{eq:5.36}.
We now turn to the  proof of \eqref{eq:3.1} with $\alpha=3$ and \eqref{eq:3.3} for $6 \le p< \infty$.
By \eqref{eq:2.24}, \eqref{eq:4.58}, \eqref{eq:5.4}, \eqref{eq:5.29} and \eqref{eq:5.37}, we get 
\begin{equation*}
	\begin{split}
		& (\text{RHS} \ \text{of}\ \eqref{eq:5.28}) \\
		& \le C  \int_{0}^{\frac{t}{2}} (1+t-\tau)^{-\frac{5}{2}(1-\frac{1}{p})-\frac{1}{2}} 
		(1+\tau)^{-2} d \tau \\
		& +C  \int_{\frac{t}{2}}^{t} (1+t-\tau)^{-\frac{5}{2}(\frac{1}{2}-\frac{1}{p})} \| \nabla F_{j}(u)(\tau) \|_{2}  d \tau   \\
		&  +C  \int_{0}^{t} e^{-c(t-\tau)}
		(1+\tau)^{-2}  (\| \nabla^{3} u(\tau) \|_{p} + \| \nabla^{2} \partial_{t} u(\tau) \|_{p})d \tau   \\ 
		&  +C  \int_{0}^{t} e^{-c(t-\tau)}
		(t-\tau)^{-\frac{5}{2}(\frac{1}{5}-\frac{1}{p})-\frac{1}{2}} \| \nabla F_{j}(u)(\tau) \|_{5} d \tau   \\ 
		& \le C  (1+t)^{-\frac{5}{2}(1-\frac{1}{p})-\frac{1}{2}} 
		+C  \int_{\frac{t}{2}}^{t} (1+t-\tau)^{-\frac{5}{2}(\frac{1}{2}-\frac{1}{p})} (1+\tau)^{-\frac{13}{4}}  d \tau   \\
		&  +C  \int_{0}^{t} e^{-c(t-\tau)}
		(1+\tau)^{-2}  (\| \nabla^{3} u(\tau) \|_{p} + \| \nabla^{2} \partial_{t} u(\tau) \|_{p})d \tau   \\ 
		&  +C  \int_{0}^{t} e^{-c(t-\tau)}
		(t-\tau)^{-\frac{5}{2}(\frac{1}{5}-\frac{1}{p})-\frac{1}{2}} (1+\tau)^{-4} \tau^{-\frac{1}{2}} d \tau   \\ 
		& \le C(1+t)^{-\frac{5}{2}(1-\frac{1}{p})}t^{-\frac{1}{2}} + C  \int_{0}^{t} e^{-c(t-\tau)}
		(1+\tau)^{-2}  (\| \nabla^{3} u(\tau) \|_{p} + \| \nabla^{2} \partial_{t} u(\tau) \|_{p})d \tau,
	\end{split}
\end{equation*}
where we used the facts that $-\frac{5}{2}(\frac{1}{5}-\frac{1}{p})-\frac{1}{2}>-1$ and 
\begin{equation*} 
	\begin{split}
		& \int_{0}^{t} e^{-c(t-\tau)}
		(t-\tau)^{-\frac{5}{2}(\frac{1}{5}-\frac{1}{p})-\frac{1}{2}} (1+\tau)^{-4} \tau^{-\frac{1}{2}} d \tau   \\ 
		& \le Ce^{-ct} t^{-\frac{5}{2}(\frac{1}{5}-\frac{1}{p})-\frac{1}{2}} \int_{0}^{\frac{t}{2}} 
		\tau^{-\frac{1}{2}} d \tau + C (1+t)^{-4} t^{-\frac{1}{2}}  \int_{\frac{t}{2}}^{t} (t-\tau)^{-\frac{5}{2}(\frac{1}{5}-\frac{1}{p})-\frac{1}{2}} 
		d \tau \\ 
		& \le C(1+t)^{-4} t^{-\frac{5}{2}(\frac{1}{5}-\frac{1}{p})} \le  C(1+t)^{-\frac{5}{2}(1-\frac{1}{p})}t^{-\frac{1}{2}}
	\end{split}
\end{equation*}
for $6 \le p <\infty$.
Therefore we arrive at \eqref{eq:5.33} again and have \eqref{eq:5.34} for $6 \le p< \infty$.
As in \eqref{eq:5.35}, 
we also have \eqref{eq:3.1} with $\alpha=3$ for $6 \le p< \infty$.
Summing up the above estimates, we obtain \eqref{eq:3.1} with $\alpha=3$ and \eqref{eq:3.3} for $1< p< \infty$. 
 
Once we have them, the other estimates in Theorem \ref{thm:3.2} are obtained by the same way as in the proof of Theorem \ref{thm:3.1}.
We omit the detail. 
We complete the proof of Theorem \ref{thm:3.2}. 
%
\section{Proof of main results ($p=1$)} 
%
\begin{proof}[Proof of Theorems \ref{thm:3.3}-\ref{thm:3.4}]
	Theorems \ref{thm:3.3}-\ref{thm:3.4} are shown in a similar way. We only prove Theorem \ref{thm:3.3}.
	
	We begin with the linear estimates.
	Denoting 
	$$\| f_{0}, f_{1}\|_{Y_{1}}:= \| f_{0} \|_{\dot{H}^{3} \cap \dot{W}^{1,1} \cap \dot{W}^{3,1}} + \| f_{1} \|_{H^{1} \cap W^{1,1}},$$ 
	we simply apply the estimates \eqref{eq:2.23},\eqref{eq:2.24}, \eqref{eq:4.1}, \eqref{eq:4.2}, \eqref{eq:4.59} and \eqref{eq:4.60} to have 
	\begin{equation} \label{eq:6.1}
	\begin{split}
		 \sup_{t \ge 0} \biggl\{  \sum_{\alpha=0,1} (1+t)^{-\frac{1}{2}+\frac{\alpha}{2}} \| \nabla^{\alpha} \partial_{t} u_{lin}(t) \|_{1} + \| \partial_{t}^{2} u_{lin}(t) \|_{1}  
		\biggr\} \le C \| f_{0}, f_{1}\|_{Y_{1}}.
	\end{split}
\end{equation}
We also have 	
	\begin{equation} \label{eq:6.2}
	\begin{split}
		 \sup_{t \ge 0}  (1+t)^{-\frac{1}{2} } \| \nabla u_{lin}(t) \|_{1} \le C \| f_{0}, f_{1}\|_{Y_{1}}
	\end{split}
\end{equation}	
	by \eqref{eq:2.23},\eqref{eq:2.24}, \eqref{eq:4.2} and \eqref{eq:4.60}, and the fact that 
		\begin{equation} \label{eq:6.3}
		\begin{split}
			& 
			\| \nabla (K_{0}^{(\beta)}(t)-K_{0}^{(\gamma)}(t) )\mathcal{R}_{a}  \mathcal{R}_{b} g \|_{1} \\
			& 
			\le \sum_{k=L,M,H}
			\| \nabla (K_{0k}^{(\beta)}(t)-K_{0k}^{(\gamma)}(t) )\mathcal{R}_{a}  \mathcal{R}_{b} g \|_{1} \\
			& 
			\le
			\| \nabla (K_{0L}^{(\beta)}(t)-G_{0L}^{(\beta)}(t) ) \mathcal{R}_{a}  \mathcal{R}_{b} g \|_{1}
			+ \| \nabla (K_{0L}^{(\gamma)}(t)-G_{0L}^{(\gamma)}(t) ) \mathcal{R}_{a}  \mathcal{R}_{b} g \|_{1} \\
			& +\| \nabla (G_{0L}^{(\beta)}(t)-G_{0L}^{(\gamma)}(t) ) \ast \mathcal{R}_{a}  \mathcal{R}_{b} g \|_{1} \\
			& + \sum_{k=M,H}
			(\| \nabla K_{0k}^{(\beta)}(t)\mathcal{R}_{a}  \mathcal{R}_{b} g \|_{1} +\| \nabla K_{0k}^{(\gamma)}(t) \mathcal{R}_{a}  \mathcal{R}_{b} g \|_{1}) \\
			& \le C (1+t)^{\frac{1}{2}} \| \nabla g \|_{1} + C e^{-ct} \| \nabla^{3} g \|_{1}
		\end{split}
	\end{equation}
	by \eqref{eq:4.5}, \eqref{eq:4.9}, \eqref{eq:4.18} and \eqref{eq:4.59}.
	
	For the nonlinear terms, we again use the estimates \eqref{eq:2.24}, \eqref{eq:2.28}, \eqref{eq:4.2}, \eqref{eq:4.60} and \eqref{eq:5.2} to obtain  
	\begin{equation} \label{eq:6.4}
		\begin{split}
			& \left\| \nabla  \int_{0}^{t} K^{(\beta)}_{1}(t-\tau) \mathcal{R}_{a} \mathcal{R}_{b} F_{j}(u)(\tau) d \tau \right\|_{1} 
			+ \left\| \nabla  \int_{0}^{t} K^{(\beta)}_{1}(t-\tau)  F_{j}(u)(\tau) d \tau \right\|_{1} \\
			& \le \sum_{k=L,M,H}  \int_{0}^{t} \left( 
			\left\| \nabla K^{(\beta)}_{1k}(t-\tau) \mathcal{R}_{a} \mathcal{R}_{b}  F_{j}(u)(\tau) \right\|_{1}  
			+
			\left\| \nabla K^{(\beta)}_{1k}(t-\tau)  F_{j}(u)(\tau) \right\|_{1}  
			\right)
			d \tau \\
			& \le C  \int_{0}^{t} (1+t-\tau)^{\frac{1}{2}} 
			\| F_{j}(u)(\tau) \|_{1} d \tau 
			+
			C \int_{0}^{t} e^{-c(t-\tau)} 
			\| F_{j}(u)(\tau) \|_{1} d \tau   \\
			& \le  C  \int_{0}^{t} (1+t-\tau)^{\frac{1}{2}} 
			(1+\tau)^{-2} d \tau 
			+
			C \int_{0}^{t} e^{-c(t-\tau)} (1+\tau)^{-2}
			 d \tau   \\
			& \le  C(1+t)^{\frac{1}{2}}.
		\end{split}
	\end{equation}
	By the same way we have 
	\begin{equation}  \label{eq:6.5}
		\begin{split}
			& \left\| \nabla^{\alpha} \partial_{t}  \int_{0}^{t} K^{(\beta)}_{1}(t-\tau) \mathcal{R}_{a} \mathcal{R}_{b} F_{j}(u)(\tau) d \tau \right\|_{1} 
			+ \left\| \nabla^{\alpha} \partial_{t}  \int_{0}^{t} K^{(\beta)}_{1}(t-\tau)  F_{j}(u)(\tau) d \tau \right\|_{1} \\
			& \le C (1+t)^{\frac{1}{2}-\frac{\alpha}{2}}
		\end{split}
	\end{equation}
	for $0 \le \alpha \le 1$ and 
	\begin{equation} \label{eq:6.6}
		\begin{split}
			\left\| \partial_{t}^{2}  \int_{0}^{t} (K^{(\beta)}_{1}(t-\tau)- K^{(\gamma)}_{1}(t-\tau) )\mathcal{R}_{a} \mathcal{R}_{b} F_{j}(u)(\tau) d \tau \right\|_{1} \le C
		\end{split}
	\end{equation}
	where we used \eqref{eq:5.20} and \eqref{eq:5.19} again for \eqref{eq:6.6}. 
	Therefore by \eqref{eq:6.1}, \eqref{eq:6.2} and \eqref{eq:6.4}, we have the estimate \eqref{eq:3.8}.
	We also have \eqref{eq:3.9} by \eqref{eq:6.1} and \eqref{eq:6.5}. 
	Likewise, it follows from \eqref{eq:6.1} and \eqref{eq:6.6} that the estimate \eqref{eq:3.10}, which is the desired conclusion.
	Once we have \eqref{eq:3.8}-\eqref{eq:3.10}, the proof of the estimates \eqref{eq:3.11}-\eqref{eq:3.13} is shown by the same argument as in \eqref{eq:3.5}-\eqref{eq:3.7}. 
	The only point remaining concerns the estimates for $K_{0}^{(\beta)}(t)\mathcal{R}_{a}  \mathcal{R}_{b} f_{0}$ in $\dot{W}^{1,1}$, since $L^{1}$-$L^{1}$ estimate is not valid.
	
	The estimates \eqref{eq:4.5}, \eqref{eq:4.18}, \eqref{eq:4.19} and \eqref{eq:4.59} yield
	\begin{equation*} 
		\begin{split}
			& 
			\| \nabla (K_{0}^{(\beta)}(t)-K_{0}^{(\gamma)}(t) )\mathcal{R}_{a}  \mathcal{R}_{b} g - m_{\nabla g} \nabla^{-1} \mathcal{R}_{a} \mathcal{R}_{b}  (G_{0}^{(\beta)}(t)-G_{0}^{(\gamma)}(t) )  \|_{1} \\
			& 
			\le
			\| \nabla (K_{0L}^{(\beta)}(t)-G_{0L}^{(\beta)}(t) ) \mathcal{R}_{a}  \mathcal{R}_{b} g \|_{1}
			+ \| \nabla (K_{0L}^{(\gamma)}(t)-G_{0L}^{(\gamma)}(t) ) \mathcal{R}_{a}  \mathcal{R}_{b} g \|_{1} \\
			& + \| 
			(\mathbb{G}_{0}^{(\beta)}(t) -\mathbb{G}_{0}^{(\gamma)}(t))\ast \nabla g - m_{\nabla g} (\mathbb{G}_{0}^{(\beta)}(t) -\mathbb{G}_{0}^{(\gamma)}(t))
			\|_{1} 
			 \\
			& + \sum_{k=M,H}
			(\| \nabla K_{0k}^{(\beta)}(t)\mathcal{R}_{a}  \mathcal{R}_{b} g \|_{1} +\| \nabla K_{0k}^{(\gamma)}(t) \mathcal{R}_{a}  \mathcal{R}_{b} g \|_{1}) \\
			& + C\| \mathcal{R}_{a}  \mathcal{R}_{b} \mathcal{F}^{-1}[ \mathcal{G}^{(\beta)}(t,\xi)(\chi_{M}+\chi_{H}) ] \|_{1} + C\| \mathcal{R}_{a}  \mathcal{R}_{b} \mathcal{F}^{-1}[ \mathcal{G}^{(\gamma)}(t,\xi)  (\chi_{M}+\chi_{H}) ] \|_{1}\\
			& =o(t^{\frac{1}{2}})
		\end{split}
	\end{equation*}
	as $t \to \infty$, which implies 
	\begin{align*}
		& \| \nabla (u_{lin}(t)-G_{lin}(t)) \|_{2} = o(t^{\frac{1}{2}}),
	\end{align*}
	as $t \to \infty$.
	We complete the proof of Theorem \ref{thm:3.3}.
\end{proof}

%
\section{Proof of main results ($p=\infty$)} 
%
\begin{proof}[Proof of Theorem \ref{thm:3.5}]
	At first, we note that once we have the estimate \eqref{eq:3.14}, we can apply the same argument as in \eqref{eq:3.5}-\eqref{eq:3.7} again to get \eqref{eq:3.15}.
	Then it suffices to show the estimate \eqref{eq:3.14} for $F(u)=\nabla u \nabla^2 u$.
	
	We firstly claim that,  under the assumption on Theorem \ref{thm:3.5}, it holds that
	$$
	u(t) \in \{ C([0,\infty); \dot{W}^{3,q} ) \cap C^{1}((0,\infty); \dot{W}^{2,q}) \}^{3},
	$$ 
	\begin{align} 
		\| \nabla^{3} u(t) \|_{q} & 
		\le C (1+t)^{-\frac{5}{2}(1-\frac{1}{q})-\frac{1}{2}} \label{eq:7.1} 
	\end{align}
	for $t \ge 0$, and
	\begin{align} 
		\| \nabla^{2} \partial_{t} u(t) \|_{q} 
		& \le C (1+t)^{-\frac{5}{2}(1-\frac{1}{q})} t^{-\frac{1}{2}} \quad \label{eq:7.2}
	\end{align}
	for $t > 0$, where $2 \le q < \infty$.
	Indeed, noting the facts that $\| g \|_{q} \le \| g \|_{\infty}^{1-\frac{2}{q}} \| g \|_{2}^{\frac{2}{q}}$, we have 
	\begin{equation*} 
		\begin{split}
			\sup_{t \ge 0} \left\{ 
			 (1+t)^{\frac{5}{2}(1-\frac{1}{q}) +\frac{1}{2}} \| \nabla^{3} u_{lin} (t) \|_{q}
			+ 
			 (1+t)^{\frac{5}{2}(1-\frac{1}{q}) } t^{\frac{1}{2}} \| \nabla^{2} \partial_{t} u_{lin} (t) \|_{q}
			\right\} \le C \| f_{0}, f_{1} \|_{Y_{\infty}}  
		\end{split}
	\end{equation*}
	for $2 \le q < \infty$ as in \eqref{eq:5.5}, 
	where 
	$$\| f_{0}, f_{1}\|_{Y_{\infty}}:= \| f_{0} \|_{\dot{H}^{3} \cap \dot{W}^{1,1} \cap \dot{W}^{3,\infty}} + \| f_{1} \|_{H^{1} \cap L^{1} \cap \dot{W}^{1,\infty}}$$ 
	Therefore
	we can apply Theorem \ref{thm:3.1} for each $q \in [2, \infty)$ to obtain \eqref{eq:7.1} and \eqref{eq:7.2}.
	As a consequence, we see that 
	\begin{equation} \label{eq:7.3}
		\begin{split}
			& \| \nabla F(u)(t) \|_{q} \le C(1+t)^{-\frac{5}{2}(1-\frac{1}{q})-\frac{5}{2}}
		\end{split}
	\end{equation}
	for $2 \le q < \infty$ by \eqref{eq:5.17} and \eqref{eq:7.1} and,
\begin{equation*} 
		\begin{split}
			& \| \nabla^{2} u(t)  \|_{\infty} \le C \| \nabla u  \|_{\infty}^{\frac{q}{2q-3}} \| \nabla^{3} u \|_{q}^{\frac{q-3}{2q-3}} \le C(1+t)^{-\frac{10q^{2}-23q+15 }{2q(2q-3)}}
		\end{split}
\end{equation*}
for $3 < q < \infty$ by \eqref{eq:2.39}, \eqref{eq:2.3} and \eqref{eq:7.1}.
Choosing $q$ sufficiently large (formally $q=\infty-\tilde{\varepsilon}$ for sufficiently small $\tilde{\varepsilon}>0$),
we arrive at the estimate 
\begin{equation} \label{eq:7.4}
		\begin{split}
			& \| \nabla^{2} u(t)  \|_{\infty} \le  C(1+t)^{-\frac{5}{4}+ \varepsilon}
		\end{split}
\end{equation}
with small enough $\varepsilon>0$.
Moreover as in \eqref{eq:5.5}, 
we easily have 
	\begin{equation} \label{eq:7.5}
		\begin{split}
			& \| \partial_{t}^{2}  u_{lin}(t) \|_{\infty} \le C(1+t)^{-2} t^{-\frac{1}{2}}\| f_{0}, f_{1}\|_{Y_{\infty}}.
		\end{split}
	\end{equation}

	Now we estimate the nonlinear terms.
	We observe that 
\begin{align*}
	& \partial_{t}^{2} \int_{0}^{t} (K^{(\beta)}_{1}(t-\tau)-K^{(\gamma)}_{1}(t-\tau)) \mathcal{R}_{a} \mathcal{R}_{b} F_{j}(u)(\tau) d \tau \\
	& = \int_{0}^{t} \partial_{t}^{2}  (K^{(\beta)}_{1}(t-\tau)-K^{(\gamma)}_{1}(t-\tau)) \mathcal{R}_{a} \mathcal{R}_{b} F_{j}(u)(\tau) d \tau 
\end{align*}
	by \eqref{eq:5.18}
	and \eqref{eq:5.19} again, and 
	\begin{equation*}
	\| F_{j}(u)(t) \|_{\infty} \le \| \nabla u \|_{\infty} \| \nabla^{2} u \|_{\infty} \le C(1+t)^{-\frac{13}{4}+\varepsilon}
	\end{equation*}
	by \eqref{eq:2.3} and \eqref{eq:7.4}.
	Using these facts, together with \eqref{eq:2.24}, \eqref{eq:4.2}, \eqref{eq:4.59}, \eqref{eq:4.60}, \eqref{eq:5.2} and \eqref{eq:7.3}, 
	we get 
	\begin{equation} \label{eq:7.6}
		\begin{split}
			& \left\| \partial_{t}^{2}  \int_{0}^{t} (K^{(\beta)}_{1}(t-\tau)- K^{(\gamma)}_{1}(t-\tau) )\mathcal{R}_{a} \mathcal{R}_{b} F_{j}(u)(\tau) d \tau \right\|_{\infty} \\
			&
			+ \left\|  \partial_{t}^{2}  \int_{0}^{t} K^{(\gamma)}_{1}(t-\tau)  F_{j}(u)(\tau) d \tau \right\|_{\infty} \\
			& \le \sum_{k=L,M,H}  \int_{0}^{t} 
			\left\| \partial_{t}^{2}  (K^{(\beta)}_{1k}(t-\tau)- K^{(\gamma)}_{1k}(t-\tau) ) \mathcal{R}_{a} \mathcal{R}_{b}  F_{j}(u)(\tau) \right\|_{\infty} d \tau \\
			& + \left\| F_{j}(u)(\tau) \right\|_{\infty} + \sum_{k=L,M,H}  \int_{0}^{t}
			\left\| \partial_{t}^{2}   K^{(\gamma)}_{1k}(t-\tau) F_{j}(u)(\tau) \right\|_{\infty} d \tau \\
			& \le C  \int_{0}^{\frac{t}{2}} (1+t-\tau)^{-\frac{5}{2}}
			\| F_{j}(u)(\tau) \|_{1} d \tau
			+C  \int_{\frac{t}{2}}^{t} 
			\| F_{j}(u)(\tau) \|_{\infty} d \tau \\
			&	
			+C(1+t)^{-\frac{13}{4}+\varepsilon}  
			+
			C \int_{0}^{t} e^{-c(t-\tau)}
			\{ 
			\|  F_{j}(u)(\tau) \|_{\infty} +(t-\tau)^{-\frac{3}{2q}-\frac{1}{2}} \| \nabla F_{j}(u)(\tau) \|_{q} 
			\}
			d \tau   \\
			& \le  C  \int_{0}^{\frac{t}{2}} (1+t-\tau)^{-\frac{5}{2}}
			(1+\tau)^{-2} d \tau + C  \int_{\frac{t}{2}}^{t} 
			(1+\tau)^{-\frac{13}{4}+\varepsilon}  d \tau \\
			& +C(1+t)^{-\frac{13}{4}+\varepsilon} 
			+
			C \int_{0}^{t} e^{-c(t-\tau)}(1+\tau)^{-\frac{13}{4}+\varepsilon}  d \tau \\
			& +
			C \int_{0}^{t} e^{-c(t-\tau)}(t-\tau)^{-\frac{3}{2q}-\frac{1}{2}}  (1+\tau)^{-\frac{5}{2}(1-\frac{1}{q})-\frac{5}{2}} d \tau \\
			& \le C(1+t)^{-\frac{5}{2}},
		\end{split}
	\end{equation}
	where we choose $3<q<\infty$ and used the fact that $-\frac{3}{2q}-\frac{1}{2}>-1$. 
	Combining \eqref{eq:7.5} and \eqref{eq:7.6}, we obtain the desired estimate \eqref{eq:3.14}, 
	which proves Theorem \ref{thm:3.5}.
\end{proof}
\begin{proof}[Proof of Theorem \ref{thm:3.6}]
As mentioned in the proof of Theorem \ref{thm:3.5}, the proof of Theorem \ref{thm:3.6} is completed by showing the estimate \eqref{eq:3.14} for $F(u)=\nabla u \nabla \partial_{t} u$.
We firstly obtain \eqref{eq:7.1} and \eqref{eq:7.2} by Theorem \ref{thm:3.2}. 
The linear estimates \eqref{eq:7.5} are also still valid. Then we shall have established the theorem if we prove the nonlinear estimate.
For this purpose, applying the same argument as in the proof of Theorem \ref{thm:3.5}, 
we have 
	\begin{equation} \label{eq:7.7}
		\begin{split}
			& \| \nabla F(u)(t) \|_{q} \le C(1+t)^{-\frac{5}{2}(1-\frac{1}{q})-2} t^{-\frac{1}{2}}
		\end{split}
	\end{equation}
	for  $F(u)=\nabla u \nabla \partial_{t} u$ and $2 \le q < \infty$ by \eqref{eq:5.37}, \eqref{eq:7.1} and \eqref{eq:7.2}. 
We also have 
\begin{equation} \label{eq:7.8}
	\| F_{j}(u)(t) \|_{\infty} \le \| \nabla u \|_{\infty} \| \nabla \partial_{t} u \|_{\infty} \le C (1+t)^{-\frac{17}{4}} t^{-\frac{1}{4}}
\end{equation}
by \eqref{eq:2.3} and \eqref{eq:2.10}.
Therefore as in \eqref{eq:7.6}, with the choice of $3 < q < \infty$, we can obtain 
\begin{equation} \label{eq:7.9}
		\begin{split}
			& \left\| \partial_{t}^{2}  \int_{0}^{t} (K^{(\beta)}_{1}(t-\tau)- K^{(\gamma)}_{1}(t-\tau) )\mathcal{R}_{a} \mathcal{R}_{b} F_{j}(u)(\tau) d \tau \right\|_{\infty} \\
			&
			+ \left\|  \partial_{t}^{2}  \int_{0}^{t} K^{(\gamma)}_{1}(t-\tau)  F_{j}(u)(\tau) d \tau \right\|_{\infty} \\
			& \le C(1+t)^{-2} t^{-\frac{1}{2}},
		\end{split}
	\end{equation}
where we used \eqref{eq:7.7}, \eqref{eq:7.8} and the facts that
\begin{equation*} 
		\begin{split}
			\int_{0}^{t} e^{-c(t-\tau)}(1+\tau)^{-\frac{17}{4}} \tau^{-\frac{1}{4}} d \tau 
			& \le C  e^{-ct} \int_{0}^{\frac{t}{2}} \tau^{-\frac{1}{4}} d \tau +C  (1+t)^{-\frac{17}{4}} t^{-\frac{1}{4}} \int_{\frac{t}{2}}^{t}  d \tau \\
			& \le C(1+t)^{-\frac{7}{2}} 
		\end{split}
\end{equation*}
and 
\begin{equation*} 
		\begin{split}
			& \int_{0}^{t} e^{-c(t-\tau)}(t-\tau)^{-\frac{3}{2q}-\frac{1}{2}}  (1+\tau)^{-\frac{5}{2}(1-\frac{1}{q})-2}\tau^{-\frac{1}{2}} d \tau \\
			& \le  C  e^{-ct} t^{-\frac{3}{2q}-\frac{1}{2}}  \int_{0}^{\frac{t}{2}} \tau^{-\frac{1}{2}} d \tau 
			+  (1+t)^{-\frac{5}{2}(1-\frac{1}{q})-2} t^{-\frac{1}{2}}  \int_{\frac{t}{2}}^{t} (t-\tau)^{-\frac{3}{2q}-\frac{1}{2}} d \tau \\
			& \le  C  e^{-ct} t^{-\frac{3}{2q}} 
			+  C (1+t)^{-\frac{5}{2}(1-\frac{1}{q})-2} t^{-\frac{3}{2q}} \le  C (1+t)^{-2} t^{-\frac{1}{2}}. 
		\end{split}
\end{equation*}
Thus we have the estimate \eqref{eq:3.14} by \eqref{eq:7.5} and \eqref{eq:7.9}.
We complete the proof of Theorem \ref{thm:3.6}. 
\end{proof}

\vspace*{5mm}
\noindent
\textbf{Acknowledgments. }

Y. Kagei was supported in part by JSPS Grant-in-Aid for Scientific Research (A) 20H00118.
H. Takeda is partially supported by JSPS Grant-in-Aid for Scientific Research (C) 19K03596.



\begin{thebibliography}{99}
%
\bibitem{A}
Agemi, R., 
Global existence of nonlinear elastic waves,
Invent. Math. {\bf 142} (2000), 225-250.
%
\bibitem{BTW} 
Brenner, P., Thom\'ee, V. and Wahlbin, L.,
Besov spaces and applications to difference methods for initial value problems,
Lecture Notes in Mathematics, Vol. 434. Springer-Verlag,
Berlin-New York, 1975.
%
\bibitem{C}
Cazenave, T.,
Semilinear Schr\"{o}dinger equations. Courant Lecture Notes in Mathematics, 10. New York University, Courant Institute of Mathematical Sciences, New York; American Mathematical Society, Providence, RI, 2003.
%
\bibitem{C-I}
Char\~ao C. R. and Ikehata, R.,
Note on asymptotic profile of solutions to the linearized compressible Navier-Stokes flow. Hokkaido Math. J. {\bf 48} (2019), 357-383.
%
%
%
\bibitem{GGH}
%
Ghisi, M., Gobbino, M. and Haraux, A.,
Local and global smoothing effects for some linear hyperbolic equations with a strong dissipation. 
Trans. Amer. Math. Soc. {\bf 368} (2016), 2039-2079.
%
\bibitem{Gr}
Grafakos, L.,
Classical Fourier analysis. Second edition. Graduate Texts in Mathematics, 249. Springer, New York, 2008. 
%
%
\bibitem{H-Z1}
Hoff, D. and Zumbrun, K.,
Multi-dimensional diffusion waves for the Navier-Stokes equations of compressible flow,
Indiana Univ. Math. J. {\bf 44} (1995), 603-676.
%
\bibitem{H-Z2}
Hoff, D. and Zumbrun, K.,
Pointwise decay estimates for multidimensional Navier-Stokes diffusion waves,
Z. Angew. Math. Phys. {\bf 48} (1997), 597-614.
%
\bibitem{I-K-M}
Ikehata, R., Kobayashi, T. and Matsuyama, T.,
Remark on the $L^{2}$ estimates of the density for the compressible Navier-Stokes flow in $\R^{3}$,
Proceedings of the Third World Congress of Nonlinear Analysts, Part 4 (Catania, 2000).
Nonlinear Anal. {\bf 47} (2001), 2519-2526.
%
\bibitem{I1}
Ikehata, R.,
Asymptotic profiles for wave equations with strong damping,
J. Differential Equations {\bf 257} (2014), 2159-2177.
%
\bibitem{J-S}
Jonov, B. and Sideris, T. C., 
Global and almost global existence of small solutions to a dissipative wave equation in 3D with nearly null nonlinear terms,
Commun. Pure Appl. Anal. {\bf 14} (2015), 1407-1442.
%
%
%
\bibitem{K-T}
Kagei, Y. and Takeda, H.,
Smoothing effect and large time behavior of solutions to nonlinear elastic wave equations with viscoelastic term,
preprint
%
\bibitem{K-S}
Kobayashi, T. and Shibata, Y.,
Remark on the rate of decay of solutions to linearized compressible Navier-Stokes equations. Pacific J. Math. {\bf 207} (2002), 199-234.
%
\bibitem{Ponce}
Ponce, G., 
Global existence of small solutions to a class of nonlinear evolution equations, 
Nonlinear Anal. {\bf 9} (1985), 399-418.
%
\bibitem{S-S}
Shatah, J. and Struwe, M.,
Geometric wave equations,
Courant Lecture Notes in Mathematics, 2. New York University, Courant Institute of Mathematical Sciences, New York; American Mathematical Society, Providence, RI, 1998.
%
\bibitem{Shibata}
Shibata, Y., 
On the rate of decay of solutions to linear viscoelastic equation, 
Math. Methods Appl. Sci. {\bf 23} (2000), 203-226.
%
\bibitem{S}
Sideris, T. C.,
The null condition and global existence of nonlinear elastic waves,
Invent. Math. {\bf 123} (1996), 323-342.
%
\bibitem{Takeda}
Takeda, H., 
Large time behavior of solutions to elastic wave with structural damping, submitted.
%

\end{thebibliography}
\end{document}